\documentclass[final]{amsart}
\usepackage{amsmath,amssymb,mathtools}
\usepackage{graphicx}
\usepackage[margin=1.25in]{geometry}
\usepackage{siunitx}
\usepackage{hyperref,enumerate}
\usepackage{cleveref}
\usepackage{esint}
\usepackage{mathrsfs}

\DeclareMathOperator*{\argmin}{arg\,min}

\newcommand{\wt}[1]{\widetilde{#1}}

\def\l{\left}
\def\r{\right}
\def\mR{\mathbb{R}}

\def\mH{\mathcal{H}}
\def\mP{\mathcal{P}}
\def\mE{\mathcal{E}}
\def\mD{\mathcal{D}}
\def\mL{\mathfrak{L}}
\def\Hold{H\"{o}lder }
\def\tceil{\lceil t \rceil}
\def\ve{\varepsilon}
\newcommand{\diff}{\textup{d}}

\newcommand{\bW}{\mathbf{W}}
\newcommand{\bV}{\mathbf{V}}
\newcommand{\bK}{\mathbf{K}}
\newcommand{\bF}{\mathbf{F}}
\newcommand{\bU}{\mathbf{U}}
\newcommand{\mO}{\mathcal{O}}
\newcommand{\mV}{\mathcal{V}}
\newcommand{\fa}{\mathfrak{a}}
\newcommand{\fA}{\mathfrak{A}}
\newcommand{\fS}{\mathfrak{S}}

\usepackage[usenames,dvipsnames]{color}
\newtheorem{Theorem}{Theorem}[section]
\newtheorem{Lemma}[Theorem]{Lemma}
\newtheorem{Proposition}[Theorem]{Proposition}

\newtheorem{Remark}[Theorem]{Remark}
\newtheorem{Definition}[Theorem]{Definition}

\numberwithin{equation}{section}

\begin{document}

\title[Time fractional gradient flow]{Time fractional gradient flows: Theory and numerics}

\author[W.~Li]{Wenbo Li}
\email[W.~Li]{wli50@utk.edu}
\author[A.J.~Salgado]{Abner J.~Salgado}
\email[A.J.~Salgado]{asalgad1@utk.edu}

\address{Department of Mathematics, University of Tennessee, Knoxville TN 37996 USA.}

\date{Draft version of \today}

\makeatletter
\@namedef{subjclassname@2020}{%
	\textup{2020} Mathematics Subject Classification}
\makeatother

\subjclass[2020]{
34G20,                    %%  Nonlinear differential equations in abstract spaces
35R11,                    %%  Fractional partial differential equations
65J08,                    %% Numerical solutions to abstract evolution equations
65M06,                    %%  Numerical solutions to abstract evolution equations
65M15,                    %%  Error bounds for initial value and initial-boundary value problems involving PDEs
65M50                     %%  Mesh generation, refinement, and adaptive methods for the numerical solution of initial value and initial-boundary value problems involving PDEs
}

\keywords{Caputo derivative, gradient flows, a posteriori error estimate, variable time stepping}

%%%%%%%%%%%%%%%%%%%%%%%%%%%%
%%%%%%%%%%%%%%%%%%%%%%%%%%%%

\begin{abstract}
We develop the theory of fractional gradient flows: an evolution aimed at the minimization of a convex, l.s.c.~energy, with memory effects. This memory is characterized by the fact that the negative of the (sub)gradient of the energy equals the so-called Caputo derivative of the state. We introduce the notion of energy solutions, for which we provide existence, uniqueness and certain regularizing effects. We also consider Lipschitz perturbations of this energy. For these problems we provide an a posteriori error estimate and show its reliability. This estimate depends only on the problem data, and imposes no constraints between consecutive time-steps. On the basis of this estimate we provide an a priori error analysis that makes no assumptions on the smoothness of the solution.
\end{abstract}

\maketitle

%%%%%%%%%%%%%%%%%%%%%%%%%%%%
%%%%%%%%%%%%%%%%%%%%%%%%%%%%
\section{Introduction} \label{sec:intro}

In recent times problems involving fractional derivatives have garnered considerable attention, as it is claimed that they better describe certain fundamental relations between the processes of interest; see, for instance \cite{MR2785462,MR2249732,MR3191030}. In this, and many other references the models considered are linear. However, it is well known that real world phenomena are not linear, not even smooth. It is only natural then to consider nonlinear/nonsmooth models with fractional derivatives.

The purpose of this work is to develop the theory and numerical analysis of so-called \emph{time-fractional} gradient flows: an evolution equation aimed at the minimization of a convex and lower semicontinuous (l.s.c.) energy, but where the evolution has memory effects. This memory is characterized by the fact that the negative of the (sub)gradient of the energy equals the so-called Caputo derivative of the state.

The Caputo derivative, introduced in \cite{GJI:GJI529}, is one of the existing models of fractional derivatives. It is defined, for $\alpha \in (0,1)$, by
\begin{equation}
\label{eq:defofCaputo}
  D_c^\alpha w(t) = \frac1{\Gamma(1-\alpha)} \int_0^t \frac{\dot w(r)}{(t-r)^\alpha} \diff r,
\end{equation}
where $\Gamma$ denotes the Gamma function. This definition, from the onset, seems unnatural. To define a derivative of a fractional order, it seems necessary for the function to be at least differentiable. Below we briefly describe several attempts at circumventing this issue. We focus, in particular, on the results developed in a series of papers by Li and Liu, see \cite{LiLiuGeneralized-18,LiLiuGradientFlow-19,LiLiuNoteDeconvo-18,LiLiuCompact-18}, where they developed a distributional theory for this derivative; see also \cite{2007.10244}. The authors of these works also constructed, in \cite{LiLiuNoteDeconvo-18}, so-called deconvolution schemes that aim at discretizing this derivative. With the help of this definition and the schemes that they develop the authors were able to study several classes of equations, in particular time fractional gradient flows.

Let us be precise in what we mean by this term. Let $T>0$ be a final time, $\mH$ be a separable Hilbert space, $\Phi : \mH \to \mR \cup \{+\infty\}$ be a convex and l.s.c.~functional, which we will call \emph{energy}. Given $u_0 \in \mH$, and $f:(0,T] \to \mH$ we seek for a function $u:[0,T] \to \mH$ that satisfies
\begin{equation}
  \begin{dcases}
    D_c^\alpha u(t) + \partial\Phi(u(t)) \ni f(t), & t\in (0,T], \\
    u(0) = u_0,
  \end{dcases}
\label{eq:theGradFlow}
\end{equation}
where by $\partial \Phi$ we denote the subdifferential of $\Phi$. Our objectives in this work can be stated as follows: We will introduce the notion of ``energy solutions'' of \eqref{eq:theGradFlow}, and we will refine the results regarding existence, uniqueness, and regularizing effects provided in \cite{LiLiuGradientFlow-19}. This will be done by generalizing, to non-uniform time steps the ``deconvolution'' schemes of \cite{LiLiuNoteDeconvo-18,LiLiuGradientFlow-19}, and developing a sort of ``fractional minimizing movements'' scheme. We will also provide an a priori error estimate that seems optimal in light of the regularizing effects proved above. We also develop an a posteriori error estimate, in the spirit of \cite{MR1737503} and show its reliability.

We comment, in passing, that nonlinear evolution problems with fractional time derivative have been considered in other works. From a modeling point of view, their advantages have been observed in \cite{MR2249732,MR3164777}. Some other types of nonlinear problems have been studied in \cite{MR3335453,MR1744336,MR2824708,MR3982325,MR4087352,MR2390082,MR3965393} and \cite{2007.14855,MR4036095} where, for a particular type of nonlinear problem other ``energy dissipation inequalities'' than those we obtain are derived. Regularity properties for nonlinear problems with fractional time derivatives have been obtained in \cite{MR3563229,MR4080675,MR4037586,MR3912710,MR3184573,MR3038123,MR3000457,MR2355009}. Of particular interest to us are \cite{LiLiuGradientFlow-19} which we described above and \cite{MR4040846} which also considers time fractional gradient flows. The assumptions on the data, however, are slightly different than ours. As such, some of the results in \cite{MR4040846} are stronger, and some weaker than ours; in particular, we conduct a numerical analysis of this problem. Nevertheless, we refer to this reference for a nice historical account and particular applications to PDEs.

Our presentation will be organized as follows. We will establish notation and the framework we will adopt in \Cref{sec:Notation}. Here, in particular, we will study several properties of a particular space, which we denote by $L^p_\alpha(0,T;\mH)$, and that will be used to characterize the requirements on the right hand side $f$ of \eqref{eq:theGradFlow}. In addition, we also review the various proposed generalizations of the classical definition of Caputo derivatives, with particular attention to that of \cite{LiLiuGeneralized-18,LiLiuGradientFlow-19,LiLiuCompact-18}; since this is the one we shall adopt. In \Cref{sec:deconvo} we generalize the deconvolution schemes of \cite{LiLiuNoteDeconvo-18,LiLiuGradientFlow-19} and their properties, to the case of nonuniform time stepping. Many of the simple properties of these schemes are lost in this case, but we retain enough of them for our purposes. \Cref{sec:theory} introduces the notion of energy solutions for \eqref{eq:theGradFlow} and shows existence and uniqueness of these. This is accomplished by introducing, on the basis of our generalized deconvolution formulas, a fractional minimizing movements scheme; and showing that the discrete solutions have enough compactness to pass to the limit in the size of the partition. In \Cref{sec:numerics} we provide an error analysis of the fractional minimizing movements scheme. First, we show how an error estimate follows as a side result from the existence proof. Then, in the spirit of \cite{MR1737503}, we provide an a posteriori error estimator for our scheme and show its reliability. This estimator is then used to independently show rates of convergence. This section is concluded with some particular instances in which the rate of convergence can be improved. \Cref{sec:LipschitzPerturbation} is dedicated to the case in which we allow a Lipschitz perturbation of the subdifferential. We extend the existence, uniqueness, a priori, and a posteriori approximation results of the fractional gradient flow. Finally, \Cref{sec:experiments} presents some simple numerical experiments that illustrate, explore, and expand our theory.

\section{Notation and preliminaries}
\label{sec:Notation}

Let us begin by presenting the main notation and assumptions we shall operate under. We will denote by $T \in (0,\infty)$ our final (positive) time. By $\mH$ we will always denote a separable Hilbert space with scalar product $\langle \cdot, \cdot \rangle$ and norm $\| \cdot \|$. As it is by now customary, by $C$ we will denote a nonessential constant whose value may change at each occurrence.

\subsection{Convex energies}
The \emph{energy} will be a convex, l.s.c., functional $\Phi : \mH \to \mR \cup \{+\infty\}$ with nonempty effective domain of definition, that is
\[
  D(\Phi) = \{ w \in \mH: \Phi(w)<+\infty \} \neq \emptyset.
\]
We will always assume that our energy is bounded from below, that is
\[
  \Phi_{\inf} = \inf_{u \in \mH} \Phi(u) > -\infty.  
\]
As we are not assuming smoothness in our energy beyond convexity, a useful substitute for its derivative is the subdifferential, that is, 
\[
  \partial \Phi(w) = \left\{ \xi \in \mH: \langle \xi, v-w \rangle \leq \Phi(v) -\Phi(w) \quad \forall v \in \mH\right\}.
\]
The effective domain of the subdifferential is
$
  D(\partial \Phi) = \l\{ w \in \mH: \partial\Phi(w) \neq \emptyset \r\}.
$
Recall that, in our setting, we always have that $\overline{D(\partial\Phi)} = \overline{D(\Phi)}$. We refer the reader to \cite{MR1058436,MR2330778} for basic facts on convex analysis.

In applications, it is sometimes useful to obtain error estimates on (semi)norms stronger than those of the ambient space, and that are dictated by the structure of the energy. For this reason, we introduce the following \emph{coercivity modulus} of $\Phi$, see \cite[Definition 2.3]{MR1737503}.

\begin{Definition}[coercivity modulus]
\label{def:sigma-rho}
For every $w_1 \in D(\Phi)$ and $w_2 \in D(\partial \Phi)$, let $\sigma(w_1; w_2) \ge 0$ be
\[
  \sigma(w_1; w_2) = \Phi(w_2) - \Phi(w_1) - \sup_{\xi \in \partial \Phi(w_1)} \langle \xi, w_2 - w_1 \rangle.
\]
Then for every $w_1, w_2 \in D(\partial \Phi)$ we define
\[
  \rho(w_1, w_2) = \sigma(w_1; w_2) + \sigma(w_2; w_1) = \inf_{\xi_1 \in \partial \Phi(w_1), \xi_2 \in \partial \Phi(w_2)} \langle \xi_1 - \xi_2, w_1 - w_2 \rangle.
\]
\end{Definition}

We comment that, by the definition, $\rho(\cdot, \cdot)$ is symmetric, whereas $\sigma(\cdot; \cdot)$ might not be. Furthermore, the separability of $\mH$ guarantees that $\sigma$ and $\rho$ are both Borel measurable \cite[Remark 2.4]{MR1737503}. One may also refer to \cite[Section 2.3]{MR1737503} for discussions and properties of $\sigma$ and $\rho$ for certain choices of $\Phi$. Definition~\ref{def:sigma-rho} enables us to write
\begin{equation}\label{eq:def-subdiff-strong}
    \xi \in \partial \Phi(w) \iff 
\langle \xi, v - w \rangle + \sigma(w; v) \leq \Phi(v) - \Phi(w), \quad \forall v \in \mH.
\end{equation}

\subsection{Vector valued time dependent functions}
We will follow standard notation regarding Bochner spaces of vector valued functions, see \cite[Section 1.5]{MR3014456}. For any $w \in L^1(0,T;\mH)$ and $E \subset [0,T]$ that is measurable, we define the average by
\[
  \fint_E w(t) \diff t = \frac1{|E|} \int_E w(t) \diff t,
\]
where $|E|$ denotes the Lebesgue measure of $E$.

Since eventually we will have to deal with time discretization, we also introduce notation for time-discrete vector valued functions. Let $\mP$ be a partition of the time interval $[0,T]$
\begin{equation}\label{eq:def-partition}
  \mP = \{0 = t_0 < t_1 < \ldots < t_{N-1} < t_N = T\},
\end{equation}
with variable steps $\tau_n = t_n - t_{n-1}$ and $\tau = \max \{ \tau_n: n \in \{1,\ldots, N\} \}$. We will always denote by $N$ the size of a partition. For $t \in [0,T]$ we define
\[
  \lfloor t \rfloor_\mP = \max\{ r \in \mP : r < t \}, \quad \lceil t \rceil_\mP = \min\{ r \in \mP: t \leq r \},
\]
and $n(t)$ to be the index of $\tceil_\mP$, so that $t \in (\lfloor t \rfloor_\mP, \tceil_\mP] = (t_{n(t)-1},t_{n(t)}]$. Given a partition $\mP$, for $\bW = \{W_i\}_{i=1}^N \subset \mH^N$ we define its \emph{piecewise constant} interpolant with respect to $\mP$ to be the function $\overline{W}_\mP \in L^\infty(0,T;\mH)$ defined by
\begin{equation}
\label{eq:def-pw-const}
  \overline{W}_\mP(t) = W_{n(t)}.
\end{equation}

\subsubsection{The space $L^p_\alpha(0,T;\mH)$}
To quantify the assumptions we need on the right hand side $f$ of \eqref{eq:theGradFlow} we introduce the following space. 

\begin{Definition}[space $L^p_{\alpha}(0, T; \mH)$]
\label{def:spaceLpalpha}
Let $p \in [1,\infty)$ and $\alpha \in (0,1)$. We say that the function $w :[0,T] \to \mH$ belongs to the space $L^p_{\alpha}(0, T; \mH)$ iff 
\begin{equation}\label{eq:def-Lpalpha}
  \Vert w \Vert_{L^p_{\alpha}(0, T; \mH)} = \sup_{t \in [0,T] } \l( \int_0^t (t-s)^{\alpha-1} \Vert w(s) \Vert^p \diff s \r)^{1/p} < \infty.
\end{equation}
\end{Definition}

Let us show some basic embedding results about this space. 
%\wl{(Should we mention the paper by Jianguo Liu here? For example ``see also \cite[Section 4.1]{LiLiuCompact-18}'')}

\begin{Proposition}[embedding]\label{proposition:embeddingLpalpha}
Let $p \in [1,\infty)$, $\alpha \in (0,1)$, and $q > p/\alpha$. Then we have that 
\[
  L^q(0, T; \mH) \hookrightarrow L^p_{\alpha}(0, T; \mH) \hookrightarrow L^p(0,T;\mH).
\]
\end{Proposition}
\begin{proof}
The second embedding is immediate. For any $t \in (0,T]$
\[
  \int_0^t \|  w(s) \|^p \diff s \leq \sup_{s\in [0,t]} (t-s)^{1-\alpha} \int_0^t (t-s)^{\alpha-1} \| w(s) \|^p \diff s \leq T^{1-\alpha} \| w \|_{L^p_\alpha(0,T;\mH)}^p,
\]
where we used that $1-\alpha > 0$.

The proof of the first embedding is a simple application of \Hold inequality. Indeed, we have
\[
  \l( \int_0^t (t-s)^{\alpha-1} \Vert w(s) \Vert^p \diff s\r)^{1/p} \le
  \l( \frac{q-p}{q\alpha - p} \r)^{(q-p)/q} t^{\alpha - p/q} \Vert w \Vert_{L^q(0, t; \mH)},
\]
and hence
\begin{equation}\label{eq:Lq-embed-Lpalpha}
  \Vert w \Vert_{L^p_{\alpha}(0, T; \mH)} \le
  \l( \frac{q-p}{q\alpha - p} \r)^{(q-p)/q} T^{\alpha - p/q} \Vert w \Vert_{L^q(0, T; \mH)},
\end{equation}
as we intended to show.
\end{proof}

When dealing with discretization we will approximate the right hand side $f$ of \eqref{eq:theGradFlow} by its local averages over a partition $\mP$. Thus, we must provide a bound on this operation that is independent of the partition.

%\wl{The only reason I want to change all the $p \in (1, \infty)$ to $p \in [1,\infty)$ is the inequality above \Cref{thm:post-rate-fL2alpha-LipPerturb} when we using \Cref{lemma:alpha-int-w}. So in fact We only need to change \Cref{lemma:alpha-int-w}.}

\begin{Lemma}[continuity of averaging]\label{lemma:pw-interp-L2alpha}
Let $p \in [1,\infty)$, $\alpha \in (0,1)$, $f \in L^p_\alpha(0,T;\mH)$, and $\mP$ be a partition of $[0,T]$ as in \eqref{eq:def-partition}. Define $\bF = \{ \fint_{t_{n-1}}^{t_n} f(t) \diff t\}_{n=1}^N \subset \mH^N$ and let $\overline{F}_\mP$ be defined as in \eqref{eq:def-pw-const}. Then, there exists a constant $C > 0$ only depending on $p$ and $\alpha$ such that
\[
  \Vert \overline{F}_\mP \Vert_{L^p_{\alpha}(0, T; \mH)} \leq C \Vert f \Vert_{L^p_{\alpha}(0, T; \mH)}.
\]
\end{Lemma}
\begin{proof}
Let $p \in (1,\infty)$. We first, for $n \in \{1, \ldots, N\}$, bound the integral
\[
\int_0^{t_n} (t_n-s)^{\alpha - 1} \Vert \overline{F}_\mP(s) \Vert^p \diff s.
\]
To achieve this, we decompose this integral as
\begin{equation}\label{eq:L2alpha-proof1}
  \begin{aligned}
  \int_0^{t_n} (t_n-s)^{\alpha - 1} \Vert \overline{F}_\mP(s) \Vert^p \diff s  &= \sum_{k=1}^{n} \int_{t_{k-1}}^{t_k} (t_n-s)^{\alpha - 1} \Vert \overline{F}_\mP(s) \Vert^p \diff s \\
  &= \sum_{k=1}^{n} \Vert F_k \Vert^p \int_{t_{k-1}}^{t_k} (t_n-s)^{\alpha - 1} \diff s.
  \end{aligned}
\end{equation}
We use H\"older inequality in the definition of $F_k$ to obtain that
\begin{equation}\label{eq:L2alpha-proof2}
  \Vert F_k \Vert^p = \l \Vert \fint_{t_{k-1}}^{t_k} f(s) \diff s \r \Vert^p
  \leq  \fint_{t_{k-1}}^{t_k} (t_n-s)^{\alpha-1} \Vert f(s) \Vert^p \diff s  \l(\fint_{t_{k-1}}^{t_k} (t_n-s)^{\frac{1-\alpha}{p-1}} \diff s \r)^{p-1}.
\end{equation}
Since, for every $p \in (1,\infty)$ the function $s \mapsto s^{\alpha-1}$ belongs to the Muckenhoupt class $A_p(\mR_+)$, see \cite[Example 7.1.7]{GrafakosClassical-14}, there exists a constant $C_{p,\alpha}$ that only depends on $p$ and $\alpha$ such that
\[
  \fint_a^b s^{\alpha-1} \diff s \l( \fint_a^b s^{\frac{1-\alpha}{p-1}} \diff s \r)^{p-1}  \le C_{p,\alpha}, \quad \forall 0 \le a < b.
\]
Therefore, for any $k$, we have
\begin{equation}\label{eq:L2alpha-proof3}
\begin{aligned}
  \fint_{t_{k-1}}^{t_k} (t_n-s)^{\alpha-1} \diff s \l[  \fint_{t_{k-1}}^{t_k} (t_n-s)^{\frac{1-\alpha}{p-1}} \diff s \r]^{p-1}
= \fint_{t_n-t_{k-1}}^{t_n-t_k}  s^{\alpha-1} \diff s \l[ \fint_{t_n-t_{k-1}}^{t_n-t_k}  s^{\frac{1-\alpha}{p-1}} \diff s \r]^{p-1} \le C_{p,\alpha}.
\end{aligned}
\end{equation}
Substituting \eqref{eq:L2alpha-proof2} and \eqref{eq:L2alpha-proof3} into  \eqref{eq:L2alpha-proof1} we get
\[
\begin{aligned}
\int_0^{t_n} (t_n-s)^{\alpha-1} \Vert \overline{F}_\mP(s) \Vert^p \diff s 
&\le \sum_{k=1}^{n}
C_{p,\alpha} \int_{t_{k-1}}^{t_k} (t_n-s)^{\alpha-1} \Vert f(s) \Vert^p \diff s \\
&= C_{p,\alpha} \int_0^{t_n} (t_n-s)^{\alpha-1} \Vert f(s) \Vert^p \diff s \le C_{p,\alpha} \Vert f \Vert^p_{L^p_{\alpha}(0, T; \mH)}.
\end{aligned}
\]

Now consider $t \in [0,T]$. Taking advantage of the estimate we obtained above we write
\begin{equation}\label{eq:L2alpha-proof4}
\begin{aligned}
  \int_0^t (t-s)^{\alpha-1} \Vert \overline{F}_\mP(s) \Vert^p \diff s & = \int_0^{\lfloor t\rfloor_{\mP}} (t-s)^{\alpha-1} \Vert \overline{F}_\mP(s) \Vert^p \diff s
    + \int_{\lfloor t\rfloor_{\mP}}^t (t-s)^{\alpha-1} \Vert \overline{F}_\mP(s) \Vert^p \diff s \\
    & = \int_0^{\lfloor t\rfloor_{\mP}} (t-s)^{\alpha-1} \Vert \overline{F}_\mP(s) \Vert^p \diff s
    + \Vert \overline{F}_\mP(\lceil t \rceil_\mP) \Vert^p \int_{\lfloor t\rfloor_{\mP}}^t (t-s)^{\alpha-1} \diff s \\
    & \le \int_0^{\lfloor t\rfloor_{\mP}} (\lfloor t\rfloor_{\mP}-s)^{\alpha-1} \Vert \overline{F}_\mP(s) \Vert^p \diff s
    + \Vert \overline{F}_\mP(\lceil t \rceil_\mP) \Vert^p \int_{\lfloor t\rfloor_{\mP}}^{\lceil t \rceil_\mP} (\lceil t \rceil_\mP-s)^{\alpha-1} \diff s \\
    & \le C_{p,\alpha} \Vert f \Vert^p_{L^p_{\alpha}(0, T; \mH)} + 
    \int_0^{\lceil t \rceil_\mP} (\lceil t \rceil_\mP -s)^{\alpha-1} \Vert \overline{F}(s) \Vert^p \diff s
    \le 2C_{p,\alpha} \Vert f \Vert^p_{L^p_{\alpha}(0, T; \mH)}.
\end{aligned}
\end{equation}
Therefore by taking supremum over $t \in [0, T]$ and $C = (2C_{p,\alpha})^{1/p}$, we finish the proof of this lemma.

For $p = 1$, the proof proceeds almost the same way as before. The only difference worth noting is that, instead of \eqref{eq:L2alpha-proof2}, we have
\[
  \Vert F_k \Vert = \l \Vert \fint_{t_{k-1}}^{t_k} f(s) \diff s \r \Vert
  \leq  \fint_{t_{k-1}}^{t_k} (t_n-s)^{\alpha-1} \Vert f(s) \Vert \diff s  \sup_{s \in [t_{k-1},t_k]} \frac1{(t_n-s)^{\alpha - 1}} .  
\]
Next, we observe that, since $\alpha -1 \in (-1,0)$, then the function $s \mapsto s^{\alpha - 1}$ belongs to the Muckenhoupt class $A_1(\mR_+)$. Thus,
\[
  \sup_{s \in [a,b]} \frac1{s^{\alpha-1}} \fint_a^b s^{\alpha-1} \diff s \leq C_\alpha, \quad \forall 0\leq a<b.
\]
With this information, the proof proceeds without change.
\end{proof}

It turns out that averaging is not only continuous, but possesses suitable approximation properties in this space. Namely, we have a control on the difference between fractional integrals of $f \in L^p_{\alpha}(0, T; \mH)$ and its averages.

\begin{Lemma}[approximation]\label{lemma:alpha-int-f-Fbar}
Let $p \in [1,\infty)$, $\alpha \in (0,1)$, $f \in L^p_\alpha(0,T;\mH)$, and $\mP$ be a partition of $[0,T]$ as in \eqref{eq:def-partition}. Let $p'$ be the \Hold conjugate of $p$, $\bF = \{ \fint_{t_{n-1}}^{t_n} f(t) \diff t\}_{n=1}^N \subset \mH^N$, and let $\overline{F}_\mP$ be defined as in \eqref{eq:def-pw-const}. Then we have
\begin{equation}\label{eq:alpha-int-f-Fbar}
  \sup_{t \in [0,T]} \l\Vert \int_0^t (t-s)^{\alpha-1} \l( f(s) - \overline{F}_\mP(s) \r) \diff s \r\Vert \le C \tau^{\alpha/p'} \Vert f - \overline{F}_\mP \Vert_{L^p_{\alpha}(0, T; \mH)} \le C' \tau^{\alpha/p'} \Vert f \Vert_{L^p_{\alpha}(0, T; \mH)},
\end{equation}
where the constants $C, C'$ depend only on $p$ and $\alpha$. In addition, for any $\beta \in (0,1)$ we also have
\begin{equation}\label{eq:double-alpha-int-f-Fbar}
\begin{aligned}
  \sup_{r \in [0,T]} \int_0^r (r-t)^{\alpha-1}  \l\Vert \int_0^t (t-s)^{\beta-1} \l( f(s) - \overline{F}_\mP(s) \r) \diff s \r\Vert^p \diff t  \\
  \le C_1 \tau^{p\beta} \Vert f - \overline{F}_\mP \Vert^p_{L^p_{\alpha}(0, T; \mH)} \le C'_1 \tau^{p\beta} \Vert f \Vert^p_{L^p_{\alpha}(0, T; \mH)},
\end{aligned}
\end{equation}
where the constants $C_1, C'_1$ depend on $p$, $\alpha$, and $\beta$. As usual, when $p = 1$, we have $p' = \infty$ and $1/p'$ is treated as $0$.
\end{Lemma}
\begin{proof}
We first notice that the second inequalities in both \eqref{eq:alpha-int-f-Fbar} and \eqref{eq:double-alpha-int-f-Fbar} follow directly from \Cref{lemma:pw-interp-L2alpha} and the triangle inequality. 

To show the first inequality in \eqref{eq:alpha-int-f-Fbar}, given $\mP$ we consider $t \in [0,T]$. Using that $f-\overline{F}_\mP$ has zero mean on each subinterval of the partition, we can write
\begin{equation}\label{eq:alpha-int-f-Fbar-proof1}
\begin{aligned}
&\int_0^t (t-s)^{\alpha-1} \l( f(s) - \overline{F}_\mP(s) \r) \diff s \\
&= \int_{\lfloor t \rfloor_\mP}^t (t-s)^{\alpha-1} \l( f(s) - \overline{F}_\mP(s) \r) \diff s + \sum_{k=1}^{n(t)-1} \int_{t_{k-1}}^{t_k} (t-s)^{\alpha-1} \l( f(s) - \overline{F}_\mP(s) \r) \diff s \\
&= \int_{\lfloor t\rfloor_{\mP}}^t (t-s)^{\alpha-1} \l( f(s) - \overline{F}_\mP(s) \r) \diff s \\ &+ \sum_{k=1}^{n(t)-1} \int_{t_{k-1}}^{t_k} \l( (t-s)^{\alpha-1} - (t-t_{k-1})^{\alpha-1}\r) \l( f(s) - \overline{F}_\mP(s) \r) \diff s = \mathrm{I}_1(t) + \mathrm{I}_2(t).
\end{aligned}
\end{equation}
For the first term, denoted $I_1(t)$, we have
\[
\begin{aligned}
\l\Vert \mathrm{I}_1(t) \r\Vert 
&\le \l( \int_{\lfloor t\rfloor_{\mP}}^t (t-s)^{\alpha-1} \l\Vert f(s) - \overline{F}_\mP(s) \r\Vert^p \diff s\r)^{1/p} \l( \int_{\lfloor t\rfloor_{\mP}}^t (t-s)^{\alpha-1} \diff s \r)^{1/p'} \\
& \le \Vert f - \overline{F}_\mP \Vert_{L^p_{\alpha}(0, T; \mH)} \l( \frac{1}{\alpha} (t-\lfloor t\rfloor_{\mP})^{\alpha} \r)^{1/p'}
\le C_1 \tau^{\alpha/p'} \Vert f - \overline{F}_\mP \Vert_{L^p_{\alpha}(0, T; \mH)},
\end{aligned}
\]
where $C_1$ only depends on $p$ and $\alpha$. For the second term, noticing that $t - t_{k-1} + \tau > t - s$ for $s \in (t_{k-1}, t_k)$ we have
\[
\begin{aligned}
\l\Vert \mathrm{I}_2 \r\Vert 
&\le \int_0^{\lfloor t\rfloor_{\mP}} \l( (t-s)^{\alpha-1} - (t-s+\tau)^{\alpha-1} \r) \Vert f(s) - \overline{F}_\mP(s) \Vert \diff s \\
& \le \l[ \int_{0}^{\lfloor t\rfloor_{\mP}} (t-s)^{\alpha-1} \l\Vert f(s) - \overline{F}_\mP(s) \r\Vert^p \diff s\r]^{1/p} \l[ \int_0^{\lfloor t\rfloor_{\mP}} (t-s)^{\alpha-1} \l[ 1 - \l[ \frac{t-s+\tau}{t-s}\r]^{\alpha-1} \r]^{p'} \diff s\r]^{1/p'} \\
&\le \Vert f - \overline{F}_\mP \Vert_{L^p_{\alpha}(0, T; \mH)} \l( \int_0^{\lfloor t\rfloor_{\mP}} (t-s)^{\alpha-1} - (t-s+\tau)^{\alpha-1} \diff s\r)^{1/p'}.
\end{aligned}
\]
Since
\[
\begin{aligned}
 \int_0^{\lfloor t\rfloor_{\mP}} (t-s)^{\alpha-1} - (t-s+\tau)^{\alpha-1} \diff s
&= \frac{1}{\alpha} \l( t^{\alpha} - (t-\lfloor t\rfloor_{\mP})^{\alpha} - (t+\tau)^{\alpha} + (t-\lfloor t\rfloor_{\mP} + \tau)^{\alpha} \r) \\
& \le \frac{1}{\alpha} \l( (t-\lfloor t\rfloor_{\mP} + \tau)^{\alpha} - (t-\lfloor t\rfloor_{\mP})^{\alpha} \r)
\le \frac{\tau^{\alpha}}{\alpha},
\end{aligned}
\]
we obtain
\[
\l\Vert \mathrm{I}_2(t) \r\Vert \le C_2 \tau^{\alpha/p'} \Vert f - \overline{F}_\mP \Vert_{L^p_{\alpha}(0, T; \mH)},
\]
and \eqref{eq:alpha-int-f-Fbar} follows after combining the bounds for $\mathrm{I}_1(t)$ and $\mathrm{I}_2(t)$ that we have obtained.

To prove \eqref{eq:double-alpha-int-f-Fbar} we apply the \Hold inequality to \eqref{eq:alpha-int-f-Fbar-proof1} with $\alpha$ replaced by $\beta$ to get
\[
  \l\Vert \int_0^t (t-s)^{\beta-1} \l( f(s) - \overline{F}_\mP(s) \r) \diff s \r\Vert^p \leq  \mathrm{II}_1(t)^{p-1} \cdot \left( \mathrm{II}_2(t) + \mathrm{II}_3(t) \right),
\]
where
\begin{align*}
  \mathrm{II}_1(t) &= \int_{\lfloor t \rfloor_\mP}^t (t-s)^{\beta-1} \diff s + \sum_{k=1}^{n(t)-1} \int_{t_{k-1}}^{t_k} \l[ (t-s)^{\beta-1} - (t-t_{k-1})^{\beta-1} \r] \diff s, \\
  \mathrm{II}_2(t) &= \int_{\lfloor t \rfloor_\mP}^t \!\!\! (t-s)^{\beta-1} \l\Vert f(s) - \overline{F}_\mP(s) \r\Vert^p \diff s, \\
  \mathrm{II}_3(t) &= \sum_{k=1}^{n(t)-1} \int_{t_{k-1}}^{t_k} \!\!\! \l( (t-s)^{\beta-1} - (t-t_{k-1})^{\beta-1}\r) \l\Vert f(s) - \overline{F}_\mP(s) \r\Vert^p \diff s.
\end{align*}
Arguing as in the bound for $\mathrm{I}_2(t)$
\begin{align*}
  \mathrm{II}_1(t) &= \frac{1}{\beta}(t-\lfloor t \rfloor_\mP)^{\beta} + \int_0^{\lfloor t \rfloor_\mP} \l[ (t-s)^{\beta-1} - (t-s+\tau)^{\beta-1} \r]\diff s \leq \frac2\beta \tau^\beta.
\end{align*}
Thus, to obtain \eqref{eq:double-alpha-int-f-Fbar} it suffices to show that, for every $r \in [0,T]$,
\[
\int_0^r (r-t)^{\alpha-1} \l( \mathrm{II}_2(t) + \mathrm{II}_3(t) \r) \diff t \le C_2 \tau^{\beta} \Vert f - \overline{F}_\mP \Vert^p_{L^p_{\alpha}(0, T; \mH)}
\]
with some constant $C_2$ only depending on $p$, $\alpha$, and $\beta$. To estimate the fractional integral of $\mathrm{II}_2$ by Fubini's theorem we have
\begin{equation}\label{eq:double-alpha-int-f-Fbar-proof1}
\int_0^r (r-t)^{\alpha-1} \mathrm{II}_2(t) \diff t = 
\int_0^r \l\Vert f(s) - \overline{F}_\mP(s) \r\Vert^p 
\int_s^{\lceil s \rceil_\mP \wedge r }  (r-t)^{\alpha-1} (t-s)^{\beta-1} \diff t \diff s,
\end{equation}
where we set $a\wedge b = \min \{a,b\}$.
We claim that there exists a constant $C_3$ depending on $\alpha$ and $\beta$ such that
\begin{equation}\label{eq:double-alpha-int-f-Fbar-proof2}
\int_s^{ \lceil s \rceil_\mP \wedge r } (r-t)^{\alpha-1} (t-s)^{\beta-1} \diff t \le C_3 (r-s)^{\alpha-1} \tau^{\beta}.
\end{equation}
On the one hand, for $r-s \le 2\tau$, we simply have
\[
\begin{aligned}
\int_s^{ \lceil s \rceil_\mP \wedge r } (r-t)^{\alpha-1} (t-s)^{\beta-1} \diff t &\le \int_s^r (r-t)^{\alpha-1} (t-s)^{\beta-1} \diff t \\
&= \frac{\Gamma(\alpha)\Gamma(\beta)}{\Gamma(\alpha+\beta)} (r-s)^{\alpha+\beta-1} \le \frac{\Gamma(\alpha)\Gamma(\beta)}{\Gamma(\alpha+\beta)} 
(r-s)^{\alpha-1} (2\tau)^{\beta}.
\end{aligned}
\]
On the other hand, if $r-s > 2\tau$, then
\[
\begin{aligned}
&\int_s^{ \lceil s \rceil_\mP \wedge r } (r-t)^{\alpha-1} (t-s)^{\beta-1} \diff t 
\le \int_s^{s+\tau} (r-t)^{\alpha-1} (t-s)^{\beta-1} \diff t \\
& \le \int_s^{s+\tau} \l( \frac{r-s}{2} \r)^{\alpha-1} (t-s)^{\beta-1} \diff t
= \frac{2^{1-\alpha}}{\beta} (r-s)^{\alpha-1} \tau^{\beta}.
\end{aligned}
\]
Therefore \eqref{eq:double-alpha-int-f-Fbar-proof2} is proved, and thus \eqref{eq:double-alpha-int-f-Fbar-proof2} implies that
\[
\int_0^r (r-t)^{\alpha-1} \mathrm{II}_2(t) \diff t \le C_3 \tau^{\beta} \int_0^r (r-s)^{\alpha-1} \l\Vert f(s) - \overline{F}_\mP(s) \r\Vert^p \diff s
\le C_3 \tau^{\beta} \Vert f - \overline{F}_\mP \Vert^p_{L^p_{\alpha}(0, T; \mH)}.
\]
For $\mathrm{II}_3(t)$, we again apply Fubini's theorem to obtain
\[
\int_0^r (r-t)^{\alpha-1} \mathrm{II}_3(t) \diff t = 
\int_0^r \l\Vert f(s) - \overline{F}_\mP(s) \r\Vert^p 
\int_s^r (r-t)^{\alpha-1} \l( (t-s)^{\beta-1} - (t-s+\tau)^{\beta-1}\r) \diff t \diff s .
\]
To conclude, we claim that
\begin{equation}
\label{eq:newstuffBetafunc}
A = \int_s^r (r-t)^{\alpha-1} \l( (t-s)^{\beta-1} - (t-s+\tau)^{\beta-1}\r) \diff t
\le C_4 \tau^{\beta} (r-s)^{\alpha-1},
\end{equation}
for a constant $C_4$ depending on $\alpha$ and $\beta$. Indeed, if this is the case, we have
\[
\int_0^r (r-t)^{\alpha-1} \mathrm{II}_3(t) \diff t \le C_4 \tau^{\beta} \int_0^r (r-s)^{\alpha-1} \l\Vert f(s) - \overline{F}_\mP(s) \r\Vert^p \diff s
\le C_4 \tau^{\beta} \Vert f - \overline{F}_\mP \Vert^p_{L^p_{\alpha}(0, T; \mH)},
\]
and we combine the estimates for $\mathrm{II}_2(t)$ and $\mathrm{II}_3(t)$ together and conclude the proof of \eqref{eq:double-alpha-int-f-Fbar}.

Let us now turn to the proof of \eqref{eq:newstuffBetafunc}. First, if $r-s \leq \tau$ then it suffices to observe that
\[
 A \leq  \int_s^r (r-t)^{\alpha -1}(t-s)^{\beta-1} \diff t = \frac{\Gamma(\alpha)\Gamma(\beta)}{\Gamma(\alpha+\beta)} (r-s)^{\alpha+\beta-1} \leq \frac{\Gamma(\alpha)\Gamma(\beta)}{\Gamma(\alpha+\beta)} \tau^\beta (r-s)^{\alpha-1}.
\]
Now, if $r-s>\tau$, we estimate as
\begin{align*}
  A &= \int_s^r(r-t)^{\alpha-1}(t-s)^{\beta-1}\diff t - \int_s^r (r-t)^{\alpha-1}(t-s+\tau)^{\beta-1} \diff t \\
  &= \frac{\Gamma(\alpha)\Gamma(\beta)}{\Gamma(\alpha+\beta)} (r-s)^{\alpha+\beta-1} -\int_{-\tau}^{r-s} (t+\tau)^{\beta-1}(r-t-s)^{\alpha-1} \diff t + \int_0^\tau (r-s-t+\tau)^{\alpha-1}t^{\beta-1} \diff t \\
  &= \frac{\Gamma(\alpha)\Gamma(\beta)}{\Gamma(\alpha+\beta)} \left( (r-s)^{\alpha+\beta-1} - (r-s + \tau)^{\alpha+\beta-1} \right) + \int_0^\tau (r-s-t+\tau)^{\alpha-1}t^{\beta-1} \diff t.
\end{align*}
The first term can be bounded using that $r-s>\tau$ as follows
\[
  (r-s)^{\alpha+\beta-1} - (r-s + \tau)^{\alpha+\beta-1} \leq \max\{\alpha+\beta-1,0\} \tau (r-s)^{\alpha+\beta-2} \leq \tau^\beta (r-s)^{\alpha-1}.
\]
On the other hand, since for $t \in (0,\tau)$ we have that $r-s+\tau-t \ge r-s$, the second term can be estimated as
\[
  \int_0^\tau (r-s-t+\tau)^{\alpha-1}t^{\beta-1} \diff t \leq (r-s)^{\alpha-1} \int_0^{\tau} t^{\beta-1} \diff t = \frac{1}{\alpha} (r-s)^{\alpha-1} \tau^{\beta}.
\]
This concludes the proof.
\end{proof}

We refer the reader to \cite[section 4]{LiLiuCompact-18} for further results concerning the space $L^p_{\alpha}(0, T; \mH)$. 

\subsection{The Caputo derivative}
\label{sub:Caputo}

As we mentioned in the Introduction, the definition of the Caputo derivative, given in \eqref{eq:defofCaputo} seems unnatural. Smoothness of higher order is needed to define a fractional derivative. Several attempts at resolving this discrepancy have been proposed in the literature and we here quickly describe a few of them.

First, one of the main reasons that motivate practitioners to use, among the many possible definitions, the Caputo derivative \eqref{eq:defofCaputo} is, first, that $D_c^\alpha 1 = 0$ and second that this derivative allows one to pose initial value problems like \eqref{eq:theGradFlow}. However, it is by now known that even in the linear case solutions of problems involving the Caputo derivative possess a weak singularity in time \cite{MR3966570,MR3820275,MR3589365}. This singular behavior of the solution forces one to wonder: If fractional derivatives describe processes with memory, why is it sufficient to know the state at one particular point (initial condition) to uniquely describe the state at all future times? Is it possible that the singularity is precisely caused by the fact that we are ignoring the past states of the system? This motivates the following: Set $w(t) =w_0$ for $t\leq 0$. Therefore,
\begin{equation}
\label{eq:defofMarchaud}
  \begin{aligned}
    D_c^\alpha w(t) &= \frac1{\Gamma(1-\alpha)} \int_{-\infty}^t \frac{\dot w(r)}{(t-r)^\alpha} \diff r =
    \frac1{\Gamma(1-\alpha)} \int_{-\infty}^t \frac{(w(r)-w(t))\dot{}}{(t-r)^\alpha} \diff r \\
    &= \frac1{\Gamma(-\alpha)} \int_{-\infty}^t \frac{ w(r) - w(t)}{(t-r)^{\alpha+1}} \diff r = D_m^\alpha w(t),
  \end{aligned}
\end{equation}
where, in the last step, we integrated by parts. The expression $D_m^\alpha w(t)$ is known as the Marchaud derivative of order $\alpha$ of the function $w$. This is the way that the Caputo derivative has been understood, for instance, in \cite{MR3488533,AllenNondiv,AllenPorous,AllenNondivHolder}. We comment, in passing, that owing to \cite{MR3456835} this fractional derivative satisfies an extension problem similar to the (by now) classical Caffarelli Silvestre extension \cite{MR2354493,MR2754080} for the fractional Laplacian.

Another approach, and the one we shall adopt here, is to notice that \eqref{eq:defofCaputo} can be converted, for sufficiently smooth functions, into a Volterra type equation
\begin{equation}
\label{eq:Volterra-eq}
  w(t) = w(0) + \dfrac{1}{\Gamma(\alpha)} \int_0^t (t-s)^{\alpha-1} D_c^{\alpha} w(s) \diff s, \quad \forall t \in [0, T].
\end{equation}
This identity is the beginning of the theory developed in \cite{LiLiuGeneralized-18} to extend the notion of Caputo derivative.
To be more specific, \cite{LiLiuGeneralized-18} considers the set of distributions
\[
\mathscr{E}^T = \{ w \in \mathscr{D}'(\mR;\mH): \exists M_{w} \in (-\infty, T), \textrm{supp}(w) \subset [-M_w, T) \}.
\]
for a fixed time $T > 0$. Then the modified Riemann Liouville derivative for any distribution $w \in \mathscr{E}^T$ is defined, following classical references like \cite[Section 1.5.5]{MR0097715}, as
\[
D_{rl}^\alpha w = w * g_{-\alpha} \in \mathscr{E}^T
\]
where $g_{-\alpha}(t) = \frac{1}{\Gamma(1-\alpha)} D(\theta(t) t^{-\alpha})$, with $\theta$ being the Heaviside function, is a distribution supported in $[0, \infty)$ and the convolution is understood as the generalized definition between distributions. Here $D$ denotes the distributional derivative. Reference \cite{LiLiuGeneralized-18} then uses this to define the generalized Caputo derivative of $w \in L^1_{\textrm{loc}}([0,T);\mH)$ associated with $w_0$ by
\[
D_{c}^\alpha  w = D_{rl}^\alpha (w - w_0).
\]
If there exists $w(0) \in \mH$ such that $\lim_{t \downarrow 0} \fint_0^t \Vert w(s) - w(0) \Vert \diff s = 0$, then we always impose $w_0 = w(0)$ in this definition. It is shown in \cite[Theorem 3.7]{LiLiuGeneralized-18} that for such a function $w$, \eqref{eq:Volterra-eq} holds for Lebesgue a.e.~$t \in (0,T)$ provided that the generalized Caputo derivative $D_{c}^\alpha w \in L^1_{\textrm{loc}}([0,T);\mH)$.

We also comment that \cite[Proposition 3.11(ii)]{LiLiuGeneralized-18} implies that for every function $w \in L^2(0,T;\mH)$ with
$D_c^\alpha w \in L^2(0,T;\mH)$ we have
\begin{equation}\label{eq:Caputo-ineq-norm-sq}
  \frac12 D_c^{\alpha} \Vert w \Vert^2(t) \le \l\langle D_c^{\alpha} w(t), w(t) \r\rangle.
\end{equation}

Finally, we recall that the Mittag-Leffler function of order $\alpha\in (0,1)$ is defined via
\[
  E_\alpha(z) = \sum_{k=0}^\infty \frac{z^k}{\Gamma(\alpha k+1)}.
\]
We refer the reader to \cite{MR3244285} for an extensive treatise on this function. Here we just mention that this function satisfies, for any $\lambda \in \mR$, the identity
\begin{equation}\label{eq:GronwallIsExp}
  D_c^\alpha E_\alpha( \lambda t^\alpha) = \lambda E_\alpha( \lambda t^\alpha), \qquad E_\alpha(0) = 1.
\end{equation}

\subsubsection{An auxiliary estimate}
Having defined the Caputo derivative of a function, we present an auxiliary result. Namely, an estimate on functions that have piecewise constant, over some partition $\mP$, Caputo derivative.

\begin{Lemma}[continuity]\label{lemma:alpha-int-w}
Let $p \in [1,\infty)$; $\mP$ be a partition, as in \eqref{eq:def-partition}, of $[0,T]$; and $w \in L^1(0,T;\mH)$ be such that its generalized Caputo derivative $D_c^{\alpha} w \in L^p_{\alpha}(0, T; \mH)$,  and it is piecewise constant over $\mP$. Then we have
\begin{equation}\label{eq:alpha-int-w}
  \sup_{r \in [0,T]} \int_0^r (r-t)^{\alpha-1} \Vert w(\tceil_\mP) - w(t) \Vert^p \diff t \le C \tau^{p\alpha} \Vert D_c^{\alpha}w \Vert^p_{L^p_{\alpha}(0, T; \mH)},
\end{equation}
where the constant $C$ depends only on $\alpha$.
\end{Lemma}
\begin{proof}
The representation \eqref{eq:Volterra-eq} allows us to write
\begin{multline*}
  w(\tceil_\mP) - w(t) =\\  \frac{1}{\Gamma(\alpha)} \Bigg[
  \int_0^t D_c^{\alpha}w(s) \l( (\tceil_\mP-s)^{\alpha-1} - (t-s)^{\alpha-1} \r) \diff s 
  + \int_{t}^{\tceil_\mP} D_c^{\alpha}w(s) (\tceil_\mP-s)^{\alpha-1} \diff s \Bigg].
\end{multline*}
Therefore by H\"older inequality, we have
\[
\begin{aligned}
  & \Vert w(\tceil_\mP) - w(t) \Vert^p
  \le \frac{1}{\Gamma^p(\alpha)} \l(  \int_0^t \l| (\tceil_\mP-s)^{\alpha-1} - (t-s)^{\alpha-1} \r| \diff s + \int_{t}^{\tceil_\mP} (\tceil_\mP-s)^{\alpha-1} \diff s \r)^{p-1} \\
  & \qquad \Bigg( \int_0^t \Vert D_c^{\alpha}w(s) \Vert^p \l| (\tceil_\mP-s)^{\alpha-1} - (t-s)^{\alpha-1} \r| \diff s
  + \int_{t}^{\tceil_\mP} \Vert D_c^{\alpha}w(s) \Vert^p (\tceil_\mP-s)^{\alpha-1} \diff s \Bigg)  \\
  &\le C \tau^{\alpha(p-1)}\Bigg( \int_0^t \Vert D_c^{\alpha}w(s) \Vert^p \l| (\tceil_\mP-s)^{\alpha-1} - (t-s)^{\alpha-1} \r| \diff s + \frac{(\tceil_\mP - t)^{\alpha}}{\alpha} \Vert D_c^{\alpha}w(t) \Vert^p \Bigg) \\
  &= C_1 \tau^{\alpha(p-1)} \int_0^t \Vert D_c^{\alpha}w(s) \Vert^p \l| (\tceil_\mP-s)^{\alpha-1} - (t-s)^{\alpha-1} \r| \diff s
  + C_2 \tau^{p\alpha}\Vert D_c^{\alpha}w(t) \Vert^p
  = \mathrm{I}_1(t) + \mathrm{I}_2(t),
\end{aligned}
\]
where the constants $C$, $C_1$, and $C_2$ depend only on $p$ and $\alpha$.

For $\mathrm{I}_2(t)$, we simply have
\[
\int_0^r (r-t)^{\alpha-1} \mathrm{I}_2(t) \diff t \le
C \tau^{p\alpha} \Vert D_c^{\alpha}w \Vert^p_{L^p_{\alpha}(0, T; \mH)}.
\]
Now to bound the integral for $\mathrm{I}_1(t)$, we use Fubini's theorem to get
\[
  \begin{aligned}
    &\int_0^r (r-t)^{\alpha-1} \mathrm{I}_1(t) \diff t= 
    C_1 \tau^{(p-1)\alpha} \int_0^r \l\Vert D_c^{\alpha}w(s) \r\Vert^p  \int_s^r (r-t)^{\alpha-1}
    \l| (\tceil_\mP-s)^{\alpha-1} - (t-s)^{\alpha-1} \r| \diff t \diff s.
  \end{aligned}
\]
We claim that
\begin{equation}\label{eq:alpha-int-w-proof1}
\int_s^r (r-t)^{\alpha-1}
\l| (\tceil_\mP-s)^{\alpha-1} - (t-s)^{\alpha-1} \r| \diff t
\le C_3 (r-s)^{\alpha-1} \tau^{\alpha},
\end{equation}
where $C_3$ only depends $\alpha$. If this is true, then we have
\[
\begin{aligned}
&\int_0^r (r-t)^{\alpha-1} \mathrm{I}_1(t) \diff t \le
C \tau^{p\alpha} \int_0^r \l\Vert D_c^{\alpha}w(s) \r\Vert^p (r-s)^{\alpha-1} \diff s
\le C \tau^{p\alpha} \Vert D_c^{\alpha}w \Vert^p_{L^p_{\alpha}(0, T; \mH)}.
\end{aligned}
\]

The proof of \eqref{eq:alpha-int-w-proof1} proceeds as the one for \eqref{eq:newstuffBetafunc}. For brevity we skip the details.
\end{proof}

\subsection{Some comparison estimates}
As a final preparatory step we present some  auxiliary results that shall be repeatedly used and are related to differential inequalities involving the Caputo derivative, and a Gr\"onwall-like lemma.

First, we present a comparison principle which is similar to \cite[Proposition 4.2]{MR3848192}. The proof can be done easily by contradiction, and therefore it is omitted here.

\begin{Lemma}[comparison]\label{lemma:frac-int-comparison}
Let $g_1,g_2 : [0,T] \times \mR \to \mR$ be both nondecreasing in their second argument and $g_2$ be measurable. Assume that $v, w \in C([0,T];\mR)$ satisfy $v(0) < w(0)$, and there is some $\alpha \in (0,1)$, for which
\[
  \begin{aligned}
    v(t) &\le g_1(t, v(t)) + \frac{1}{\Gamma(\alpha)} \int_0^t (t-s)^{\alpha-1} g_2(s,v(s)) \diff s, \\
    w(t) &> g_1(t, w(t)) + \frac{1}{\Gamma(\alpha)} \int_0^t (t-s)^{\alpha-1} g_2(s,w(s)) \diff s,
  \end{aligned}
\]
for every $t \in [0, T]$. Then we have $v < w$ on $[0,T]$. 
\end{Lemma}

We now present a result that can be interpreted as an extension of \cite[Lemma 3.7]{MR1737503} to the fractional case. 
However, unlike the classical case, here we have the restriction that $\lambda \ge 0$ because we have to argue from a fractional integral inequality. Nevertheless, this is sufficient for our purposes.

\begin{Lemma}[fractional Gr\"onwall] \label{lemma:frac-Gronwall-diff-ineq}
Let $a \in C([0,T];\mR)$ with $D_c^{\alpha} a^2 \in L^1_{\textrm{loc}}([0,T);\mR)$, $b,c,d: [0,T] \to [0, +\infty]$ be measurable functions, and $\lambda \geq 0$. If the following differential inequality is satisfied
\begin{equation}\label{eq:frac-Gronwall-diff-ineq-assump}
  D_c^{\alpha} a^2(t) + b(t) \le 2 \lambda a^2(t) + c(t) + 2 d(t) a(t), \quad a.e. \;\; t \in (0,T),
\end{equation}
then we have
\[
  \l( \sup_{t \in [0,T]} a^2(t) + \frac{1}{\Gamma(\alpha)} \| b\|_{L^1_\alpha(0,T;\mR)} \r)^{1/2} \leq 2\wt{D}(T) E_{\alpha}(2\lambda T^{\alpha}) + 
  \sqrt{a^2(0) + \wt{C}(T)} \sqrt{E_{\alpha}(2\lambda T^{\alpha})}
\]
where
\begin{equation}\label{eq:frac-Gronwall-diff-ineq-def-CD}
  \wt{C}(t) = \frac1{\Gamma(\alpha)} \| c \|_{L^1_\alpha(0,t;\mR)}, \quad 
%   \sup_{0 \le s \le t} \frac{1}{\Gamma(\alpha)} \int_0^t (t-s)^{\alpha-1} c(s) \; \diff s, \quad
  \wt{D}(t) = \frac1{\Gamma(\alpha)} \| d \|_{L^1_\alpha(0,t;\mR)}.
%   \sup_{0 \le s \le t} \frac{1}{\Gamma(\alpha)} \int_0^t (t-s)^{\alpha-1} d(s) \; \diff s.
\end{equation}
\end{Lemma}
\begin{proof}
From \eqref{eq:frac-Gronwall-diff-ineq-assump} we obtain that
\begin{equation}\label{eq:lemma:frac-Gronwall-diff-ineq-proof1}
\begin{aligned}
  a^2(t) + \frac{1}{\Gamma(\alpha)} \int_0^t (t-s)^{\alpha-1} b(s) \diff s
  & \leq a^2(0) + \frac{1}{\Gamma(\alpha)} \int_0^t (t-s)^{\alpha-1} \l[ c(s) + 2 d(s) a(s) + 2 \lambda a^2(s) \r] \diff s \\
  & \leq a^2(0) + \wt{C}(t) + 2 \wt{a}(t) \wt{D}(t) + \frac{2\lambda}{\Gamma(\alpha)} \int_0^t (t-s)^{\alpha-1} \wt{a}^2(s) \; \diff s,
\end{aligned}
\end{equation}
where $\wt{a}(t) = \max_{0 \le s \le t} a(s)$ and the functions $\wt{C}, \wt{D}$ are defined in \eqref{eq:frac-Gronwall-diff-ineq-def-CD}. This immediately implies that
\[
  \wt{a}^2(t) \leq a^2(0) + \wt{C}(t) + 2 \wt{a}(t) \wt{D}(t) + \frac{2\lambda}{\Gamma(\alpha)} \int_0^t (t-s)^{\alpha-1} \wt{a}^2(s)  \diff s.
\]
In order to bound $\wt{a}$, we construct a barrier function $e(t) = K \sqrt{E_{\alpha}(2 \lambda t^{\alpha})}$ where the constant $K$ is chosen so that
\[
  e^2(t) > a^2(0) + \wt{C}(t) + 2 e(t) \wt{D}(t) + \frac{2\lambda}{\Gamma(\alpha)} \int_0^t (t-s)^{\alpha-1} e^2(s) \diff s, \quad \forall t \in (0,T).
\]
Indeed, owing to \eqref{eq:GronwallIsExp} we see that
\[
  \frac{2\lambda}{\Gamma(\alpha)} \int_0^t (t-s)^{\alpha-1} E_{\alpha}(2 \lambda s^{\alpha}) \; \diff s = E_{\alpha}(2 \lambda t^{\alpha}) - E_{\alpha}(0)
  = E_{\alpha}(2 \lambda t^{\alpha}) - 1
\]
and hence
\[
  \begin{aligned}
    & a^2(0) + \wt{C}(t) + e(t) \wt{D}(t) + \frac{2\lambda}{\Gamma(\alpha)} \int_0^t (t-s)^{\alpha-1} e^2(s) \; \diff s \\
    &=  a^2(0) + \wt{C}(t) + 2 K \sqrt{E_{\alpha}(2 \lambda t^{\alpha})} \wt{D}(t)
    + K^2 \l( E_{\alpha}(2 \lambda t^{\alpha}) - 1 \r) < K^2 E_{\alpha}(2 \lambda t^{\alpha}) = e^2(t),
  \end{aligned}
\]
for every $t \in (0,T)$ provided that
\begin{equation}\label{eq:lemma:frac-Gronwall-diff-ineq-proof2}
  K > \wt{D}(T) \sqrt{E_{\alpha}(2\lambda T^{\alpha})} + \sqrt{a^2(0) + \wt{C}(T) + \wt{D}^2(t) E_{\alpha}(2 \lambda T^{\alpha}) }.
\end{equation}
Applying \Cref{lemma:frac-int-comparison} we obtain that
\[
  \wt{a}(t) \le e(t) = K \sqrt{E_{\alpha}(2 \lambda t^{\alpha})}.
\]
Plugging this back into \eqref{eq:lemma:frac-Gronwall-diff-ineq-proof1} and noticing that this holds for any $K$ satisfying \eqref{eq:lemma:frac-Gronwall-diff-ineq-proof2} we obtain that
\[
  \begin{aligned}
    & \sup_{t \in [0,T]} a^2(t) + \frac{1}{\Gamma(\alpha)} \int_0^t (t-s)^{\alpha-1} b(s) \diff s \\ 
    & \leq \l( \wt{D}(T) \sqrt{E_{\alpha}(2\lambda T^{\alpha})} + 
    \sqrt{a^2(0) + \wt{C}(T) + \wt{D}^2(t) E_{\alpha}(2 \lambda T^{\alpha}) } \r)^2 E_{\alpha}(2\lambda T^{\alpha}) \\
    & \leq \l( 2\wt{D}(T) E_{\alpha}(2\lambda T^{\alpha}) + 
    \sqrt{a^2(0) + \wt{C}(T)} \sqrt{E_{\alpha}(2\lambda T^{\alpha})} \r)^2
  \end{aligned}
\]
which is the desired result.
\end{proof}

\section{Deconvolutional discretization of the Caputo derivative} \label{sec:deconvo}

To discretize the Caputo fractional derivative, references \cite{LiLiuNoteDeconvo-18, LiLiuGradientFlow-19} consider a so-called deconvolutional scheme on uniform time grids and prove some properties of this discretization. In this section, we generalize this deconvolutional scheme to the variable time step setting, and prove properties that will be useful in deriving a posteriori error estimates later, in Section~\ref{sec:post-error}. 

\subsection{The discrete Caputo derivative}
Let $\mP$ be a partition as in \eqref{eq:def-partition}. To motivate this discretization, let us assume that $w:[0,T] \to \mH$ is such that $D_c^{\alpha} w(t)$ is piecewise constant on the partition $\mP$, with
\[
  D_c^{\alpha} w(t) = V_{n(t)}.
\]
Then formally by \eqref{eq:Volterra-eq}, we have
\begin{equation}\label{eq:utn_Fi}
\begin{aligned}
  w(t_n) &= w(0) + \dfrac{1}{\Gamma(\alpha)} \int_0^{t_n} (t_n-s)^{\alpha-1} D_c^{\alpha} w(s) \diff s\\
  &= w(0) + \dfrac{1}{\Gamma(\alpha+1)} \sum_{i=1}^n \l( (t_n - t_{i-1})^{\alpha} - (t_n - t_{i})^{\alpha} \r) V_i, \quad n \in \{1,\ldots, N\}.
\end{aligned}
\end{equation}

Let $\bK_\mP \in \mR^{N \times N}$ be the matrix induced by the partition $\mP$, which is defined as
\begin{equation}\label{eq:def-K}
  \bK_{\mP,ni} 
  = 
  \begin{dcases}
  \dfrac{1}{\Gamma(\alpha+1)}\Big( (t_n - t_{i-1})^{\alpha} - (t_n - t_{i})^{\alpha} \Big), & 1 \le i \le n \le N, \\
  0, & 1 \le n < i \le N.
  \end{dcases} 
\end{equation}
Then we can rewrite \eqref{eq:utn_Fi} in matrix form as
\[
  \bW = \bW_0 + \bK_{\mP} \bV,
\]
where $\bV, \bW, \bW_0 \in \mH^N$ with $\bV_n = V_n$, $\bW_n = w(t_n)$, and $(\bW_0)_n = w(0)$.
Notice that $\bK_{\mP}$ is lower triangular and all the elements on and below the main diagonal are positive. Therefore $\bK_{\mP}$ is invertible and its inverse is also lower triangular. Thus, the previous identity is equivalent to
\[
  \bV = \bK^{-1}_{\mP} (\bW - \bW_0),
\]
in other words
\[
  V_n = \sum_{i=1}^n \bK^{-1}_{\mP, ni} (W_i - W_0) = \bK^{-1}_{\mP, n0} W_0 + \sum_{i=1}^n \bK^{-1}_{\mP, ni} W_i, 
\]
where we set $\bK^{-1}_{\mP,n0} = -\sum_{j=1}^n \bK^{-1}_{\mP,nj}$. This motivates the following approximation of the Caputo derivative provided $\bW \in \mH^N$ and $W_0 \in \mH$ are given. For $n \in \{1, \ldots, N\}$ we set
\begin{equation}\label{eq:discrete-Caputo}
  \l( D_{\mP}^{\alpha} \bW \r)_n = \sum_{i=1}^n \bK^{-1}_{\mP, ni} (W_i - W_0) = %K^{-1}_{\mP, n0} U_0 + 
\sum_{i=0}^n \bK^{-1}_{\mP, ni} W_i
= \sum_{i=0}^{n-1} \bK^{-1}_{\mP, ni} (W_i - W_n).
\end{equation}

\subsection{Properties of $\bK_{\mP}^{-1}$}
We note that, when the partition is uniform, both $\bK_{\mP}$ and its inverse will be Toeplitz matrices, and hence the product $\bK_\mP \bV$ can be interpreted as the convolution of sequences. Consequently, multiplication by $\bK_\mP^{-1}$ is equivalent to taking a sequence deconvolution. This motivates the name of this scheme and enables \cite{LiLiuGradientFlow-19} to apply techniques for the deconvolution of a completely monotone sequence and prove properties of $\bK_{\mP}^{-1}$.

We were not successful in extending, to a general partition $\mP$, all the properties of $\bK_\mP^{-1}$ presented in \cite{LiLiuGradientFlow-19} for the case when the partition is uniform. This is mainly because their techniques are based on ideas that rely on completely monotone sequences, which do not easily extend to a general $\mP$. Nevertheless we have obtained sufficient, for our purposes, properties. The following result is the counterpart to \cite[Proposition 3.2(1)]{LiLiuGradientFlow-19}.

\begin{Proposition}[properties of $\bK_\mP^{-1}$]
\label{prop:mat_prop1}
Let $\mP$ be a partition as in \eqref{eq:def-partition}, and $\bK_{\mP}$ be defined in \eqref{eq:def-K}. The matrix $\bK_\mP$ is invertible, and its inverse satisfies: 
\begin{eqnarray}
  \bK^{-1}_{\mP,n0} = -\sum_{j=1}^n  \bK^{-1}_{\mP,nj} < 0, \quad n \in \{1, \ldots, N\} \label{eq:mat_prop1-1}, \\
  \bK^{-1}_{\mP,ii} > 0 \quad i \in \{1, \ldots, N\}, \quad \bK^{-1}_{\mP,ni} < 0 \quad 1 \le i < n \le N. \label{eq:mat_prop1-2} 
\end{eqnarray}
\end{Proposition}
\begin{proof}
We already showed that $\bK_\mP$ is nonsingular. We prove \eqref{eq:mat_prop1-1} and \eqref{eq:mat_prop1-2} separately.

First, to prove that $\bK^{-1}_{\mP,n0} < 0$. For this, it suffices to show that for a vector $\bW \in \mR^N$ such that $W_i = 1$ for any $i \ge 1$, then the vector $\bF = \bK_{\mP}^{-1}\bW$ satisfies
\[
  F_n > 0 \quad \forall n \ge 1.
\]
We prove this by induction on $n$. For $n = 1$, clearly
\[
  F_1 = \frac{W_1}{\bK_{\mP,1,1}} = \frac{1}{\bK_{\mP,1,1}} > 0.
\]
Suppose that $F_j > 0$ for all $1 \le j \le k$, now we want to show that $F_{k+1} > 0$ as well. Notice that
\[
  1 = W_k = \sum_{j=1}^k \bK_{\mP,k,j}F_j, \quad 1 = W_{k+1} = \sum_{j=1}^{k+1} \bK_{\mP,k+1,j}F_j,
\]
then taking the difference we have
\begin{equation}\label{eq:mat_prop1-proof2}
  0 = \sum_{j=1}^{k+1} \bK_{\mP,k+1,j}F_j - \sum_{j=1}^k \bK_{\mP,k,j}F_j = \bK_{\mP,k+1,k+1}F_{k+1} + \sum_{j=1}^k (\bK_{\mP,k+1,j}-\bK_{\mP,k,j}) F_j.
\end{equation}
We claim that $\bK_{\mP,k+1,j}-\bK_{\mP,k,j} < 0$ for any $j$. In fact, this can be seen through the definition of the entries of $\bK_\mP$
\[
  \begin{aligned}
    \bK_{\mP,k+1,j}-\bK_{\mP,k,j} < 0 & \iff (t_{k+1} - t_{j-1})^{\alpha} - (t_{k+1} - t_{j})^{\alpha} < (t_k - t_{j-1})^{\alpha} - (t_k - t_{j})^{\alpha}\\
    &\iff
    \int_{0}^{t_j - t_{j-1}} (t_{k+1} - t_j + s)^{\alpha -1} \diff s
    < \int_{0}^{t_j - t_{j-1}} (t_{k} - t_j + s)^{\alpha -1} \diff s.
  \end{aligned}
\]
Using $\bK_{\mP,k+1,j}-\bK_{\mP,k,j} < 0$ and $F_j > 0$ for all $j \in \{1, \ldots, k\}$ in \eqref{eq:mat_prop1-proof2}, we see that $\bK_{\mP,k+1,k+1}F_{k+1} > 0$ and thus $F_{k+1} > 0$.
Therefore by induction we proved that $\bK^{-1}_{\mP,n0} < 0$ for $n \ge 1$.

Next, we prove that $\bK^{-1}_{\mP,ii} > 0$ and $\bK^{-1}_{\mP,ni} < 0$. Consider a vector $\bW \in \mR^N$ that is such that $W_i = 1$ and $W_j = 0$ for $j \neq i$. It suffices to prove that for, $\bF = \bK^{-1}_\mP \bW$, we have $F_i > 0$ and if $n>i$
\begin{equation}\label{eq:mat_prop1-proof1}
  F_n < 0.
\end{equation}
Since $\bK_\mP^{-1}$ is lower triangular, we know $F_j = 0$ for $j \in \{1, \ldots, i-1\}$. From $\bK_\mP \bF = \bW$, we see that
\[
  1 = W_i = (\bK_\mP \bF)_i = \sum_{j=1}^i \bK_{\mP,ij} F_j = \bK^{-1}_{\mP,ii} F_i
\]
and thus $F_i = 1/\bK_{\mP,ii} > 0$. Now we prove by induction that \eqref{eq:mat_prop1-proof1} holds. First, when $n = i+1$, we have
\[
  0 = W_{i+1} = (\bK_\mP \bF)_{i+1} = \bK_{\mP,i+1,i} F_i + \bK_{\mP,i+1,i+1} F_{i+1}
\]
and hence
\[
  F_{i+1} = -\frac{ \bK_{\mP,i+1,i} F_i }{ \bK_{\mP,i+1,i+1} } < 0.
\]
This shows that \eqref{eq:mat_prop1-proof1} is true for $n=i+1$. Now suppose that we have already shown that $F_n < 0$ for $n$ satisfying $n \in \{i+1, \ldots, k\}$, we want to prove $F_{k+1} < 0$. To this aim, notice that
\[
  0 = W_{k+1} = (\bK_{\mP}\bF)_{k+1} = \sum_{j=i}^{k} \bK_{\mP,k+1,j}F_j + \bK_{\mP,k+1,k+1} F_{k+1},
\]
therefore we only need to show $\sum_{j=i}^{k} \bK_{\mP,k+1,j}F_j > 0$. Recall that
\[
  0 = W_k = (\bK_\mP \bF)_{k} = \sum_{j=i}^{k} \bK_{\mP,k,j}F_j,
\]
and thus, since $\bK_{\mP,k,i} > 0$, we can get
\[
  \sum_{j=i}^{k} \bK_{\mP,k+1,j}F_j = \sum_{j=i}^{k} \bK_{\mP,k+1,j}F_j - \frac{\bK_{\mP,k+1,i}}{\bK_{\mP,k,i}} \sum_{j=i}^{k} \bK_{\mP,k,j}F_j = \sum_{j=i+1}^{k} \l( \bK_{\mP,k+1,j} - \frac{\bK_{\mP,k+1,i}}{\bK_{\mP,k,i}}\bK_{\mP,k,j} \r) F_j.
\]
Since by the induction hypothesis $F_j < 0$ for $j \in \{i+1, \ldots, k\}$, it only remains to show that
\[
  \bK_{\mP,k+1,j} - \frac{\bK_{\mP,k+1,i}}{\bK_{\mP,k,i}}\bK_{\mP,k,j} < 0 \; \iff \;
  \frac{\bK_{\mP,k+1,i}}{\bK_{\mP,k,i}} > \frac{\bK_{\mP,k+1,j}}{\bK_{\mP,k,j}}.
\]
Applying Cauchy's mean value theorem, there exists $\eta \in (t_k - t_i, t_k-t_{i-1})$
such that
\[
  \frac{\bK_{\mP,k+1,i}}{\bK_{\mP,k,i}} = \frac{(t_{k+1} - t_{i-1})^{\alpha} - (t_{k+1} - t_{i})^{\alpha}}{(t_k - t_{i-1})^{\alpha} - (t_k - t_{i})^{\alpha}} = \frac{ \alpha(\eta + \tau_{k+1})^{\alpha-1} }{ \alpha \eta^{\alpha-1} } = \l( \frac{\eta + \tau_{k+1}}{\eta} \r)^{\alpha-1}.
\]
Similarly there exists $\xi \in (t_k - t_j, t_k-t_{j-1})$ such that
\[
  \frac{\bK_{\mP,k+1,j}}{\bK_{\mP,k,j}} = \l( \frac{\xi + \tau_{k+1}}{\xi} \r)^{\alpha-1}.
\]
Due to $j > i$, we have $\xi < \eta$ and hence
\[
  \frac{\bK_{\mP,k+1,j}}{\bK_{\mP,k,j}} = \l( \frac{\xi + \tau_{k+1}}{\xi} \r)^{\alpha-1} < \l( \frac{\eta + \tau_{k+1}}{\eta} \r)^{\alpha-1} = \frac{\bK_{\mP,k+1,i}}{\bK_{\mP,k,i}}.
\]
Therefore from the arguments above we see that $F_{k+1} < 0$, and by induction $\bK^{-1}_{\mP,ni} < 0$ for $n > i$.
\end{proof}

\begin{Remark}[generalization]
The discretization of the Caputo derivative, described in \eqref{eq:discrete-Caputo}, and its properties presented in Proposition~\ref{prop:mat_prop1} can be extended to more general kernels. Indeed, for a general convolutional kernel $g \in L^1(0,T;\mR)$ the entries of the matrix $\bK_\mP$ will be 
\[
  \bK_{\mP,ni} = \int_{t_n - t_{i-1}}^{t_n - t_i} g(t) \diff t.
\]
The proof of \eqref{eq:mat_prop1-1} follows \emph{verbatim} provided $g'(t) < 0$, as the reader can readily verify. The proof of \eqref{eq:mat_prop1-2} only requires that the function $G(t) = \ln(g(t))$, satisfies $G''(t) > 0$.
\end{Remark}

For a uniform time grid $\mP$, \cite[Theorem 2.3]{LiLiuNoteDeconvo-18} proves that, for every $i$, the sequence $\{-\bK^{-1}_{\mP,n+i,i} \}_{n \ge 1}$ is completely monotone. The following result holds for a general partition $\mP$, and is a direct consequence of \cite[Theorem 2.3]{LiLiuNoteDeconvo-18} for uniform time stepping.

\begin{Proposition}[monotonicity]\label{prop:mat_prop2}
Let $\mP$ be a partition of $[0,T]$ as in \eqref{eq:def-partition}, and $\bK_\mP$ be defined as in \eqref{eq:def-K}. Then, its inverse satisfies:
\begin{enumerate}[1.]
  \item For $n \in \{1, \ldots, N-1\}$,
  \begin{equation}\label{eq:mat_prop2-1}
    -\sum_{j=1}^n \bK^{-1}_{\mP,nj} = \bK^{-1}_{\mP,n0} < \bK^{-1}_{\mP,n+1,0} = -\sum_{j=1}^{n+1} \bK^{-1}_{\mP,n+1, j}.
  \end{equation}
  
  \item For $1 \le i < n < N$,
  \begin{equation}\label{eq:mat_prop2-2}
    \bK^{-1}_{\mP,ni} < \bK^{-1}_{\mP,n+1,i}.
  \end{equation}
\end{enumerate}
\end{Proposition}
\begin{proof}
To prove \eqref{eq:mat_prop2-1} it suffices to show that for a vector $\bW \in \mR^N$ such that $W_i = 1$ for any $i \ge 1$, then the vector $\bF = \bK^{-1}_\mP \bW$ satisfies
\[
  F_n > F_{n+1} \quad \forall n \ge 1.
\]
We prove this by induction on $n$. For $n = 1$,
\[
  \begin{aligned}
    1 &= W_1 = (\bK_\mP \bF)_1 = \bK_{\mP,11} F_1, \\
    1 &= W_2 = (\bK_\mP \bF)_2 = \bK_{\mP,21} F_1 + \bK_{\mP,22} F_2 = (\bK_{\mP,21}+\bK_{\mP,22}) F_1 + \bK_{\mP,22} (F_2-F_1).
  \end{aligned}
\]
Clearly,
\[
  F_1 > 0, \quad \bK_{\mP,11} = (t_1 - t_0)^{\alpha} < (t_2 - t_0)^{\alpha} = \bK_{\mP,21}+\bK_{\mP,22}.
\]
Hence we have
\[
  \bK_{\mP,22} (F_2-F_1) = 1 - (\bK_{\mP,21}+\bK_{\mP,22}) F_1 <  1 - \bK_{\mP,11} F_1 = 0,
\]
which, since $\bK_{\mP,22}>0$, implies that $F_2 - F_1 < 0$, i.e. $F_1 > F_2$. So the claim holds for $n = 1$.

Suppose $F_{j+1} < F_j$ for all $1 \le j < k$, now we want to show that $F_{k+1} < F_k$ as well. Notice that
\[
  \begin{aligned}
    1 &= W_k = \sum_{i=1}^{k} \bK_{\mP,ki} F_i =
    \sum_{i=0}^{k-1} \l( \sum_{j=i+1}^k \bK_{\mP,kj} \r) (F_{i+1} - F_i)
    = \sum_{i=0}^{k-1} (t_k - t_i)^{\alpha} (F_{i+1} - F_i), \\
    1 &= W_{k+1} = \sum_{i=1}^{k+1} \bK_{\mP,k+1,i} F_i = \sum_{i=0}^{k} (t_{k+1} - t_i)^{\alpha} (F_{i+1} - F_i),
  \end{aligned}
\]
where we set $F_0 = 0$ in the equations above. Therefore to show $F_{k+1} < F_k$, we only need to prove that
\begin{equation}\label{eq:mat_prop2-proof1}
  \begin{aligned}
    0 < \sum_{i=0}^{k-1} (t_{k+1} - t_i)^{\alpha} (F_{i+1} - F_i) - 1
    &= \sum_{i=0}^{k-1} (t_{k+1} - t_i)^{\alpha} (F_{i+1} - F_i) - \sum_{i=0}^{k-1} (t_k - t_i)^{\alpha} (F_{i+1} - F_i) \\
    &= \sum_{i=0}^{k-1} \big( (t_{k+1} - t_i)^{\alpha} - (t_k - t_i)^{\alpha} \big) (F_{i+1} - F_i).
  \end{aligned}
\end{equation}
Since we also have
\[
  1 = W_{k-1} = \sum_{i=1}^{k-1} \bK_{\mP,k-1,i} F_i = \sum_{i=0}^{k-2} (t_{k-1} - t_i)^{\alpha} (F_{i+1} - F_i) = \sum_{i=0}^{k-1} (t_{k-1} - t_i)^{\alpha} (F_{i+1} - F_i),
\]
Taking the difference between the equation above and the one for $W_k$, we obtain that
\begin{equation*}
  \begin{aligned}
    0 = W_k - W_{k-1} &= \sum_{i=0}^{k-1} (t_k - t_i)^{\alpha} (F_{i+1} - F_i) - \sum_{i=0}^{k-1} (t_{k-1} - t_i)^{\alpha} (F_{i+1} - F_i) \\
    &= \sum_{i=0}^{k-1} \big( (t_{k} - t_i)^{\alpha} - (t_{k-1} - t_i)^{\alpha} \big) (F_{i+1} - F_i)
  \end{aligned}
\end{equation*}
In light of this identity, we claim that to obtain \eqref{eq:mat_prop2-proof1} it suffices to show that
\begin{equation}\label{eq:mat_prop2-proof2}
  \dfrac{t_{k+1}^{\alpha} - t_k^{\alpha}}{ t_{k}^{\alpha} - t_{k-1}^{\alpha}}
  = \dfrac{(t_{k+1} - t_0)^{\alpha} - (t_k - t_0)^{\alpha}}{ (t_{k} - t_0)^{\alpha} - (t_{k-1} - t_0)^{\alpha}} > 
  \dfrac{(t_{k+1} - t_i)^{\alpha} - (t_k - t_i)^{\alpha}}{ (t_{k} - t_i)^{\alpha} - (t_{k-1} - t_i)^{\alpha}}, \quad i \in \{1 , \ldots, k-1\}.
\end{equation}
If this is true, letting $c = \l( t_{k+1}^{\alpha} - t_k^{\alpha}\r)/\l(t_{k}^{\alpha} - t_{k-1}^{\alpha}\r)$ we have:
\[
  \begin{aligned}
    & \sum_{i=0}^{k-1} \big( (t_{k+1} - t_i)^{\alpha} - (t_k - t_i)^{\alpha} \big) (F_{i+1} - F_i) \\
    & = \sum_{i=0}^{k-1} \Big( \big( (t_{k+1} - t_i)^{\alpha} - (t_k - t_i)^{\alpha} \big) - c \big( (t_{k} - t_i)^{\alpha} - (t_{k-1} - t_i)^{\alpha} \big) \Big) (F_{i+1} - F_i) \\
    & = \sum_{i=1}^{k-1} \Big( \big( (t_{k+1} - t_i)^{\alpha} - (t_k - t_i)^{\alpha} \big) - c \big( (t_{k} - t_i)^{\alpha} - (t_{k-1} - t_i)^{\alpha} \big) \Big) (F_{i+1} - F_i) \\
    &= \sum_{i=1}^{k-1} d_i \; (F_{i+1} - F_i),
  \end{aligned}
\]
where $d_i = \big( (t_{k+1} - t_i)^{\alpha} - (t_k - t_i)^{\alpha} \big) - c \big( (t_{k} - t_i)^{\alpha} - (t_{k-1} - t_i)^{\alpha} \big) < 0$ due to \eqref{eq:mat_prop2-proof2}. By the inductive hypothesis, $F_{i+1} - F_i < 0$ for $1 \le i \le k-1$, so the equation above implies \eqref{eq:mat_prop2-proof1}, and hence 
$F_{k+1} < F_k$ is proved.

To finish the proof, we focus on \eqref{eq:mat_prop2-proof2}, fix $i$ and define $c_1 = t_{k-1} - t_i$, $c_2 = t_{k} - t_i$, $c_3 = t_{k+1} - t_i$ and function
\[
  h(x) = \dfrac{(x+c_3)^{\alpha} - (x+c_2)^{\alpha}}{(x+c_2)^{\alpha} - (x+c_1)^{\alpha}}.
\]
Then \eqref{eq:mat_prop2-proof2} is equivalent to $h(t_i - t_0) > h(0)$, and it remains to show that $h(x)$ is strictly increasing for $x > 0$. We observe that
\[
\dfrac{\diff }{\diff x}\l( \ln(h(x)) \r) =
\alpha \l[ \dfrac{(x+c_3)^{\alpha-1} - (x+c_2)^{\alpha-1}}{(x+c_3)^{\alpha} - (x+c_2)^{\alpha}} - \dfrac{(x+c_2)^{\alpha-1} - (x+c_1)^{\alpha-1}}{(x+c_2)^{\alpha} - (x+c_1)^{\alpha}} \r].
\]
Applying Cauchy's mean-value theorem to the two fractions above, we know there exists $\eta \in (x+c_2, x+c_3)$ and $\xi \in (x+c_1, x+c_2)$ such that
\[
\dfrac{\diff}{\diff x}\l( \ln(h(x)) \r) = \alpha \l[ \dfrac{(\alpha-1) \eta^{\alpha-2}}{\alpha \eta^{\alpha-1}} - 
\dfrac{(\alpha-1) \xi^{\alpha-2}}{\alpha \xi^{\alpha-1}}
\r]
= (\alpha-1) \l( \eta^{-1} - \xi^{-1} \r) > 0,
\]
where the last inequality holds because $\alpha < 1$ and $\xi< x+c_2< \eta$. This shows the monotonicity of function $h$ and confirms \eqref{eq:mat_prop2-proof2}. This concludes the inductive step and proves \eqref{eq:mat_prop2-1}. 

The proof of \eqref{eq:mat_prop2-2} is obtained similarly. For convenience we only write the proof for $i=1$, but the extension to general $i$ is straightforward. Consider a vector $\bW \in \mR^N$ such that $W_j = 1$ if $j = 1$ and $W_j = 0$ if $j \neq 1$, then it suffices to prove that vector $\bF = \bK_\mP^{-1} \bW$ satisfies
\begin{equation}\label{eq:mat_prop2-proof3}
  F_n < F_{n+1}
\end{equation}
for $n \in \{2,\ldots, N-1\}$. We prove \eqref{eq:mat_prop2-proof3} by induction on $n$. For $n = 2$, observe that
\[
  W_k = \sum_{j=0}^k (t_{k} - t_j)^{\alpha} (F_{j+1} - F_j) = \sum_{j=0}^{k-1} (t_{k} - t_j)^{\alpha} (F_{j+1} - F_j)
\]
from the proof of \eqref{eq:mat_prop2-1} with $F_0 = 0$, we have
\[
  \begin{aligned}
    1 &= W_1 = (t_1 - t_{0})^{\alpha} (F_{1} - F_{0}) \\
    0 &= W_2 = (t_{2} - t_{0})^{\alpha} (F_{1} - F_{0}) + (t_{2} - t_{1})^{\alpha} (F_{2} - F_{1}) \\
    0 &= W_3 = (t_{3} - t_{0})^{\alpha} (F_{1} - F_{0}) + (t_{3} - t_{1})^{\alpha} (F_{2} - F_{1}) + (t_{3} - t_{2})^{\alpha} (F_{3} - F_{2})
  \end{aligned}
\]
From the first and second equation above, we see that $F_1 > 0$ and $F_2 - F_1 < 0$. Combining the second and the third equation we deduce that
\[
  0 = W_3 - \dfrac{t_3^{\alpha}}{t_2^{\alpha}} W_2
  = \l[(t_3 - t_1)^{\alpha} - (t_2 - t_1)^{\alpha} \dfrac{t_3^{\alpha}}{t_2^{\alpha}} \r] (F_2 - F_1) + (t_3 - t_2)^{\alpha} (F_{3} - F_{2}).
\]
Since $(t_3 - t_1)^{\alpha} - (t_2 - t_1)^{\alpha} (t_3/t_2)^{\alpha} = (t_3 - t_1)^{\alpha} - (t_3 - (t_1 t_3/t_2))^{\alpha} > 0$, we obtain that $F_3 - F_2 > 0$ which is \eqref{eq:mat_prop2-proof3} for $n = 2$.

It also remains to prove that when \eqref{eq:mat_prop2-proof3} holds for $n \in \{2, \ldots, k-1\}$, then it also holds for $n = k$, i.e. $F_k < F_{k+1}$, provided that $k < N$. To this aim, we first see that
\begin{equation}\label{eq:mat_prop2-proof4}
  0 = W_{k+1} - \dfrac{t_{k+1}^{\alpha}}{t_{k}^{\alpha}} W_{k}
  = \sum_{j=1}^{k} \l( (t_{k+1} - t_j)^{\alpha} - (t_k - t_j)^{\alpha} 
  \dfrac{t_{k+1}^{\alpha}}{t_{k}^{\alpha}}\r) (F_{j+1} - F_j).
\end{equation}
Therefore in order to prove $F_k < F_{k+1}$, we only need to show that
\begin{equation}\label{eq:mat_prop2-proof5}
  \sum_{j=1}^{k-1} \l( (t_{k+1} - t_j)^{\alpha} - (t_k - t_j)^{\alpha} 
  \dfrac{t_{k+1}^{\alpha}}{t_{k}^{\alpha}}\r) (F_{j+1} - F_j) < 0.
\end{equation}
Similar to \eqref{eq:mat_prop2-proof4} we also have
\[
  \begin{aligned}
    0 = W_{k} - \dfrac{t_{k}^{\alpha}}{t_{k-1}^{\alpha}} W_{k-1}
    = \sum_{j=1}^{k-1} \l( (t_{k} - t_j)^{\alpha} - (t_{k-1} - t_j)^{\alpha} 
    \dfrac{t_{k}^{\alpha}}{t_{k-1}^{\alpha}}\r) (F_{j+1} - F_j).
  \end{aligned}
\]
Thanks to the inductive hypothesis, we know that $F_{j+1} - F_j < 0$ for $j = 2$ and $F_{j+1} - F_j > 0$ for $j \in \{3, \ldots, k-1\}$, 
Therefore using a similar argument used in the proof for \eqref{eq:mat_prop2-1}, to prove \eqref{eq:mat_prop2-proof5} we only need to show
\begin{equation}\label{eq:mat_prop2-proof6}
  \dfrac{(t_{k+1} - t_1)^{\alpha} - (t_k - t_1)^{\alpha} (t_{k+1}/t_{k})^{\alpha} }{ (t_{k} - t_1)^{\alpha} - (t_{k-1} - t_1)^{\alpha} (t_{k}/t_{k-1})^{\alpha} } > 
  \dfrac{(t_{k+1} - t_j)^{\alpha} - (t_k - t_j)^{\alpha} (t_{k+1}/t_{k})^{\alpha} }{ (t_{k} - t_j)^{\alpha} - (t_{k-1} - t_j)^{\alpha} (t_{k}/t_{k-1})^{\alpha} }, 
  \quad j \in \{2, \ldots, k-1\},
\end{equation}
which is similar to \eqref{eq:mat_prop2-proof2}. We rewrite the inequality above as
\[
  \dfrac{(1 - t_1/t_{k+1})^{\alpha} - (1 - t_1/t_{k})^{\alpha} }{ (1 - t_1/t_k)^{\alpha} - (1 - t_1/t_{k-1})^{\alpha} } > 
  \dfrac{(1 - t_j/t_{k+1})^{\alpha} - (1 - t_j/t_{k})^{\alpha} }{ (1 - t_j/t_k)^{\alpha} - (1 - t_j/t_{k-1})^{\alpha} }, 
  \quad j \in \{2, \ldots, k-1\},
\]
and define the function
\[
  h_1(x) =  \dfrac{(1 - x/t_{k+1})^{\alpha} - (1 - x/t_{k})^{\alpha} }{ (1 - x/t_k)^{\alpha} - (1 - x/t_{k-1})^{\alpha} },
\]
then it suffices to show that $h_1'(x) < 0$ for $0 < x < t_{k-1}$. Observing that
\[
  \begin{aligned}
    & \dfrac{\diff}{\diff x} \ln(h_1(x))  =
    -\dfrac{\alpha}{x} \bigg[ \dfrac{(x/t_{k+1})(1 - x/t_{k+1})^{\alpha-1} - (x/t_{k})(1 - x/t_{k})^{\alpha-1} }{ (1 - x/t_{k+1})^{\alpha} - (1 - x/t_{k})^{\alpha} } \\
    & \qquad \qquad \qquad \qquad \qquad -\dfrac{(x/t_{k})(1 - x/t_{k})^{\alpha-1} - (x/t_{k-1})(1 - x/t_{k-1})^{\alpha-1} }{ (1 - x/t_{k})^{\alpha} - (1 - x/t_{k-1})^{\alpha} } \bigg].
  \end{aligned}
\]
Letting $h_2(x) = (1-x)x^{\alpha-1}, h_3(x) = x^{\alpha}$, by Cauchy's mean-value theorem, there exists $\eta \in (1 - x/t_{k}, 1 - x/t_{k+1})$ and $\xi \in (1 - x/t_{k-1}, 1 - x/t_{k})$ such that
\[
  \dfrac{\diff}{\diff x}\l( \ln(h_1(x)) \r) = -\dfrac{\alpha}{x} 
  \l( \dfrac{h'_2(\eta)}{h'_3(\eta)} - \dfrac{h'_2(\xi)}{h'_3(\xi)} \r) = 
  -\dfrac{\alpha}{x} 
  \l( \l( \dfrac{\alpha-1}{\alpha \eta} -1 \r) - 
  \l( \dfrac{\alpha-1}{\alpha \xi} - 1\r) \r) < 0
\]
because $0<\xi<\eta$. This implies that $h_1'(x) < 0$ for $0 < x < t_{k-1}$ and finishes inductive step of the induction. Hence \eqref{eq:mat_prop2-2} is proved.
\end{proof}

\begin{Remark}[generalization]
Notice that, for a general kernel $g$, property \eqref{eq:mat_prop2-1} remains valid provided $G(t) = \ln(g(t))$ satisfies $G''(t) > 0$.
\end{Remark}

\subsection{A continuous interpolant}

Given a partition $\mP$, a sequence $\bW \in \mH^N$, and $W_0 \in \mH$, we defined the discrete Caputo derivative $\l( D^{\alpha}_\mP \bW \r)_n$ via \eqref{eq:discrete-Caputo}. Motivated by the Volterra type equation \eqref{eq:Volterra-eq} between a continuous function $w$ and its Caputo derivative $D^{\alpha}_c w$, it is possible, following \cite{LiLiuGradientFlow-19}, to define, over $\mP$, a natural continuous interpolant of $W_n$ by
\begin{equation}\label{eq:def-Uinterp}
  \widehat{W}_\mP(t) = W_0 + \frac{1}{\Gamma(\alpha)} \int_0^t (t-s)^{\alpha-1} \overline{V}_\mP(s) \diff s
\end{equation}
where $\overline{V}_\mP$ is defined by
\begin{equation}\label{eq:def-Uinterp-Ft}
  \overline{V}_\mP(t) = \l( D^{\alpha}_\mP \bW \r)_{n(t)}.
\end{equation}
By definition, we have that $\widehat{W}_\mP(t_n) = W_n$. Moreover,
\begin{equation}
  \begin{aligned}
    \widehat{W}_\mP(t) &= W_0 + \frac{1}{\Gamma(\alpha+1)} \sum_{j=1}^{n - 1} \l( (t-t_{j-1})^{\alpha} - (t-t_j)^{\alpha} \r) \l( D^{\alpha}_\mP \bW \r)_j
    + (t_{n} - t)^{\alpha} \l( D^{\alpha}_\mP \bW \r)_{n} \\
    &= \sum_{i=0}^{n(t)} W_i \varphi_{\mP,i}(t),
  \end{aligned}
\end{equation}
where we defined
\begin{equation}\label{eq:hat_function}
  \begin{aligned}
    \varphi_{\mP,0}(t) &= 1 + \frac{1}{\Gamma(\alpha+1)} \sum_{j=1}^{n(t) - 1} \l( (t-t_{j-1})^{\alpha} - (t-t_j)^{\alpha} \r) \bK^{-1}_{\mP,j0} + (t_{n} - t)^{\alpha} \bK^{-1}_{\mP,n0}, \\
    \varphi_{\mP,i}(t) &= \frac{1}{\Gamma(\alpha+1)} \sum_{j=i}^{n(t) - 1} \l( (t-t_{j-1})^{\alpha} - (t-t_j)^{\alpha} \r) \bK^{-1}_{\mP,ji} + (t_{n} - t)^{\alpha} \bK^{-1}_{\mP,ni}, \quad i \in \{1, \ldots, N\}.
  \end{aligned}
\end{equation}

\begin{figure}
  \begin{tabular}{ccc}
    $\alpha = 0.1$ & $\alpha = 0.5$ & $\alpha = 0.9$ \\
    \includegraphics[scale=0.3]{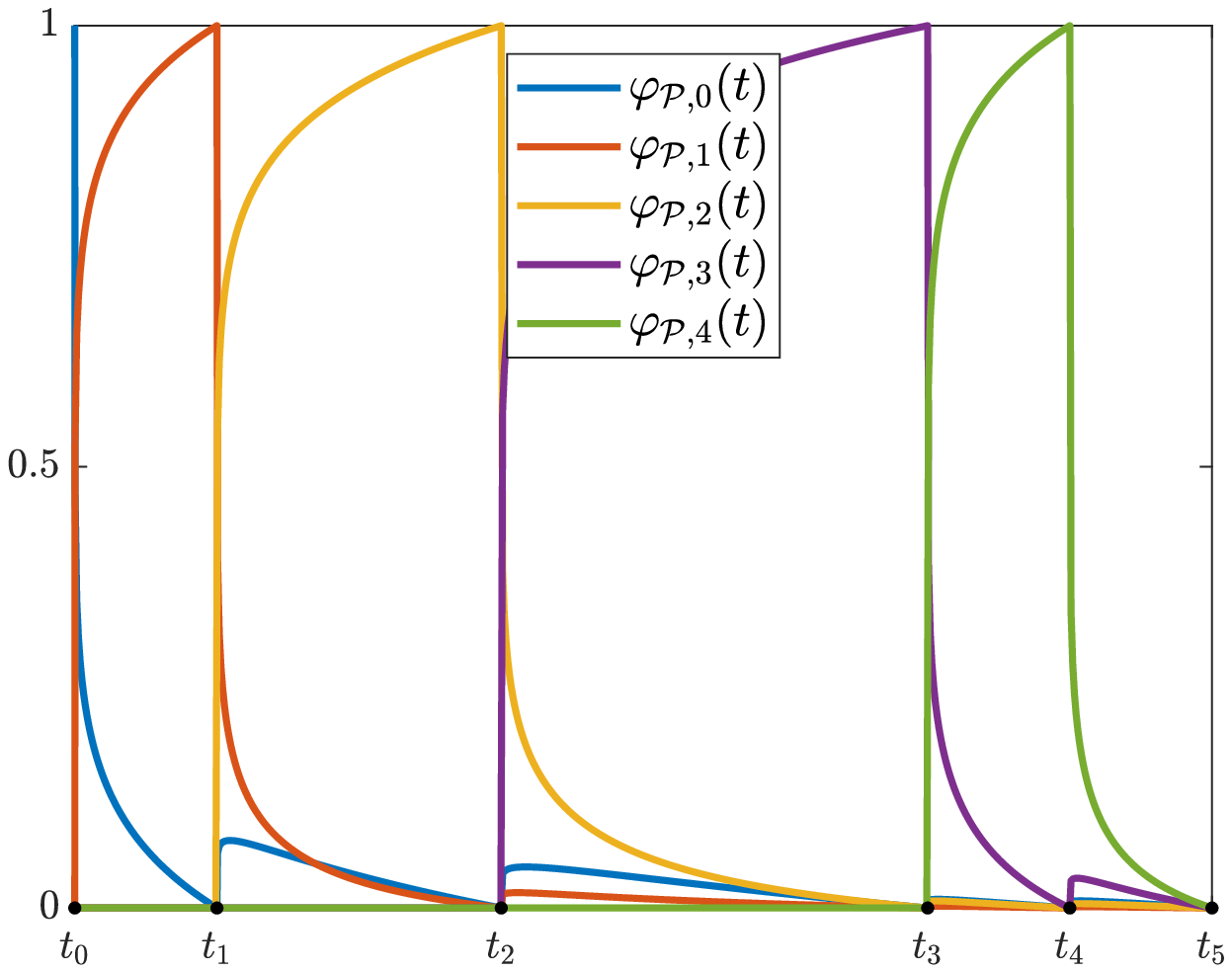} &
    \includegraphics[scale=0.3]{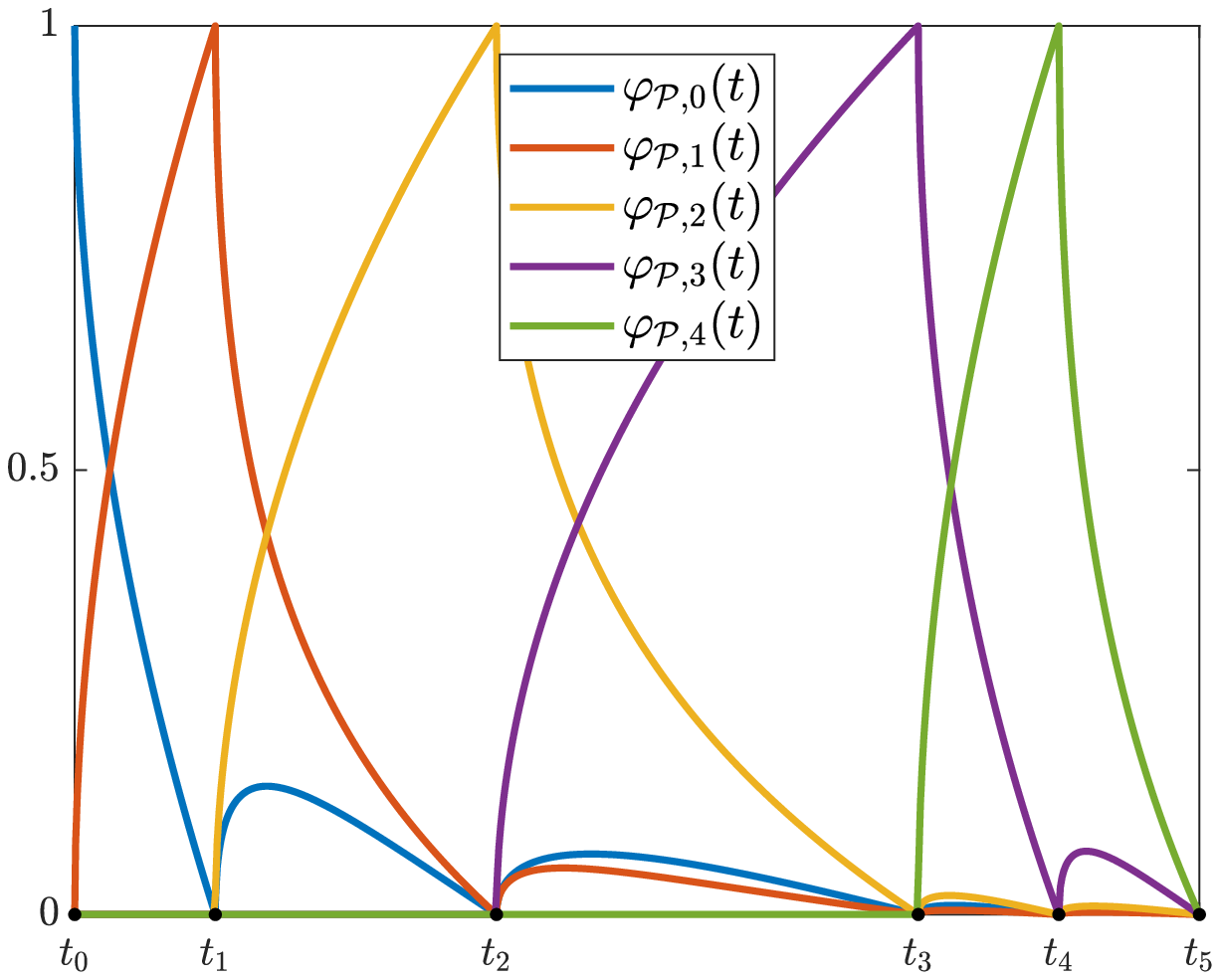} &
    \includegraphics[scale=0.3]{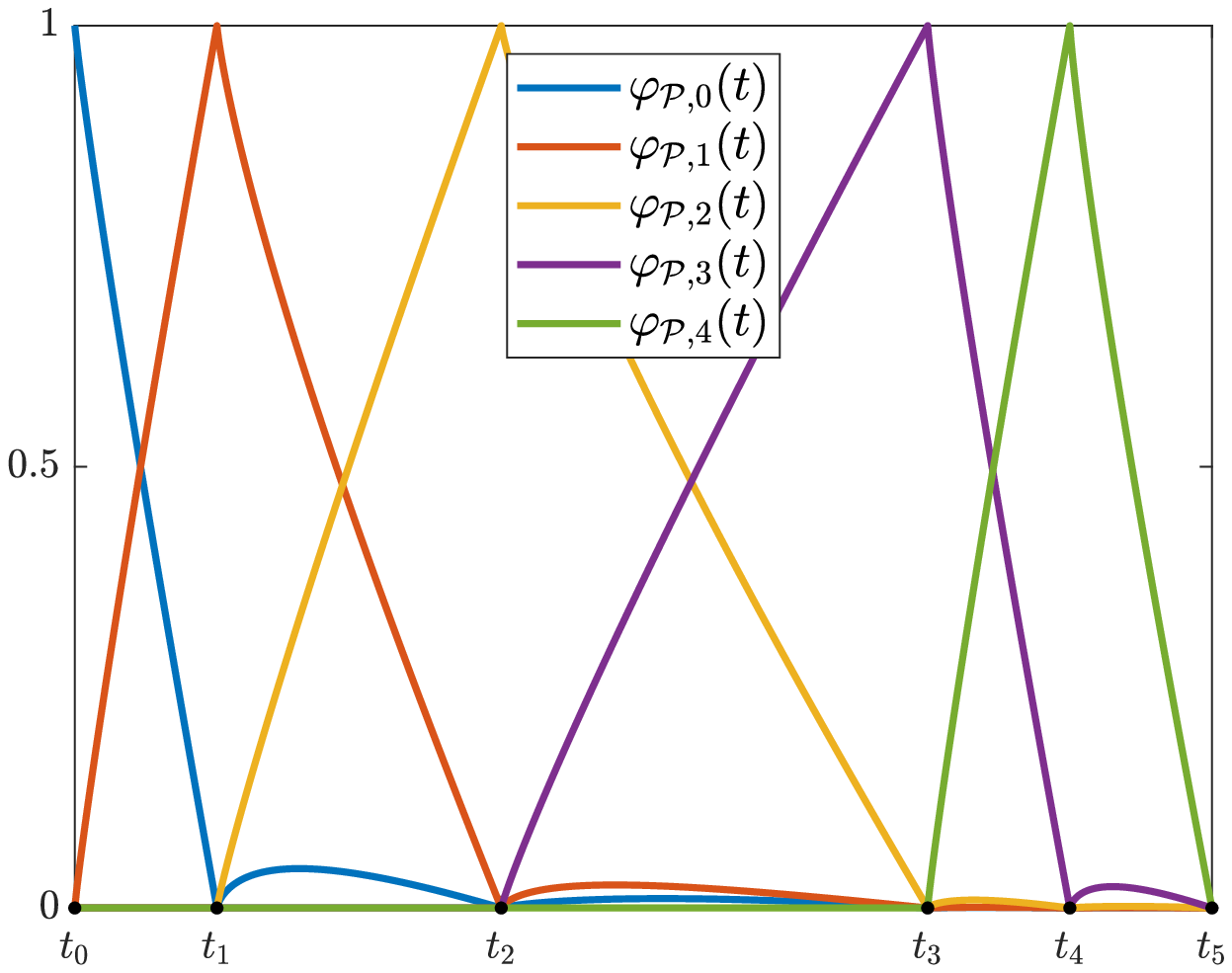}
  \end{tabular}

  \caption{Given a partition $\mP$, the figure shows the nonlocal basis functions $\{\varphi_{\mP,i}\}_{i=0}^N$ for different values of $\alpha$. Every function whose Caputo derivative is piecewise constant can be written as a linear combination of these functions. Notice that, for any partition point $\varphi_{\mP,i}(t_j) = \delta_{ij}$. In addition, Proposition~\ref{prop:interp-convex-comb} shows that these functions form a partition of unity.}\label{fig:illustration}
\end{figure}

The functions $\{\varphi_{\mP,i}\}_{i=0}^N$ play the role, in this context, of the standard ``hat'' basis functions used for piecewise linear interpolation over a partition $\mP$. Indeed, they are such that any function with piecewise constant (Caputo) derivative can be written as a linear combination of them. Figure~\ref{fig:illustration} illustrates the behavior of these functions. As expected, and in contrast to the hat basis functions, these functions are nonlocal, in the sense that they have global support. Something worth noticing is also that the figure seems to indicate that, as $\alpha \downarrow 0$, the functions resemble piecewise constants and, in contrast, when $\alpha \uparrow 1$ they tend to the classical hat basis functions.

An important feature of the hat basis functions is that they form a partition of unity. It is easy to check that, for any $t \in [0,T]$ we have $\sum_{i=0}^{n(t)} \varphi_{\mP,i}(t) = 1$. The following result shows that $\varphi_{\mP,i}(t) \ge 0$. Thus, for any $t \in [0,T]$, $\widehat{W}_\mP(t)$ is a convex combination of its nodal values $\{W_j\}_{j=0}^N$. This observation will be crucial to derive an a posteriori error estimate in \Cref{sec:post-error}.

\begin{Proposition}[positivity]
\label{prop:interp-convex-comb}
Let $\mP$ be a partition defined as in \eqref{eq:def-partition}. Let the functions $\{\varphi_{\mP,i} \}_{i=0}^N$ be defined as in \eqref{eq:hat_function}. Then, for any $i \in \{0, \ldots, N\}$ and $t \in[0,T]$, we have $\varphi_{\mP,i}(t) \ge 0$. In addition, for $t \notin \mP$ and $i \in \{0, \ldots, n(t) \}$ we have $\varphi_{\mP,i}(t) > 0$.
\end{Proposition}
\begin{proof}
By definition, for $t = t_n$, we have $\varphi_{\mP,n}(t_n) = 1$ and $\varphi_{\mP,i}(t_n) = 0$ for any $i \neq n$. Also, for  $i> n(t)$, we see that $\varphi_{\mP,i}(t) = 0$, and hence it only remains to show that $\varphi_{\mP,i}(t) > 0$ for $i \leq n(t)$. To show this, consider $W_i = 1$ and $W_j = 0$ for $j \neq i$, a piecewise constant $\overline{V}_\mP$ and its interpolation $\widehat{W}_\mP$ defined in \eqref{eq:def-Uinterp} and \eqref{eq:def-Uinterp-Ft}. Then our goal is to show that $\widehat{W}_\mP(t) > 0$.

If $i = n(t) > 0$, then it is easy to check by definition that $\l( D^{\alpha}_\mP \bW \r)_n > 0$ and $\l( D^{\alpha}_\mP \bW \r)_j = 0$ for $j \in \{1, \ldots, i-1\}$. Therefore we obtain
\[
  \widehat{W}_\mP(t) = \frac{1}{\Gamma(\alpha)} \int_0^t (t-s)^{\alpha-1} \overline{V}(s) \diff s 
  = \frac{(t_n - t)^{\alpha}}{\Gamma(\alpha+1)} \l( D^{\alpha}_\mP \bW \r)_n > 0.
\]

If $i < n(t)$, the proof is not that straightforward. The trick is to insert the time $t$, which is not on the partition $\mP$, to get a new partition $\mP' = \mP \cup \{t\}$ and then apply Propositions \ref{prop:mat_prop1} and \ref{prop:mat_prop2} in an appropriate way. Let us now work out the details. Let $\mP' = \{t_k'\}_{k=0}^{N+1}$ and notice that $t_{n(t)}' = t, t_{n(t)+1}' = t_{n(t)}$. On the basis of this partition we define the vector $\bW'\in \mH^{N+1}$ via $W'_j = \widehat{W}_\mP(t_j')$, then since $\overline{V}_\mP$ is constant on $(t_{n(t)-1}', t_{n(t)+1}'] = (t_{n(t)-1}, t_{n(t)}]$, we have
\[
  \l( D^{\alpha}_{\mP'} \bW' \r)_{n(t)} = \l( D^{\alpha}_{\mP'} \bW' \r)_{n(t)+1}.
\]
Since the only possible nonzero components of $\bW'$ are $W'_i = W_i = 1$ and $W'_{n(t)} = \widehat{W}_\mP(t)$, therefore we deduce from the equality above that
\begin{align*}
  \bK^{-1}_{\mP', n(t)i} W'_i + \bK^{-1}_{\mP', n(t)n(t)} W'_{n(t)}  &= \l( D^{\alpha}_{\mP'} \bW' \r)_{n(t)} = \l( D^{\alpha}_{\mP'} \bW' \r)_{n(t)+1} \\ &= \bK^{-1}_{\mP', n(t)+1, i} W'_i + \bK^{-1}_{\mP', n(t)+1, n(t)} W'_{n(t)},
\end{align*}
which can be rearranged as
\[
  \bK^{-1}_{\mP', n(t)+1, i} - \bK^{-1}_{\mP', n(t)i} = \widehat{W}_\mP(t) \l( \bK^{-1}_{\mP', n(t)n(t)} - \bK^{-1}_{\mP', n(t)+1, n(t)} \r).
\]
From \Cref{prop:mat_prop2} we see that $\bK^{-1}_{\mP', n(t)+1, i} - \bK^{-1}_{\mP', n(t)i} > 0$ and from \Cref{prop:mat_prop1}  we see that $\bK^{-1}_{\mP', n(t)n(t)} - \bK^{-1}_{\mP', n(t)+1, n(t)} > 0$ as a consequence of $\bK^{-1}_{\mP', n(t)n(t)} > 0$ and $\bK^{-1}_{\mP', n(t)+1, n(t)} < 0$. This leads to the fact that $\widehat{W}_\mP(t) > 0$ and finishes our proof.
\end{proof}

\section{Time fractional gradient flow: Theory}\label{sec:theory}

We have now set the stage for the study of time fractional gradient flows, which were formally described in \eqref{eq:theGradFlow}. Throughout the remaining of our discussion we shall assume that the initial condition satisfies $u_0 \in D(\Phi)$ and that $f \in L^2_\alpha(0,T;\mH)$. We begin by commenting that the case $f=0$ was already studied in \cite[Section 5]{LiLiuGradientFlow-19} where they studied so-called \emph{strong solutions}, see \cite[Definition 5.4]{LiLiuGradientFlow-19}. Here we trivially extend their definition to the case $f \neq 0$.

\begin{Definition}[strong solution]\label{def:strong}
A function $u \in L_{loc}^1([0,T);\mH)$ is a strong solution to \eqref{eq:theGradFlow} if
\begin{enumerate}[(i)]
  \item (Initial condition)
  \[
    \lim_{t \downarrow 0} \fint_0^t \Vert u(s) - u_0 \Vert \diff s = 0.
  \]
 
  \item (Regularity) $D_c^{\alpha} u(t) \in L_{loc}^1([0,T);\mH)$.

  \item (Evolution) For almost every $t \in [0,T)$, we have $f(t)-D_c^{\alpha} u(t) \in \partial \Phi(u(t))$.
\end{enumerate}
\end{Definition}

\subsection{Energy solutions}

Since $\mH$ is a Hilbert space, we will mimic the theory for classical gradient flows and introduce the notion of energy solutions for \eqref{eq:theGradFlow}. To motivate it, suppose that at some $t \in (0, T)$ 
\[
  f(t)-D_c^{\alpha} u(t) \in \partial \Phi(u(t)),
\]
then, by definition of the subdifferential, this is equivalent to the \emph{evolution variational inequality} (EVI)
\begin{equation}\label{eq:sol-ineq-@t}
\langle D_c^{\alpha} u(t), u(t) - w \rangle + \Phi(u(t)) - \Phi(w) \le \langle f(t), u(t) - w \rangle, 
\quad \forall w \in \mH.
\end{equation}

\begin{Definition}[energy solution]\label{def:energy-sol}
The function $u \in L^2(0,T;\mH)$ is an energy solution to \eqref{eq:theGradFlow} if
\begin{enumerate}[(i)]
  \item (Initial condition)
  \[
    \lim_{t \downarrow 0} \fint_0^t \Vert u(s) - u_0 \Vert^2 \diff s = 0.
  \]

  \item (Regularity) $D_c^{\alpha}u  \in L^2(0,T;\mH)$.
  
  \item (EVI) For any $w \in L^2(0,T;\mH)$
  \begin{equation}\label{eq:def-energy-sol}
    \int_0^T \left[ \langle D_c^{\alpha} u(t), u(t) - w(t) \rangle + \Phi(u(t)) - \Phi(w(t)) \right] \diff t \le \int_0^T \langle f(t), u(t) - w(t) \rangle \diff t .
  \end{equation}
\end{enumerate}
\end{Definition}

Notice that, provided $u_0 \in D(\Phi)$ we can set $w(t) = u_0$ in \eqref{eq:def-energy-sol} and obtain that $\int_0^T \Phi(u(t)) \diff t<\infty$, which motivates the name for this notion of solution. In addition, as the following result shows, any energy solution is a strong solution.

\begin{Proposition}[energy vs.~strong]
An energy solution of \eqref{eq:theGradFlow} is also a strong solution.
\end{Proposition}
\begin{proof}
Evidently, it suffices to prove that that $f(t)-D_c^{\alpha} u(t) \in \partial \Phi(u(t))$ for almost every $t \in (0,T)$. Let $w_0 \in \mH$, $t_0 \in (0,T)$, and choose $h>0$ sufficiently small so that $(t_0-h,t_0+h) \subset (0,T)$. Define
\[
  w(t) = u(t) - \chi_{(t_0-h,t_0+h)} (u(t)-w_0) \in L^2(0,T;\mH)
\]
where by $\chi_S$ we denote the characteristic function of the set $S$. This choice of test function on \eqref{eq:def-energy-sol} gives
\[
  \fint_{t_0-h}^{t_0+h} \langle D_c^\alpha u(t) - f(t), u(t)-w_0 \rangle \diff t + \fint_{t_0-h}^{t_0+h}\left( \Phi(u(t)) - \Phi(w_0) \right) \diff t \leq 0.
\]
The assumptions of an energy solution guarantee that all terms inside the integrals belong to $L^1(0,T;\mR)$ so that for almost every $t_0$ we have, as $h \downarrow 0$, that
\[
  \langle D_c^\alpha u(t_0) - f(t_0), w_0 \rangle + \Phi(u(t_0)) - \Phi(w_0) \leq 0,
\]
which is \eqref{eq:sol-ineq-@t} and, as we intended to show, is equivalent to the claim.
\end{proof}

\begin{Remark}[coercivity]
By introducing the coercivity modulus of Definition~\ref{def:sigma-rho} one realizes that an energy solution $u$ satisfies, instead of \eqref{eq:sol-ineq-@t} and \eqref{eq:def-energy-sol}, the stronger inequalities
\begin{equation}\label{eq:sol-ineq-@t-strong}
  \langle D_c^{\alpha} u(t), u(t) - w \rangle + \Phi(u(t)) - \Phi(w) + \sigma(u(t); w) \le \langle f(t), u(t) - w \rangle, 
\quad \forall w \in \mH,
\end{equation}
and, for any $w \in L^2(0,T;\mH)$,
\begin{equation}\label{eq:def-energy-sol-strong}
  \int_0^T \left[ \langle D_c^{\alpha} u(t), u(t) - w(t) \rangle + \Phi(u(t)) - \Phi(w(t)) + \sigma(u(t); w(t)) \right] \diff t \le \int_0^T \langle f(t), u(t) - w(t) \rangle \diff t .
\end{equation}
\end{Remark}

\subsection{Existence and uniqueness}

In this section, we will prove the following theorem on the existence and uniqueness of energy solutions to \eqref{eq:theGradFlow} in the sense of \Cref{def:energy-sol}. The main result that we will prove reads as follows.

\begin{Theorem}[well posedness]\label{thm:exist-uniq-frac-GF}
Assume that the energy $\Phi$ is convex, l.s.c., and with nonempty effective domain. Let $u_0 \in D(\Phi)$ and $f \in L^2_\alpha(0,T;\mH)$. In this setting, the fractional gradient flow problem \eqref{eq:theGradFlow} has a unique energy solution $u$, in the sense of \Cref{def:energy-sol}. For almost every $t \in (0,T)$, the solution $u$ satisfies that $f(t)-D_c^{\alpha} u(t) \in \partial \Phi(u(t))$ and for any $t \in [0,T]$ we have
\begin{equation}
  u(t) = u_0 + \frac1{\Gamma(\alpha)}\int_0^t (t-s)^{\alpha-1} D_c^{\alpha} u(s) \diff s.
\end{equation}
In addition, 
% $D_c^{\alpha}u \in L^2_{\alpha}(0, T; \mH)$ and 
$u \in C^{0, \alpha/2}([0,T];\mH)$ with modulus of continuity
\begin{equation}\label{eq:u-hold-cont}
\Vert u(t_2) - u(t_1) \Vert \le C|t_2 - t_1|^{\alpha/2} \l( \Vert f \Vert_{L^2_{\alpha}(0, T; \mH)}^2 + \Phi(u_0) - \Phi_{\inf} \r)^{1/2}, \quad \forall t_1, t_2, \in [0 ,T].
\end{equation}
where the constant $C$ depends only on $\alpha$. 
\end{Theorem}

We point out that our assumptions are weaker than those in \cite[Theorem 5.10]{LiLiuGradientFlow-19}. First, we allow for a nonzero right hand side. In addition, we do not require \cite[Assumption 5.9]{LiLiuGradientFlow-19}, which is a sort of weak-strong continuity of subdifferentials.

The remainder of this section will be dedicated to the proof of \Cref{thm:exist-uniq-frac-GF}. To accomplish this, we follow a similar approach to \cite[Section 5]{LiLiuGradientFlow-19}. To show existence of solutions, we consider a sort of fractional minimizing movements scheme. We introduce a partition $\mP$ with maximal time step $\tau$ and compute the sequence $\bU=\{U_n\}_{n=0}^N \subset \mH$ as follows. Assume $U_0 \in D(\Phi)$ is given, the $n$--th iterate, for $n \in \{1, \ldots, N\}$, is defined recursively via
\begin{equation}\label{eq:discrete-frac-GF}
  F_n -\l( D_\mP^{\alpha} \bU \r)_n \in \partial \Phi(U_n),
\end{equation}
where 
\begin{equation}\label{eq:average-Fn}
  F_n = \fint_{t_{n-1}}^{t_n} f(t) \diff t.
\end{equation}
We will usually choose $U_0 = u_0$, but other choices of $U_0 \in D(\Phi)$ are also allowed.

From the approximation scheme \eqref{eq:discrete-frac-GF} and the expression of the discrete Caputo derivative $\l( D_\mP^{\alpha} \bU \r)_n$ given in \eqref{eq:discrete-Caputo}, it is clear that
\begin{equation}\label{eq:argmin-Un}
  U_n = \argmin_{w \in \mH} \l( 
  \Phi(w) - \langle F_n, w \rangle - \frac12 \sum_{i=0}^{n-1} 
  \bK^{-1}_{\mP,ni} \Vert w - U_i \Vert^2
\r).
\end{equation}
Thanks to \Cref{prop:mat_prop1}, for $i=0, \ldots, n-1$, we have that $\bK^{-1}_{\mP,ni} < 0$ and as a consequence the functional on the right hand side of \eqref{eq:argmin-Un} is uniformly convex.
Combining with the fact that $\Phi$ is lower semicontinuous, the functional on the right hand side has a unique minimizer, and hence $U_n$ is well-defined.

Now, in order to define a continuous in time function from $\bU$, we use the interpolation introduced in \eqref{eq:def-Uinterp}. Let $\overline{V}_\mP(t) = \l( D_\mP^{\alpha} \bU \r)_{n(t)}$. Then we have
\begin{equation}\label{eq:def-Uhat-V}
  \widehat{U}_\mP(t) = U_0 + \frac{1}{\Gamma(\alpha)} \int_0^t (t-s)^{\alpha-1} \overline{V}_\mP(s) \diff s.
\end{equation}
Recall that $\overline{F}_\mP$ can be defined from $\{F_n\}_{n=1}^N$ using \eqref{eq:def-pw-const} and that \Cref{lemma:pw-interp-L2alpha} showed that  $\overline{F}_\mP \in L^2_{\alpha}(0, T; \mH)$ with a norm bounded independently of $\mP$. We now obtain some suitable bounds for $\widehat{U}_\mP$ and $\overline{V}_\mP$.

\begin{Lemma}[a priori bounds]\label{lemma:bound-int-V}
Let $\mP$ be any partition. The functions $\widehat{U}_\mP$ and $\overline{V}_\mP$ satisfy
\begin{equation}\label{eq:V-stability}
\begin{aligned}
\sup_{t \in [0 , T]}\Phi(\widehat{U}_\mP(t))
\le \Phi(U_0) + \frac{1}{4\Gamma(\alpha)} \Vert \overline{F}_\mP \Vert_{L^2_{\alpha}(0, T; \mH)}^2
\le \Phi(U_0) + C \Vert f \Vert_{L^2_{\alpha}(0, T; \mH)}^2, \\
\Vert \overline{V}_\mP \Vert_{L^2_{\alpha}(0, T; \mH)}^2 = \sup_{t \in [0, T]} \int_0^t (t-s)^{\alpha-1} \Vert \overline{V}_\mP(s) \Vert^2 \diff s \le 
C \l( \Vert f \Vert_{L^2_{\alpha}(0, T; \mH)}^2 + \Phi(U_0) - \Phi_{\inf} \r),
\end{aligned}
\end{equation}
where the constant $C$ only depends on $\alpha$.
\end{Lemma}
\begin{proof}
Since $F_n -\l( D_\mP^{\alpha} \bU \r)_n \in \partial \Phi(U_n)$, one has
\[
\Phi(U_n) - \Phi(U_i) \le \langle F_n -\l( D_\mP^{\alpha} \bU \r)_n, U_n - U_i \rangle.
\]
Therefore noticing that $\bK^{-1}_{\mP,ni} < 0 $ for $i\in \{0, \ldots, n-1\}$, we get
\begin{equation}\label{eq:V-stability-proof1}
\begin{aligned}
\l( D_\mP^{\alpha} \Phi(\bU) \r)_n &= -\sum_{i=0}^{n-1} \bK^{-1}_{\mP,ni} \l( \Phi(U_n) - \Phi(U_i) \r) \le -\sum_{i=0}^{n-1} \bK^{-1}_{\mP,ni} \langle F_n -\l( D_\mP^{\alpha} \bU \r)_n, U_n - U_i \rangle \\
&= \langle F_n -\l( D_\mP^{\alpha} \bU \r)_n, \l( D_\mP^{\alpha} \bU \r)_n \rangle,
\end{aligned}
\end{equation}
where we denoted $\Phi(\bU) = \{\Phi(U_n)\}_{n=0}^N$.

We can now proceed to obtain the claimed estimates. To prove the first one, we use that
\[
\l( D_\mP^{\alpha} \Phi(\bU) \r)_n \le \langle F_n -\l( D_\mP^{\alpha} \bU \r)_n, \l( D_\mP^{\alpha} \bU \r)_n \rangle \le \frac14 \Vert F_n \Vert^2
\]
to obtain that for any $n$,
\[
\begin{aligned}
\Phi(U_n) &= \Phi(U_0) + \sum_{i=1}^{n} \bK_{\mP,ni} \l( D^{\alpha}_\mP \Phi(\bU) \r)_i
\le \Phi(U_0) + \frac14 \sum_{i=1}^{n} \bK_{\mP,ni} \Vert F_i \Vert^2 \\
& = \Phi(U_0) + \frac{1}{4\Gamma(\alpha)} \int_0^{t_n} (t_n - s)^{\alpha-1} \Vert \overline{F}_\mP(s) \Vert^2 \diff s  \le
\Phi(U_0) + C \Vert f \Vert_{L^2_{\alpha}(0, T; \mH)}^2,
\end{aligned} 
\]
where the constant $C$ depends only on $\alpha$. Now, since \Cref{prop:interp-convex-comb} has shown that $\widehat{U}_\mP$ is a convex combination of the values $U_n$, we have
\[
\Phi( \widehat{U}_\mP(t) ) = \Phi \l( \sum_{i=0}^N \varphi_{\mP,i}(t) U_i  \r)
\le \sum_{i=0}^N \varphi_{\mP,i}(t) \Phi \l( U_i \r) 
\le \max_n \Phi(U_n) \le \Phi(U_0) + C \Vert f \Vert_{L^2_{\alpha}(0, T; \mH)}^2,
\]
which finishes the proof of the first claim. 

We now proceed to prove the second claim. Using \eqref{eq:V-stability-proof1} we get
\[
\begin{aligned}
\Phi_{\inf} &\le \Phi( \widehat{U}_\mP(t) ) \leq
 \Phi(U_0) + \frac{1}{\Gamma(\alpha)} \int_0^{t} (t - s)^{\alpha-1} \langle \overline{F}_\mP(s) - \overline{V}_\mP(s), \overline{V}_\mP(s) \rangle \diff s \\
  &\le \Phi(U_0) + \frac{1}{\Gamma(\alpha)} \l( \int_0^{t} (t - s)^{\alpha-1} \Vert \overline{F}_\mP(s) \Vert^2 \diff s \r)^{1/2} \l( \int_0^{t} (t - s)^{\alpha-1} \Vert \overline{V}_\mP(s) \Vert^2 \diff s \r)^{1/2} \\
& - \frac{1}{\Gamma(\alpha)} \int_0^{t} (t - s)^{\alpha-1} \Vert \overline{V}_\mP(s) \Vert^2 \diff s,
\end{aligned} 
\]
for any $t \in [0,T]$.
This implies that
\[
  \int_0^{t} (t - s)^{\alpha-1} \Vert \overline{V}_\mP(s) \Vert^2 \diff s \leq \Vert \overline{F}_\mP \Vert_{L^2_{\alpha}(0, T; \mH)}^2 + 2\Gamma(\alpha)(\Phi(U_0) - \Phi_{\inf}),
\]
which, using \Cref{lemma:pw-interp-L2alpha}, implies the result.
\end{proof}

\begin{Remark}[the function $\widehat \Phi$]\label{remark:Phihat}
Notice that, during the course of the proof of the first estimate in \eqref{eq:V-stability} we also showed that, if we define $\widehat{\Phi}_\mP(t) = \sum_{i=0}^N \varphi_{\mP,i}(t) \Phi(U_i)$, then $\widehat{\Phi}(t)$ is the interpolation of $ \Phi_\mP(\bU)$ with piecewise constant Caputo derivative. Moreover,
\[
D_c^{\alpha} \widehat{\Phi}_\mP(t) \le \frac14 \l\Vert \overline{F}_\mP(t) \r\Vert^2.
\]
\end{Remark}

These estimates immediately yield a modulus of continuity estimate on the interpolant $\widehat{U}_\mP$ which is independent of the partition $\mP$.

\begin{Lemma}[H\"older continuity] \label{lemma:Hold_Uhat}
Let $\mP$ be any partition and $\bU \in \mH^{N}$ be the solution to \eqref{eq:discrete-frac-GF} associated to this partition. For $t_1, t_2 \in [0,T]$ the interpolant $\widehat{U}_\mP$, defined in \eqref{eq:def-Uinterp}, satisfies
\[
\Vert \widehat{U}(t_2) - \widehat{U}(t_1) \Vert 
% \le C \Vert D_c^{\alpha} \widehat{U} \Vert_{L^2_{\alpha}(0, T; \mH)}
\le C |t_2-t_1|^{\alpha/2} \l( \Vert f \Vert_{L^2_{\alpha}(0, T; \mH)}^2 + \Phi(U_0) - \Phi_{\inf} \r)^{1/2}
\]
where the constant $C$ depends only on $\alpha$.
\end{Lemma}
\begin{proof}
As proved in \cite[Lemma 5.8]{LiLiuGradientFlow-19}, $D_c^{\alpha}w \in L^2_{\alpha}(0, T; \mH)$ guarantees $w \in C^{0, \alpha/2}([0,T]; \mH)$. Therefore using $D_c^{\alpha} \widehat{U} = \overline{V}_\alpha \in L^2_{\alpha}(0, T; \mH)$ and the estimate from \Cref{lemma:bound-int-V}, we obtain the result.
\end{proof}

Next we control the difference between discrete solutions corresponding to different partitions.

\begin{Lemma}[equicontinuity]\label{lemma:a-priori-U12}
Let, for $i=1,2$, $\mP_i$ be partitions of $[0,T]$ with maximal step size $\tau_i$, respectively, and denote by $\bU^{(i)}$ the associated solutions to \eqref{eq:discrete-frac-GF}. Let $\widehat{U}_i$ be their interpolations, defined by \eqref{eq:def-Uhat-V}, and $\overline{U}_i$ be their piecewise constant interpolations as in \eqref{eq:def-pw-const}. Assuming that $U^{(i)}_0 = U_0$ we have
\begin{align}
 \l\Vert \widehat{U}_1 - \widehat{U}_2 \r\Vert_{L^\infty(0,T;\mH)}
\le C \l( \tau_{1}^{\alpha/2} + \tau_{2}^{\alpha/2} \r) \l( \Vert f \Vert_{L^2_{\alpha}(0, T; \mH)}^2 + \Phi(U_0) - \Phi_{\inf} \r)^{1/2} \label{eq:a-priori-U12}, \\
\sup_{t \in [0,T]} \int_0^t (t-s)^{\alpha-1} \rho(\overline{U}_1(s), \overline{U}_2(s)) \diff s
\le C \l( \tau_1^{\alpha} + \tau_2^{\alpha} \r) \l( \Vert f \Vert_{L^2_{\alpha}(0, T; \mH)}^2 + \Phi(U_0) - \Phi_{\inf} \r), \label{eq:a-priori-rho-U12}
\end{align}
where the constant $C$ only depends on $\alpha$.
\end{Lemma}
\begin{proof}
For almost  every $t\in [0,T]$, we have that
\begin{equation}
\label{eq:a-priori-U12-proof3}
  \l\langle D_c^\alpha ( \widehat{U}_1 - \widehat{U}_2 ), \widehat{U}_1 - \widehat{U}_2 \r\rangle = \mathrm{I} + \mathrm{II} +\mathrm{III},
\end{equation}
where
\begin{align*}
  \mathrm{I} &=  \l\langle (\overline{F}_2 - D_c^\alpha \widehat{U}_2) - (\overline{F}_1 - D_c^\alpha \widehat{U}_1), \overline{U}_1 - \overline{U}_2 \r\rangle \leq -\rho(\overline{U}_1,\overline{U}_2), \\
  \mathrm{II} &= \l\langle (\overline{F}_2 - D_c^\alpha \widehat{U}_2) - (\overline{F}_1 - D_c^\alpha \widehat{U}_1), (\widehat{U}_1-\overline{U}_1) - (\widehat{U}_2-\overline{U}_2) \r\rangle, \\
  \mathrm{III} &= \l\langle \overline{F}_1 - \overline{F}_2, \widehat{U}_1 - \widehat{U}_2 \r\rangle,
\end{align*}
where to bound $\mathrm{I}$ we used that $\overline{F}_i(t) - D_c^{\alpha}\widehat{U}_i(t) \in \partial \Phi(\overline{U}_i(t))$ and \Cref{def:sigma-rho}. Define now
\[
\begin{aligned}
G(t) &= \frac{1}{\Gamma(\alpha)}\int_0^t (t-s)^{\alpha-1} \l( \overline{F}_1(s)-\overline{F}_2(s) \r) \diff s \\
&= \frac{1}{\Gamma(\alpha)}\int_0^t (t-s)^{\alpha-1} \l( \overline{F}_1(s)-f(s) \r) \diff s - \frac{1}{\Gamma(\alpha)}\int_0^t (t-s)^{\alpha-1} \l( \overline{F}_2(s)-f(s) \r) \diff s,
\end{aligned}
\] 
so that $D_c^{\alpha} G(t) = \overline{F}_1(t)-\overline{F}_2(t)$ and by \eqref{eq:alpha-int-f-Fbar} of \Cref{lemma:alpha-int-f-Fbar} one further has
\begin{equation}\label{eq:a-priori-U12-proof4}
\Vert G \Vert_{L^\infty(0,T;\mH)} \le
C \l( \tau_1^{\alpha/2} + \tau_2^{\alpha/2} \r) \Vert f \Vert_{L^2_{\alpha}(0, T; \mH)},
\end{equation}
where $C$ is a constant that depends only on $\alpha$. Using these estimates, from \eqref{eq:a-priori-U12-proof3} we deduce that
\begin{equation}
\label{eq:a-priori-U12-proof5}
  \l\langle D_c^\alpha ( \widehat{U}_1 - \widehat{U}_2 -G), \widehat{U}_1 - \widehat{U}_2 -G\r\rangle + \rho(\overline{U}_1 , \overline{U}_2) \leq  \mathrm{II} -\l\langle D_c^\alpha ( \widehat{U}_1 - \widehat{U}_2 -G), G\r\rangle.
\end{equation}
Set $w = \widehat{U}_1 - \widehat{U}_2 -G$. By \eqref{eq:Caputo-ineq-norm-sq} we have that
\[
  \frac12 D_c^\alpha \| w(t) \|^2 + \rho( \overline{U}_1, \overline{U}_2 ) \leq \mathrm{II}-\l\langle D_c^\alpha w, G\r\rangle ,
\]
and, using \eqref{eq:Volterra-eq} and \eqref{eq:a-priori-U12-proof4}, we then conclude
\begin{multline*}
  \frac12 \| \widehat{U}_1(t) - \widehat{U}_2(t) \|^2 + \frac1{\Gamma(\alpha)} \int_0^t (t-s)^{\alpha-1} \rho(\overline{U}_2(s), \overline{U}_2(s)) \diff s \leq  \\
  \frac{2}{\Gamma(\alpha)} \int_0^t (t-s)^{\alpha-1} \left( \mathrm{II}(s) -\l\langle D_c^\alpha w(s), G(s)\r\rangle  \right) \diff s
  + C \l( \tau_1^{\alpha/2} + \tau_2^{\alpha/2} \r) \Vert f \Vert_{L^2_{\alpha}(0, T; \mH)}.
\end{multline*}
It remains then to estimate the fractional integral on the right hand side. We estimate each term separately.

First, owing to \Cref{lemma:pw-interp-L2alpha} and \Cref{lemma:bound-int-V} we have, for $i=1,2$, that
\[
\l\Vert \overline{F}_i - D_c^{\alpha}\widehat{U}_i \r\Vert_{L^2_{\alpha}(0, T; \mH)} \le C \l( \Vert f \Vert_{L^2_{\alpha}(0, T; \mH)}^2 + \Phi(U_0) - \Phi_{\inf} \r)^{1/2},
\]
Therefore using the Cauchy-Schwarz inequality, for any $t \in [0,T]$, we have
\[
\begin{aligned}
\int_0^t (t-s)^{\alpha-1} |\mathrm{II}(s)| \diff s &\le C \l( \Vert f \Vert_{L^2_{\alpha}(0, T; \mH)}^2 + \Phi(U_0) - \Phi_{\inf} \r)^{1/2} \sum_{i=1}^2 \l\Vert \widehat{U}_i -\overline{U}_i \r\Vert_{L^2_{\alpha}(0, T; \mH)} .
\end{aligned}
\]
Recalling that $\overline{U}_i(t) = \widehat{U}_i(\tceil_i)$ we can invoke \Cref{lemma:alpha-int-w} and, again, \Cref{lemma:bound-int-V} to arrive at
\[
  \int_0^t (t-s)^{\alpha-1} |\mathrm{II}(s)| \diff s \leq C (\tau_1^{\alpha/2} + \tau_2^{\alpha/2} )\l( \Vert f \Vert_{L^2_{\alpha}(0, T; \mH)}^2 + \Phi(U_0) - \Phi_{\inf} \r).
\]

Finally, for the remaining term, we use the Cauchy-Schwarz inequality and get
\[
\begin{aligned}
\int_0^t (t-s)^{\alpha-1} \l| \l\langle D_c^\alpha w, G\r\rangle(s) \r| \diff s 
&\le \l( \int_0^t (t-s)^{\alpha-1} \l\Vert D_c^{\alpha} w(s)\r\Vert^2 \diff s \r)^{1/2} 
\l( \int_0^t (t-s)^{\alpha-1} \l\Vert G(s) \r\Vert^2 \diff s \r)^{1/2} \\
&\le \l\Vert D_c^{\alpha} w \r\Vert_{L^2_{\alpha}(0, T; \mH)} \l\Vert G \r\Vert_{L^2_{\alpha}(0, T; \mH)}
\end{aligned}
\]
To estimate the norm of $G$ we apply \eqref{eq:double-alpha-int-f-Fbar} from \Cref{lemma:alpha-int-f-Fbar} with $\beta = \alpha$ to obtain
\[
\l\Vert G \r\Vert_{L^2_{\alpha}(0, T; \mH)} \le C \l( \tau_1^{\alpha}
+ \tau_1^{\alpha} \r) \Vert f \Vert_{L^2_{\alpha}(0, T; \mH)}.
\]
Furthermore, \Cref{lemma:pw-interp-L2alpha} and \Cref{lemma:bound-int-V} guarantee that
\[
\l\Vert D_c^{\alpha} w \r\Vert_{L^2_{\alpha}(0, T; \mH)} \le C \l( \Vert f \Vert_{L^2_{\alpha}(0, T; \mH)}^2 + \Phi(U_0) - \Phi_{\inf} \r)^{1/2}.
\]
Combining all estimates proves the desired result.
\end{proof} 

We are finally able to prove \Cref{thm:exist-uniq-frac-GF}. We will follow the same approach as in \cite[Theorem 5.10]{LiLiuGradientFlow-19}; we will pass to the limit $\tau_i \downarrow 0$ and study the limit of discrete solutions $\widehat{U}_i$.

\begin{proof}[Proof of \Cref{thm:exist-uniq-frac-GF}]
Let us first prove uniqueness of energy solutions. Suppose that we have two energy solutions $u_1, u_2$ to \eqref{eq:theGradFlow}. Let $t \in (0,T)$ be arbitrary and $h>0$ be sufficiently small so that $(t-h,t+h) \subset [0,T]$. Setting as test function, in the EVI that characterizes $u_1$, the function $w = u_1 - \chi_{(t-h,t+h)}(u_1-u_2)$ and vice versa, and adding the ensuing inequalities we obtain
\[
\int_{t-h}^{t+h} \langle D_c^{\alpha} u_1(s) - D_c^{\alpha} u_2(s), u_1(s) - u_2(s) \rangle \diff s \le 0,
\]
meaning that $\langle D_c^{\alpha} u_1(t) - D_c^{\alpha} u_2(t), u_1(t) - u_2(t) \rangle \le 0$ for almost every $t \in [0,T]$.

Define $d(t)= \Vert u_1(t) - u_2(t) \Vert^2$. Since $u_1, u_2 \in L^2(0,T;\mH)$ we clearly have $d \in L^1(0,T;\mR)$. Furthermore, 
\[
\fint_0^t |d(s)|\diff s 
 \le 2 \fint_0^t \left( \Vert (u_1(s) - u_0) \Vert^2 + \Vert (u_2(s) - u_0) \Vert^2 \right) \diff s \to 0,
\]
as $t \downarrow 0$, from \Cref{def:energy-sol}. Using \eqref{eq:Caputo-ineq-norm-sq} we then have
\[
D_c^{\alpha} d(t) \le
2\langle D_c^{\alpha} u_1(t) - D_c^{\alpha} u_2(t), u_1(t) - u_2(t) \rangle
\le 0
\]
in the distributional sense. Combining with the facts that $d \ge 0$ and $\fint_0^t | d(s) | \diff s \to 0$ we obtain, by \cite[Corollary 3.8]{LiLiuGeneralized-18}, $d(t) = 0$. This proves the uniqueness.

We now turn our attention to existence. Let $\{\mP_k\}_{k=1}^\infty$ be a sequence of partitions such that $\tau_k \downarrow 0$ as $k \to \infty$. We denote by $\bU^{(k)}$ the discrete solution, on partition $\mP_k$, given by \eqref{eq:discrete-frac-GF} with $U^{(k)}_0 = u_0$. The symbols $\widehat{U}_k$, $\overline{V}_k$ and $\overline{F}_k$ carry analogous meaning. Owing to \Cref{lemma:a-priori-U12} there exists $u \in C([0,T];\mH)$ such that $\widehat{U}_k$ converges to $u$ in $C([0,T];\mH)$.  

The embedding of \Cref{proposition:embeddingLpalpha} and an application of \Cref{lemma:bound-int-V} shows that there is a subsequence for which $\overline{V}_{k_j} \rightharpoonup v$ in $L^2(0, T; \mH)$ as $j \to \infty$. Moreover, we can again appeal to \Cref{lemma:bound-int-V} to see that, for every $t \in [0,T]$, the sequence
\[
  (t-\cdot)^{\frac{\alpha-1}{2}}\overline{V}_{k_j}(\cdot)
\]
is uniformly bounded in $L^2(0,t;\mH)$ so that by passing to a further, not retagged, subsequence
\begin{equation}\label{eq:exist-uniq-proof1}
(t-\cdot)^{\frac{\alpha-1}{2}}\overline{V}_{k_j}(\cdot) \rightharpoonup
(t-\cdot)^{\frac{\alpha-1}{2}} v(\cdot) \; \textrm{ in } L^2(0, t; \mH)
\end{equation}
for any $t \in [0,T]$. This, in addition, shows that $v \in L^2_{\alpha}(0, T; \mH)$ so that if we define
\begin{equation}
\label{eq:volterrafortilde}
\wt{u}(t) = u_0 + \frac{1}{\Gamma(\alpha)} \int_0^t (t-s)^{\alpha-1} v(s) \diff s
\end{equation}
then $D_c^{\alpha} \wt{u} = v$.

Recall that for any $j \in \mathbb{N}$ and any $t \in [0,T]$ we have that
\[
\widehat{U}_{k_j}(t) = u_0 + \frac{1}{\Gamma(\alpha)} \int_0^t (t-s)^{\alpha-1} \overline{V}_{k_j}(s) \diff s.
\]
Since, for an arbitrary $w \in \mH$ we have that $(t-\cdot)^{\frac{\alpha-1}{2}} w$ is in $L^2(0, t; \mH)$ , we can use \eqref{eq:exist-uniq-proof1} to obtain that
\[
\begin{aligned}
 \lim_{j \to \infty} \langle \widehat{U}_{k_j}(t), w \rangle
&= \lim_{j \to \infty} \l\langle u_0 + \frac{1}{\Gamma(\alpha)} \int_0^t (t-s)^{\alpha-1} \overline{V}_{k_j}(s) \diff s, w \r\rangle \\ 
&= \l\langle u_0 + \frac{1}{\Gamma(\alpha)}  \int_0^t (t-s)^{\alpha-1} v(s) \diff s,
w \r\rangle = \langle \wt{u}(t), w \rangle.
\end{aligned}
\]
The statement above holds for any $w \in \mH$ and all $t \in [0,T]$. Thus, 
\begin{equation}\label{eq:weak-convergence-Unkt}
  \widehat{U}_{k_j}(t) \rightharpoonup \wt{u}(t),
\end{equation}
in $\mH$. However, this implies that $\wt{u} = u$, as $\widehat{U}_{k_j}$ converges to $u$ in $C([0,T];\mH)$. Therefore $D_c^\alpha u= v\in L^2_{\alpha}(0, T; \mH)$ and, by \Cref{lemma:bound-int-V}, we have the estimate
\[
\Vert v \Vert_{L^2_{\alpha}(0, T; \mH)} \le C \l( \Vert f \Vert_{L^2_{\alpha}(0, T; \mH)}^2 + \Phi(U_0) - \Phi_{\inf} \r)^{1/2},
\]
for some constant $C$ depending on $\alpha$. As in the proof of \Cref{lemma:Hold_Uhat} this implies that \eqref{eq:u-hold-cont} holds. From this, we also see that the initial condition is attained in the required sense.

%Let us now show that the initial condition is attained in the required sense ({this is a direct consequence from the \Hold continuity in the previous line}). From \eqref{eq:volterrafortilde}
%\begin{align*}
%  \fint_0^t \| u(s) - u_0 \|^2 \diff s &\leq C \fint_0^t \l\Vert \int_0^s (s-r)^{\alpha-1} v(r) \diff r \r\Vert^2 \diff s \leq 
%  C \| v \|_{L^2_\alpha(0,T;\mH)}^2 \fint_0^t \int_0^s (s-r)^{\alpha-1} \diff r \diff s \\
%  &\leq Ct^\alpha \| v \|_{L^2_\alpha(0,T;\mH)}^2.
%\end{align*}

It remains to show that the EVI \eqref{eq:def-energy-sol} holds for $u$. From the construction of discrete solutions, one derives that for any $w \in L^2(0,T;\mH)$
\begin{equation}\label{eq:energy-sol-discrete}
\int_0^T \l( \Phi(\widehat{U}_{k_j}(t)) - \Phi(w(t)) \r) \diff t \le \int_0^T \langle \overline{F}_{k_j}(t) - \overline{V}_{k_j}(t), \widehat{U}_{k_j}(t) - w(t) \rangle \diff t.
\end{equation}
We will pass to the limit in this inequality. For the right hand side, it suffices to observe that $\widehat{U}_{k_j} \to u$ in $C([0,T];\mH)$, $\overline{V}_{k_j} \rightharpoonup v$ in $L^2(0, T; \mH)$ and $\overline{F}_{k_j} \to f$ in $L^2(0,T;\mH)$. Thus,
\[
\int_0^T \langle \overline{F}_{k_j}(t) - \overline{V}_{k_j}(t), \widehat{U}_{k_j}(t) - w(t) \rangle \diff t
\to \int_0^T \langle f(t) - v(t), u(t) - w(t) \rangle \diff t.
\]
For the left hand side, the uniform convergence of $\widehat{U}_{k_j}$ and the lower semicontinuity of $\Phi$, give
\[
\Phi(u(t)) \le \liminf_{j \to \infty} \Phi\l(\widehat{U}_{k_j}(t)\r),
\]
and hence
\[
\int_0^T \Phi(u(t)) - \Phi(w(t)) \diff t \le \int_0^T \langle f(t) - v(t), u(t) - w(t) \rangle \diff t.
\]
It remains to recall that $D_c^{\alpha}u = v \in L^2(0,T;\mH)$ to conclude that, according to \Cref{def:energy-sol}, $u$ is an energy solution.
\end{proof}

\begin{Remark}[other notion of solution]
The choice of $u \in L^2(0,T;\mH)$ and $D_c^{\alpha}u \in L^2(0,T;\mH)$ in \Cref{def:energy-sol} is to guarantee that \eqref{eq:def-energy-sol} makes sense. It is also necessary in the proof of uniqueness. However, other choices of spaces are also possible. For example, one could consider the following definition instead of \Cref{def:energy-sol}: $u \in L^{\infty}(0,T;\mH)$ is a solution to \eqref{eq:theGradFlow} if:
\begin{enumerate}[(i)]
  \item $\lim_{t \downarrow 0} \fint_0^t \Vert u(s) - u_0 \Vert \diff s = 0$;
  \item $D_c^{\alpha}u  \in L^1(0,T;\mH)$; and
  \item for any $w \in L^{\infty}(0,T;\mH)$,
	\begin{equation}\label{eq:def-energy-sol2}
	\int_0^T \l[ \langle D_c^{\alpha} u(t), u(t) - w(t) \rangle + \Phi(u(t)) - \Phi(w(t)) \r] \diff t \le \int_0^T \langle f(t), u(t) - w(t) \rangle \diff t .
	\end{equation}
\end{enumerate}
\Cref{thm:exist-uniq-frac-GF} also holds for this new definition. However, at least with our techniques, the requirements on the data $u_0 \in D(\Phi)$ and $f \in L^2_\alpha(0,T;\mH)$ do not change.
\end{Remark}

\section{Fractional gradient flows: Numerics}\label{sec:numerics}

Since the existence of an energy solution was proved by a rather constructive approach, namely a fractional minimizing movements scheme, it makes sense to provide error analyses for this scheme. We will provide an a priori error estimate which, in light of the smoothness $u \in C^{0, \alpha/2}([0,T];\mH)$ proved in \Cref{thm:exist-uniq-frac-GF}, is optimal. In addition, in the spirit of \cite{MR1737503} we will provide an a posteriori error analysis.

\subsection{A priori error analysis}

The a priori error estimate reads as follows. We comment that this result gives us a better rate compared to \cite[Theorem 5.10]{LiLiuGradientFlow-19}. 

\begin{Theorem}[a priori I]\label{thm:aprioriexistence}
Let $u$ be the energy solution of \eqref{eq:theGradFlow}. Given a partition $\mP$, of maximal step size $\tau$, let $\bU\in \mH^N$ be the discrete solution defined by \eqref{eq:discrete-frac-GF} starting from $U_0 \in D(\Phi)$. Let $\widehat{U}_\mP$ and $\overline{U}_\mP$ be defined as in \eqref{eq:def-Uhat-V} and \eqref{eq:def-pw-const}, respectively. Then we have,
\begin{align}
\l\Vert u - \widehat{U}_\mP \r\Vert_{L^\infty(0,T;\mH)}
\le \l\Vert u_0 - U_0 \r\Vert + C \tau^{\alpha/2} \l( \Vert f \Vert_{L^2_{\alpha}(0, T; \mH)}^2 + \Phi_0 - \Phi_{\inf} \r)^{1/2} \label{eq:a-priori-uU}, \\
\sup_{t \in [0,T]} \int_0^t (t-s)^{\alpha-1} \rho(u(s), \overline{U}_\mP(s)) \diff s
\le \l\Vert u_0 - U_0 \r\Vert^2 + C \tau^{\alpha} \l( \Vert f \Vert_{L^2_{\alpha}(0, T; \mH)}^2 + \Phi_0 - \Phi_{\inf} \r), \label{eq:a-priori-rho-uU}
\end{align}
where $\Phi_0 = \max\{\Phi(U_0), \Phi(u_0)\}$, and the constant $C$ depends only on $\alpha$.
\end{Theorem}
\begin{proof}
The proof can be obtained by following the same procedure employed in the proof of \Cref{lemma:a-priori-U12}. In the current situation, however, instead of comparing two discrete solutions we compare the exact and discrete ones. The only difference is that we allow $U_0 \neq u_0$ here, but this presents no essential difficulty. For brevity, we skip the details.
\end{proof}

\subsection{A posteriori error analysis} \label{sec:post-error}
Let us now provide an a posteriori error estimate between the discretization in \eqref{eq:discrete-frac-GF} and the solution of \eqref{eq:theGradFlow}. We will also show how, from this a posteriori error estimator, an a priori error estimate can be derived. 
% Further, we will discuss how we take advantages of this error estimate to choose time step adaptively in solving discrete problems. 
Let us first introduce the a posteriori error estimator.

\begin{Definition}[error estimator]
Let $\mP$ be a partition of $[0,T]$ as in \eqref{eq:def-partition}, and $\bU \in \mH^N$ denote the discrete solution given by \eqref{eq:discrete-frac-GF}. We define the error estimator function as
\begin{equation}\label{eq:post-error-mEt}
  \mE_\mP(t) =  \mE_{\mP,1}(t) + \mE_{\mP,2}(t),
\end{equation}
where 
\[
\mE_{\mP,1}(t) = \langle D_c^{\alpha} \widehat{U}_\mP(t) - \overline{F}_\mP(t), \widehat{U}_\mP(t) - \overline{U}_\mP(t) \rangle, \quad
\mE_{\mP,2}(t) = \Phi(\widehat{U}_\mP(t)) - \Phi(\overline{U}_\mP(t)).
\]
\end{Definition}

Notice that the quantity $\mE_\mP(t)$ is nonnegative because $\overline{F}_\mP(t) - D_c^{\alpha} \widehat{U}_\mP(t) = F_{n(t)} -\l( D^{\alpha}_\mP \bU \r)_{n(t)} \in \partial \Phi(U_{n(t)}) = \partial\Phi(\overline{U}_\mP(t))$. It is also, in principle, computable since it only depends on data, and the discrete solution $\bU$. It is then a suitable candidate for an a posteriori error estimator.

The derivation of an a posteriori error estimate begins with the observation that, for any $w \in \mH$, we have
\begin{equation}\label{eq:discrete-energy-ineq}
\begin{aligned}
& \langle D_c^{\alpha} \widehat{U}_\mP(t) - f(t), \widehat{U}_\mP(t) - w \rangle + \Phi(\widehat{U}_\mP(t)) - \Phi(w) \\
&= \mE_\mP(t) + \langle \overline{F}_\mP(t) - D_c^{\alpha} \widehat{U}_\mP(t), w - \overline{U}_\mP(t) \rangle
+ \Phi(\overline{U}_\mP(t)) - \Phi(w) + \langle f(t) - \overline{F}_\mP(t),  w- \widehat{U}_\mP(t) \rangle \\
&\le \mE_\mP(t) + \langle f(t) - \overline{F}_\mP(t),  w- \widehat{U}_\mP(t) \rangle - \sigma(\overline{U}_\mP(t); w).
\end{aligned}
\end{equation}
In other words, the function $\widehat{U}_\mP$ solves an EVI similar to \eqref{eq:sol-ineq-@t-strong} but with additional terms on the right hand side. We can then compare the EVIs by a now standard approach, that is, set $w = u(t)$ in \eqref{eq:discrete-energy-ineq} and $w=\widehat{U}_\mP(t)$ in \eqref{eq:sol-ineq-@t-strong}, respectively, to see that
\begin{multline}\label{eq:discrete-error-ineq1}
\l\langle D_c^{\alpha} \l( \widehat{U}_\mP-u\r)(t), \widehat{U}_\mP(t) - u(t) \r\rangle + \sigma(\overline{U}_\mP(t); u(t)) + \sigma(u(t); \widehat{U}_\mP(t))
 \le \\ \mE_\mP(t) + \langle f(t) - \overline{F}_\mP(t),  u(t) - \widehat{U}_\mP(t) \rangle
\end{multline}
for almost every $t \in [0,T]$. Consider the following notions of error:
\begin{equation}\label{eq:def-Error}
\begin{aligned}
&E = \l( \sup_{t \in [0,T]} \l\{E^2_{\mH}(t) + E^2_{\sigma}(t) \r\} \r)^{1/2},
\quad E_{\mH}(t) = \Vert u(t) - \widehat{U}_\mP(t) \Vert, \\
&E_{\sigma}(t) = \l( \frac{2}{\Gamma(\alpha)} \int_0^t (t-s)^{\alpha-1} \l[\sigma(u(s); \widehat{U}_\mP(s)) + \sigma(\overline{U}_\mP(s); u(s))\r] \diff s \r)^{1/2}.
\end{aligned}
\end{equation}
We have the following error estimate for $E$.

\begin{Theorem}[a posteriori]\label{thm:post-error}
Let $u$ be the energy solution of \eqref{eq:theGradFlow}. Let $\mP$ be a partition of $[0,T]$ defined as in \eqref{eq:def-partition} and let $\bU \in \mH^N$ be the discrete solution given by \eqref{eq:discrete-frac-GF} starting from $U_0 \in D(\Phi)$. Let $E$ and $\mE_\mP$ be defined in \eqref{eq:def-Error} and \eqref{eq:post-error-mEt}, respectively, The following a posteriori error estimate holds
\begin{equation}\label{eq:post-error}
E \le \l( \Vert u_0 - U_0\Vert^2 + \frac{2}{\Gamma(\alpha)} \Vert \mE_\mP \Vert_{L^1_{\alpha}(0,T;\mH) } \r)^{1/2} \!+
\frac{2}{\Gamma(\alpha)} \Vert f - \overline{F}_\mP \Vert_{L^1_{\alpha}(0,T;\mH)}.
\end{equation}
\end{Theorem}
\begin{proof}
From \eqref{eq:Caputo-ineq-norm-sq} we infer
\[
\begin{aligned}
\frac12 D_c^{\alpha} \Vert \widehat{U}_\mP-u \Vert^2(t)  &\le \l\langle D_c^{\alpha} \l( \widehat{U}_\mP-u\r)(t), \widehat{U}_\mP(t) - u(t) \r\rangle \\
&\le \mE_\mP(t) + \langle f(t) - \overline{F}_\mP(t),  u(t) - \widehat{U}_\mP(t) \rangle - \sigma(\overline{U}_\mP(t); u(t)) - \sigma(u(t); \widehat{U}_\mP(t)).
\end{aligned}
\]
The claimed a posteriori error estimate \eqref{eq:post-error} follows from \Cref{lemma:frac-Gronwall-diff-ineq} by setting
\begin{align*}
  \lambda &= 0,  &a(t) = \Vert (\widehat{U}_\mP-u)(t) \Vert, & &b(t)= 2 \l( \sigma(\overline{U}_\mP(t); u(t)) + \sigma(u(t); \widehat{U}_\mP(t)) \r), \\
  c(t)& = 2 \mE_\mP(t), & d(t)= \Vert (f - \overline{F}_\mP)(t) \Vert.
\end{align*}
\end{proof}

% {It would add unnecessary difficulties to allow $a$ in \Cref{lemma:frac-Gronwall-diff-ineq} to be discontinuous, so I simply assume $a$ is continuous. Also, I plan to delete the comments below (from early versions).
% 
% Before \cite[Sec 3.2]{MR1737503}, they mention one example to take advantages of the coercivity $\sigma(\cdot; \cdot)$ and obtain exponential decay. If we want to do something similar, we may need to construct a barrier. 
% 
% (I guess maybe we should consider putting $\Gamma(\alpha)$ in the definition of space $L^p_{\alpha}$)
% A more careful estimate like \cite[Thm 3.2, Lemma 3.7]{MR1737503} should give us better constant. This requires constructing $r(t)$ such that
% \[
% D_c^{\alpha} r(t) \ge 2 \mE(t) + 2 \Vert f(t) - \overline{F}(t) \Vert \sqrt{r(t)},
% \quad r(0) = \Vert u_0 - U_0\Vert^2.
% \]
% I didn't think much about this. Do we need to argue in this way?}

\subsection{Rate of convergence}

Although we have already established an optimal a priori rate of convergence for our scheme in \Cref{thm:aprioriexistence}, in this section we study the sharpness of the a posteriori error estimator $\mE_\mP$ by obtaining the same convergence rates through it. We comment that neither in \Cref{thm:aprioriexistence} nor in our discussion here, we require any relation between time steps. We will also consider some cases when the rate of convergence can be improved.

\subsubsection{Rate of convergence for energy solutions}
Let us now use the estimator $\mE_\mP$ to derive a convergence rate or order $\mO(\tau^{\alpha/2})$ for the error $E$, defined in \eqref{eq:def-Error}, when $f \in L^2_{\alpha}(0, T; \mH)$. Notice that such regularity a priori does not give any order of convergence for $\Vert f - \overline{F}_\mP \Vert_{L^1_{\alpha}(0,T;\mH)}$ in \eqref{eq:post-error}. Observe also that the rate that we obtain is consistent with classical gradient flow theories, where an order $\mO(\tau^{1/2})$ is proved provided that $u_0 \in D(\Phi)$ and $f \in L^2(0,T;\mH)$; see \cite[Sec 3.2]{MR1737503}.

%The strategy will be to bound each of the terms in \eqref{eq:post-error}. 
We first bound $\Vert \mE_\mP \Vert_{L^1_{\alpha}(0,T;\mH)}$.

\begin{Theorem}[bound on $\Vert \mE_\mP \Vert_{L^1_{\alpha}(0,T;\mH)}$]\label{thm:bound-L1alpha-mE}
Under the assumption that  $U_0 \in D(\Phi)$, the estimator $\mE_\mP$, defined in \eqref{eq:post-error-mEt}, satisfies
\begin{equation}\label{eq:bound-L1alpha-mE}
  \Vert \mE_\mP \Vert_{L^1_{\alpha}(0,T;\mH)} \le C \tau^{\alpha} \l( \Vert f \Vert_{L^2_{\alpha}(0, T; \mH)}^2 + \Phi(U_0) - \Phi_{\inf} \r),
\end{equation}
where the constant $C$ depends only on $\alpha$.
\end{Theorem}
\begin{proof}
We bound the contributions $\mE_{\mP,1}$ and $\mE_{\mP,2}$ separately. The bound of $\mE_{\mP,1}$ follows without change that of the term $\mathrm{II}$ of \eqref{eq:a-priori-U12-proof3} in \Cref{lemma:a-priori-U12}. Thus,
\begin{equation}\label{eq:bound-L1alpha-mE1}
  \Vert \mE_{\mP,1} \Vert_{L^1_{\alpha}(0,T;\mH)} \le C \tau^{\alpha} \l( \Vert f \Vert_{L^2_{\alpha}(0, T; \mH)}^2 + \Phi(U_0) - \Phi_{\inf} \r).
\end{equation}

To bound $\mE_{\mP,2}$, we recall the function $\widehat{\Phi}_\mP$, defined in \Cref{remark:Phihat}, and its properties. Define also $\overline{\Phi}_\mP(t) = \Phi(\overline{U}_\mP(t))$. We have
\[
\begin{aligned}
\mE_{\mP,2}(t) &= \Phi\l( \widehat{U}_\mP(t) \r) - \Phi(\overline{U}_\mP(t)) \le \widehat{\Phi}_\mP(t) - \overline{\Phi}_\mP(t) \\
&= \frac{1}{\Gamma(\alpha)} \l( 
\int_0^t (t-s)^{\alpha-1} D_c^{\alpha} \widehat{\Phi}_\mP(s) \diff s - 
\int_0^{\tceil_\mP} (\tceil_\mP-s)^{\alpha-1} D_c^{\alpha} \widehat{\Phi}_\mP(s) \diff s \r) \\
&= \frac{1}{\Gamma(\alpha)} \l( 
\int_0^t [(t-s)^{\alpha-1}- (\tceil_\mP-s)^{\alpha-1}] D_c^{\alpha} \widehat{\Phi}_\mP(s) \diff s - 
\int_t^{\tceil_\mP} (\tceil_\mP-s)^{\alpha-1} D_c^{\alpha} \widehat{\Phi}_\mP(s) \diff s \r) \\
&\le \frac{1}{4\Gamma(\alpha)} \int_0^t [(t-s)^{\alpha-1}- (\tceil_\mP-s)^{\alpha-1}] \l\Vert \overline{F}_\mP(s) \r\Vert^2 \diff s - \frac{1}{\Gamma(\alpha)} \int_t^{\tceil_\mP} (\tceil_\mP-s)^{\alpha-1} D_c^{\alpha} \widehat{\Phi}_\mP(s) \diff s \\
&= \frac{1}{4\Gamma(\alpha)} \int_0^t [(t-s)^{\alpha-1}- (\tceil_\mP-s)^{\alpha-1}] \l\Vert \overline{F}_\mP(s) \r\Vert^2 \diff s - \frac{1}{\Gamma(\alpha+1)} (\tceil_\mP - t)^{\alpha} D_c^{\alpha} \widehat{\Phi}_\mP(t) \\ &= \mathrm{I}_1(t) - \mathrm{I}_2(t).
\end{aligned}
\]
On the one hand, proceeding as in the proof of \Cref{lemma:alpha-int-w} we obtain
\[
\sup_{r \in [0,T]} \int_0^r (r-t)^{\alpha-1} \mathrm{I}_1(t) \diff t
\le C_3 \tau^{\alpha} \Vert \overline{F}_\mP \Vert^2_{L^2_{\alpha}(0, T; \mH)}.
\]
On the other hand, using
\[
-\mathrm{I}_2(t) \le \frac{-1}{\Gamma(\alpha+1)} (\tceil_\mP - t)^{\alpha}  \l( D_c^{\alpha} \widehat{\Phi}_\mP(t) - \frac{1}{4}\l\Vert \overline{F}_\mP(t) \r\Vert^2 \r)
\le \frac{\tau^{\alpha}}{\Gamma(\alpha+1)}  
\l( \frac{1}{4}\l\Vert \overline{F}_\mP(t) \r\Vert^2 - D_c^{\alpha} \widehat{\Phi}_\mP(t) \r)
\]
we have for any $r \in [0,T]$ that
\[
\begin{aligned}
& -\int_0^r (r-t)^{\alpha-1} \mathrm{I}_2(t) \diff t
\le \frac{\tau^{\alpha}}{\Gamma(\alpha+1)} \int_0^r (r-t)^{\alpha-1} \l( \frac{1}{4}\l\Vert \overline{F}(t) \r\Vert^2 - D_c^{\alpha} \widehat{\Phi}_\mP(t) \r) \diff t \\
&= \frac{\tau^{\alpha}}{4\Gamma(\alpha+1)} \int_0^r (r-t)^{\alpha-1} \l\Vert \overline{F}_\mP(t) \r\Vert^2 \diff t - \frac{\tau^{\alpha}}{\alpha} 
\l( \widehat{\Phi}_\mP(r) - \Phi(U_0) \r) \\
&\le \frac{\tau^{\alpha}}{4\Gamma(\alpha+1)} \Vert \overline{F}_\mP \Vert^2_{L^2_{\alpha}(0, T; \mH)} + \frac{\tau^{\alpha}}{\alpha}
\l( \Phi(U_0) - \Phi_{\inf}  \r).
\end{aligned}
\]
Therefore combining the estimates for $\mathrm{I}_1$ and $\mathrm{I}_2$ we have proved that
\[
\sup_{r \in [0,T]} \int_0^r (r-t)^{\alpha-1} \mE_{\mP,2}(t) \diff t
\le C_4 \tau^{\alpha} \l( \Vert f \Vert_{L^2_{\alpha}(0, T; \mH)}^2 + \Phi(U_0) - \Phi_{\inf} \r),
\]
which together with \eqref{eq:bound-L1alpha-mE1} proves \eqref{eq:bound-L1alpha-mE} because $\mE_\mP$ is nonnegative. 
\end{proof}

We next take advantage of \Cref{lemma:alpha-int-f-Fbar}, and derive a rate for $E$ without additional smoothness assumptions on the right hand side $f$.

\begin{Theorem}[a priori II]\label{thm:post-rate-fL2alpha}
Let $u$ be the energy solution of \eqref{eq:theGradFlow}. Let $\mP$ be a partition of $[0,T]$ defined as in \eqref{eq:def-partition} and $\bU \in \mH^N$ be the discrete solution given by \eqref{eq:discrete-frac-GF} starting from $U_0 \in D(\Phi)$. Let $E$ be defined in \eqref{eq:def-Error}. Then we have
\[
E \le \Vert u_0 - U_0 \Vert + C \tau^{\alpha/2} \l( \Vert f \Vert_{L^2_{\alpha}(0, T; \mH)}^2 + \Phi(U_0) - \Phi_{\inf} \r)^{1/2},
\]
where the constant $C$ depends only on $\alpha$.
\end{Theorem}
\begin{proof}
We follow closely the approach and notation in \Cref{lemma:a-priori-U12}. Define
\[
G(t) = \frac{1}{\Gamma(\alpha)} \int_0^t (t-s)^{\alpha-1} \l( f(s) - \overline{F}_\mP(s) \r) \diff s
\]
and note that,  by \Cref{lemma:alpha-int-f-Fbar}, $G$ satisfies
\begin{equation}\label{post-rate-fL2alpha-proof1}
\tau^{\alpha/2} \Vert G \Vert_{L^\infty(0,T;\mH)}  + 
\Vert G \Vert_{L^2_{\alpha}(0, T; \mH)} \le C_1 \tau^{\alpha} \Vert f \Vert_{L^2_{\alpha}(0, T; \mH)},
\end{equation}
where the constant depends only on $\alpha$. Set $e = u - \widehat{U}_\mP$ and note that \eqref{eq:discrete-error-ineq1} can be rewritten as
\[
\l\langle D_c^{\alpha} \l( e - G\r)(t), \l( e - G\r)(t) \r\rangle + \sigma(\overline{U}_\mP(t); u(t)) + \sigma(u(t); \widehat{U}_\mP(t)) \le \mE_\mP(t) - \l\langle D_c^{\alpha} \l( e - G\r)(t), G(t) \r\rangle.
\]
Notice the resemblance with \eqref{eq:a-priori-U12-proof5}. We can thus proceed as in \Cref{lemma:a-priori-U12}, and use 
\Cref{thm:bound-L1alpha-mE}, to deduce that, for some constant $C$, depending only on $\alpha$
\[
\Vert u-\widehat{U}_\mP - G \Vert^2(t) + E_{\sigma}(t)
\le \Vert u_0 - U_0 \Vert^2 + C_3 \tau^{\alpha} \l( \Vert f \Vert_{L^2_{\alpha}(0, T; \mH)}^2 + \Phi(U_0) - \Phi_{\inf} \r).
\]
Estimate \eqref{post-rate-fL2alpha-proof1} then implies the result.
\end{proof}

\subsubsection{Rate of convergence for smooth energies}\label{sub:SmoothEnergy}

Let us show that, at least for smoother energies, it is possible to obtain a better rate of convergence. We will, essentially, assume that the energy is locally $C^{1+\beta}$ for $\beta \in (0,1]$. More specifically in this section we consider energies that satisfy the following. There exists $\beta \in (0,1]$ such that for every $R > 0$, there is a constant $C_{\beta, R} > 0$ for which
\begin{equation}\label{eq:ineq-smooth-Phi-beta-R}
  \Phi(w_2) - \Phi(w_1) - \langle \xi_1, w_2 - w_1 \rangle \le 
  C_{\beta,R} \Vert w_2 - w_1 \Vert^{1+\beta}, \quad \forall w_1, w_2 \in B_R, \;  \xi_1 \in \partial \Phi(w_1),
\end{equation}
where $B_R$ denotes the ball of radius $R$ in $\mH$. Notice that, by \Cref{lemma:Hold_Uhat}, all the discrete solutions $\widehat{U}_\mP$ are uniformly bounded in $C([0,T];\mH)$. Thus, we can fix $\bar{R} > 0$ depending only on the data such that, for any partition $\mP$ and all $t \in [0,T]$, $\widehat{U}_\mP(t) \in B_{\bar{R}}$. Therefore, \eqref{eq:ineq-smooth-Phi-beta-R} implies that
\begin{equation}\label{eq:ineq-smooth-Phi-beta}
\Phi(w_2) - \Phi(w_1) - \langle \xi_1, w_2 - w_1 \rangle \le 
C_{\beta} \Vert w_2 - w_1 \Vert^{1+\beta}, \quad \forall w_1, w_2 \in \widehat{U}_\mP([0,T]), \;  \xi_1 \in \partial \Phi(w_1),
\end{equation}	
for some constant $C_{\beta}=C_{\beta,\bar{R}}$.

A particular example to which this situation applies is the following. Let $\mH = \mR^d$ and $\Phi(w) = \tfrac1p |w|^p$ with $p > 1$. In this case, \eqref{eq:ineq-smooth-Phi-beta} holds with $\beta = 1$ for $p \ge 2$ and $\beta= p-1$ for $p \in (1,2)$. For $p<2$, to reach $\beta = 1$, we must assume that $u$ and $\widehat{U}_\mP$ stay uniformly away from zero. This example can, of course, be generalized.

In this setting, we have the following improved estimate for $\Vert \mE_\mP \Vert_{L^1_{\alpha}(0,T;\mH)}$. 

\begin{Theorem}[improved bound]\label{thm:bound-L1alpha-mE-2}
Assume that the energy $\Phi$ satisfies \eqref{eq:ineq-smooth-Phi-beta}. Let $u$ be the energy solution to \eqref{eq:theGradFlow}, and denote by $\mP$ a partition of $[0,T]$ defined as in \eqref{eq:def-partition}. Denote by $\widehat{U}_\mP$ the solution of \eqref{eq:discrete-frac-GF} starting from $U_0 \in D(\Phi)$. In this setting, the estimator $\mE_\mP$ defined in \eqref{eq:post-error-mEt} satisfies
\begin{equation}\label{eq:bound-L1alpha-mE-2}
\Vert \mE_\mP \Vert_{L^1_{\alpha}(0,T;\mH)} \le C T^{\alpha(1-\beta)/2} \tau^{\alpha(\beta+1)} \l( \Vert f \Vert_{L^2_{\alpha}(0, T; \mH)}^2 + \Phi(U_0) - \Phi_{\inf} \r)^{(\beta+1)/2},
\end{equation}
for some constant $C$ that depends on $\alpha$, $\beta$, and the problem data. 
% \wl{This is incorrect because $C$ also depends on $C_{\beta}$, which depends on the data. If you do not want to write $C_{\beta}$ explicitly, we need to say that $C$ depends on the data. Changes also need to be done for the statement of \Cref{thm:post-rate-W2inf} and \Cref{thm:post-rate-W2inf-LipPerturb}. }
\end{Theorem}
\begin{proof}
Owing to \eqref{eq:ineq-smooth-Phi-beta}, the estimator $\mE_\mP$ can be bounded from above by
\[
\mE_\mP(t) = \langle D_c^{\alpha} \widehat{U}_\mP(t) - \overline{F}_\mP(t), \widehat{U}_\mP(t) - \overline{U}_\mP(t) \rangle + \Phi(\widehat{U}_\mP(t)) - \Phi(\overline{U}_\mP(t))
\le C_{\beta} \Vert \widehat{U}_\mP(t) - \overline{U}_\mP(t) \Vert^{1+\beta}.
\]
Applying \Cref{lemma:alpha-int-w} with $p=1+\beta$ we have
\[
\Vert \mE_\mP \Vert_{L^1_{\alpha}(0,T;\mH)} \le \sup_{r \in [0,T]} C_{\beta} \int_0^r (r-t)^{\alpha-1}  \Vert \widehat{U}_\mP(t) - \overline{U}_\mP(t) \Vert^{1+\beta} \diff t
\le C \tau^{\alpha(1+\beta)} \l\Vert D_c^{\alpha} \widehat{U} \r\Vert^{1+\beta}_{L^{1+\beta}_{\alpha}(0,T;\mH)},
\]
for some constant $C$ that depends on $\alpha, \beta$ and the problem data. Since $1+\beta \in (1,2]$, \Cref{lemma:bound-int-V} and the embedding 
\[
\Vert w \Vert_{L^{1+\beta}_{\alpha}(0,T;\mH)}
\le \Vert w \Vert_{L^{2}_{\alpha}(0,T;\mH)} 
\l( \frac{T^{\alpha}}{\alpha} \r)^{(1-\beta)/(2(1+\beta))},
\]
imply that
\[
\l\Vert D_c^{\alpha} \widehat{U}_\mP \r\Vert^{1+\beta}_{L^{1+\beta}_{\alpha}(0,T;\mH)}
\le C_2 T^{\alpha(1-\beta)/2} \l( \Vert f \Vert_{L^2_{\alpha}(0, T; \mH)}^2 + \Phi(U_0) - \Phi_{\inf} \r)^{(1+\beta)/2},
\]
and this implies the claim.
\end{proof}

Now, in order to obtain a convergence rate using \eqref{eq:post-error}, we still need to control $\Vert f - \overline{F}_\mP \Vert_{L^1_{\alpha}(0, T; \mH)}$. To do so, we invoke inequality \eqref{eq:Lq-embed-Lpalpha} and see that
\[
\Vert f - \overline{F}_\mP \Vert_{L^1_{\alpha}(0, T; \mH)}
\le \l( \frac{q-1}{q\alpha - 1} \r)^{(q-1)/q} T^{\alpha - 1/q} \Vert f - \overline{F}_\mP \Vert_{L^q(0, T; \mH)}
\]
for $q > 1/\alpha$. Thus, if $f \in W^{\alpha(1+\beta)/2,q}(0, T; \mH)$, then we have
\[
\Vert f - \overline{F}_\mP \Vert_{L^q(0, T; \mH)} \leq C
\tau^{\alpha(1+\beta)/2} |f|_{W^{\alpha(1+\beta)/2,q}(0, T; \mH)} 
\]
and hence
\begin{equation}\label{eq:L1alpha-f-Fbar}
\Vert f - \overline{F}_\mP \Vert_{L^1_{\alpha}(0, T; \mH)} \le C T^{\alpha - 1/q}
\tau^{\alpha(1+\beta)/2} |f|_{W^{\alpha(1+\beta)/2,q}(0, T; \mH)}
\end{equation}
for some constant $C$ that depends on $\alpha$ and $q$. Combining this with \Cref{thm:bound-L1alpha-mE-2}, the following convergence rate is a direct consequence of \Cref{thm:post-error}.

\begin{Theorem}[improved rate: smooth energies]\label{thm:post-rate-W2inf}
Assume that the energy $\Phi$ satisfies \eqref{eq:ineq-smooth-Phi-beta}. Let $u$ be the energy solution to \eqref{eq:theGradFlow}, and denote by $\mP$ a partition of $[0,T]$ defined as in \eqref{eq:def-partition}. Denote by $\widehat{U}_\mP$ the solution of \eqref{eq:discrete-frac-GF} starting from $U_0 \in D(\Phi)$. In this setting, if there is $q>1/\alpha$ for which $f \in W^{\alpha(\beta+1)/2,q}(0, T; \mH)$ then the error $E$, defined in \eqref{eq:def-Error}, satisfies
\[
E \le \Vert u_0 - U_0 \Vert + C \tau^{\alpha(\beta+1)/2} \l[ \l( \Vert f \Vert_{L^2_{\alpha}(0, T; \mH)}^2 + \Phi(U_0) - \Phi_{\inf} \r)^{(\beta+1)/4} + |f|_{W^{\alpha(\beta+1)/2,q}(0, T; \mH)} \r],
\]
where the constant $C$ depends on $\alpha$, $\beta$, $q$, $T$, and the problem data.
% C$ depends only on  $\alpha, \beta, q$ and $T$.
\end{Theorem}

\subsubsection{Rate of convergence for linear problems}\label{sub:LinearProblems}
Let us now show how for certain classes of linear problems an improved rate of convergence can be obtained. We first assume that we have a Gelfand triple,
\[
  \mV \hookrightarrow \mH \hookrightarrow \mV'
\]
and that 
\begin{equation}
\label{eq:QuadEnergy}
  \Phi(w) = \begin{dcases}
              \frac12 \fa(w,w), & w \in \mV, \\ +\infty, &w \notin \mV.
            \end{dcases}
\end{equation}
where $a: \mV \times \mV \to \mR$ is a nonnegative, symmetric, bounded, and semicoercive bilinear form. In this setting, \eqref{eq:sol-ineq-@t} becomes
\[
  \langle D_c^\alpha u, w \rangle + \fa(u,w) = \langle f, w\rangle, \quad \forall w \in \mV.
\]
Notice that the bilinear form induces an operator $\fA: \mV \to \mV'$ given by 
\[
  \langle \fA v, w \rangle_{\mV,\mV'} = \fa(v,w), \quad \forall v,w \in \mV,
\]
which implies that, for almost every $t \in (0,T)$, we have a problem in $\mV'$  which reads
\[
D_c^\alpha u(t) + \fA u(t) = f(t).
\]
So that, $u_0 \in D(\partial \Phi)$ is equivalent to $\fA u_0 \in \mH$. The bilinear form $\fa$ also induces a semi-norm on $\mV$
\[
[w]_{\mV} =  \fa(w,w) ^{1/2}.
\]
We further assume that $f \in L^2_{\alpha}(0,T; [ \cdot ]_{\mV})$. More essentially we also require $u_0 \in D(\partial\Phi)$.

The motivation for an improved rate of convergence is then the following, at this stage formal, calculation. From \eqref{eq:Caputo-ineq-norm-sq} we have
\[
\begin{aligned}
  \frac12 D_c^\alpha \| \fA u(t) \|^2 &\leq \langle D_c^\alpha \fA u(t), \fA u(t) \rangle =
  \langle \fA u(t), \fA D_c^\alpha u(t) \rangle =
  \langle  f(t) - D_c^\alpha u(t), \fA D_c^\alpha u(t) \rangle \\ &=
  \fa( f(t), D_c^\alpha u(t) ) - [D_c^\alpha u(t)]^2_{\mV}
  \leq [f(t)]_{\mV} [D_c^\alpha u(t)]_{\mV} - [D_c^\alpha u(t)]^2_{\mV}.
\end{aligned}
\]
Which then shows via \eqref{eq:Volterra-eq} that
\[
\begin{aligned}
 & \frac{\Gamma(\alpha)}{2} \| \fA u(t) \|^2 + \int_0^t (t-s)^{\alpha-1} [D_c^\alpha u(s)]^2_{\mV} \; \diff s \\
 & \leq \frac{\Gamma(\alpha)}{2} \| \fA u_0 \|^2 + \l( \int_0^t (t-s)^{\alpha-1} [f(s)]^2_{\mV} \; \diff s \r)^{1/2} \l( \int_0^t (t-s)^{\alpha-1} [D_c^\alpha u(s)]^2_{\mV} \; \diff s \r)^{1/2}.
\end{aligned}
\]
This implies that
\[
[ D_c^\alpha u ]^2_{L^2_{\alpha}(0,T; [ \cdot ]_{\mV})} = \sup_{t \in [0,T]} \int_0^t (t-s)^{\alpha-1} \l[ D_c^\alpha u(s) \r]_{\mV}^2 \diff s
\leq \Gamma(\alpha) \| \fA u_0 \|^2 + \Vert f \Vert^2_{L^2_{\alpha}(0,T; [ \cdot ]_{\mV})},
\]
which says that $D_c^\alpha u$ is uniformly bounded in $L^2_{\alpha}(0,T; [ \cdot ]_{\mV})$.

To make these considerations rigorous, we consider the discrete problem \eqref{eq:discrete-frac-GF}, which in this case reduces to
\[
  (D_\mP^\alpha \bU)_n + \fA U_n = F_n,
\]
Then the computations can be followed verbatim to obtain that
\[
\begin{aligned}
  & \frac{\Gamma(\alpha)}{2} \| \fA \widehat{U}_\mP(t) \|^2 + \int_0^t (t-s)^{\alpha-1} \l[ D_c^\alpha \widehat{U}_\mP(s) \r]^2 \diff s \\
  & \leq \frac{\Gamma(\alpha)}{2} \| \fA U_0 \|^2 + \l( \int_0^t (t-s)^{\alpha-1} \l[ D_c^\alpha \widehat{U}_\mP(s) \r]^2 \diff s \r)^{1/2} \l( \int_0^t (t-s)^{\alpha-1} \l[ \overline{F}_\mP(s) \r]^2 \diff s \r)^{1/2}
\end{aligned}
\]
and
\begin{equation}\label{eq:L2-alpha-regularity-linear}
\l[ D_c^\alpha \widehat{U}_\mP \r]_{L^2_{\alpha}(0,T; [ \cdot ]_{\mV})}^2 = \sup_{t \in [0,T]} \int_0^t (t-s)^{\alpha-1} \l[ D_c^\alpha \widehat{U}_\mP (s) \r]_{\mV}^2 \diff s
\leq \Gamma(\alpha) \| \fA U_0 \|^2 + \Vert \overline{F}_\mP \Vert^2_{L^2_{\alpha}(0,T; [ \cdot ]_{\mV})}.
\end{equation}
Similar to \Cref{lemma:pw-interp-L2alpha}, we know that
\[
\Vert \overline{F}_\mP \Vert_{L^2_{\alpha}(0,T; [ \cdot ]_{\mV})}
\leq  C \Vert f \Vert_{L^2_{\alpha}(0,T; [ \cdot ]_{\mV})}
\]
and hence $D_c^\alpha \widehat{U}_\mP$ is uniformly bounded $L^2_{\alpha}(0,T; [ \cdot ]_{\mV})$.

With this additional regularity, we can obtain an improved rate of convergence. To see this, we will use that $\Phi$ is, essentially, quadratic to observe that in this case the error estimator, defined in \eqref{eq:post-error-mEt} reduces to
\begin{equation}
\label{eq:estforGelfand}
  \mE_\mP = \frac12 \fa( \widehat{U}_\mP - \overline{U}_\mP, \widehat{U}_\mP - \overline{U}_\mP)  = \frac12 \l[ \widehat{U}_\mP - \overline{U}_\mP \r]_{\mV}^2 .
\end{equation}

These ingredients together give us the following improved estimate.

\begin{Theorem}[improved rate: linear problems]\label{thm:rateGelfandTriples}
Assume that the energy $\Phi$ is given by \eqref{eq:QuadEnergy}, that the initial data satisfies $\fA u_0 \in \mH$, and that $f \in L^2_{\alpha}(0,T; [ \cdot ]_{\mV})$. Let $u$ be the energy solution to \eqref{eq:theGradFlow}, and denote by $\mP$ a partition of $[0,T]$ defined as in \eqref{eq:def-partition}. Denote by $\widehat{U}_\mP$ the solution to \eqref{eq:discrete-frac-GF} starting from $U_0 \in \mH$, such that $\fA U_0 \in \mH$. In this setting, we have that
\begin{equation}\label{eq:bound-L1alpha-mE-linear}
  \Vert \mE_\mP \Vert_{L^1_\alpha(0,T;\mH)} \leq C \tau^{2\alpha} \l( \| \fA U_0 \|^2 + \Vert f \Vert^2_{L^2_{\alpha}(0,T; [ \cdot ]_{\mV})} \r),
\end{equation}
where the constant $C$ depends only on $\alpha$. This, immediately, implies that
\[
E \le \Vert u_0 - U_0 \Vert + C \tau^{\alpha} \l( \| \fA U_0 \| + \Vert f \Vert_{L^2_{\alpha}(0,T; [ \cdot ]_{\mV})}  \r) + \Vert f - \overline{F}_\mP \|_{L^1_\alpha,\mH)},  
\]
so that if, in addition, we further have $f \in W^{\alpha,q}(0,T;\mH)$ for some $q > 1/\alpha$, then
\begin{equation}\label{eq:post-rate-linear}
E \le \Vert u_0 - U_0 \Vert + C \tau^{\alpha} \l( \| \fA U_0 \| + \Vert f \Vert_{L^2_{\alpha}(0,T; [ \cdot ]_{\mV})} + |f|_{W^{\alpha,q}(0,T;\mH)} \r)
\end{equation}
where the constant $C$ depends only on $\alpha, q$ and $T$.
\end{Theorem}
\begin{proof}
Owing to \Cref{thm:post-error} and equation \eqref{eq:L1alpha-f-Fbar}, the convergence rate \eqref{eq:post-rate-linear} follows directly from 
\eqref{eq:bound-L1alpha-mE-linear} in the same way as \Cref{thm:post-rate-W2inf}.
We only need to prove \eqref{eq:bound-L1alpha-mE-linear} and bound $\Vert \mE_\mP \Vert_{L^1_\alpha(0,T;\mH)}$. Using \eqref{eq:estforGelfand}, for every $r \in (0,T]$ we have
\begin{align*}
  2\int_0^r (r-t)^{\alpha-1}\mE_\mP(t) \diff t &= \int_0^r (r-t)^{\alpha-1} \l[ \widehat{U}_\mP - \overline{U}_\mP \r]_{\mV}^2 (t) \diff t.
\end{align*}

Now, we invoke \Cref{lemma:alpha-int-w} with $p = 2$ and the semi-norm $[ \cdot ]_{\mV}$ to obtain that
\[
\int_0^r (r-t)^{\alpha-1} \l[ \widehat{U}_\mP - \overline{U}_\mP \r]_{\mV}^2 (t) \diff t \leq C \tau^{2\alpha} \l[ D_c^\alpha \widehat{U}_\mP \r]_{L^2_{\alpha}(0,T; [ \cdot ]_{\mV})}^2.
\]
By \eqref{eq:L2-alpha-regularity-linear}, we have that $D_c^\alpha \widehat{U}_\mP \in L^2_{\alpha}(0,T; [ \cdot ]_{\mV})$ uniformly in $\mP$ and thus arrive at
\[
\int_0^r (r-t)^{\alpha-1} \l[ \widehat{U}_\mP - \overline{U}_\mP \r]_{\mV}^2 (t) \diff t \leq C \tau^{2\alpha} \l( \| \fA U_0 \|^2 + \Vert f \Vert^2_{L^2_{\alpha}(0,T; [ \cdot ]_{\mV})} \r).
\]
This implies the desired bound
\[
\Vert \mE_\mP \Vert_{L^1_\alpha(0,T;\mH)} \leq C \tau^{2\alpha} \l( \| \fA U_0 \|^2 + \Vert f \Vert^2_{L^2_{\alpha}(0,T; [ \cdot ]_{\mV})} \r)
\]
for $\Vert \mE_\mP \Vert_{L^1_\alpha(0,T;\mH)}$ and finishes the proof.
\end{proof}

\section{Lipschitz perturbations} \label{sec:LipschitzPerturbation}

In this section, inspired by the results of \cite{MR4040846}, we consider the analysis and approximation of a fractional gradient flow with a Lipschitz perturbation. Namely, we consider the following problem
\begin{equation}
\label{eq:LipschitzPerturbation}
  \begin{dcases}
      D_c^\alpha u(t) + \partial \Phi(u(t)) + \Psi(t,u(t)) \ni f(t), &t \in (0,T], \\
      u(0) = u_0.
  \end{dcases}
\end{equation}
We assume that the perturbation function $\Psi: (0,T] \times \mH \to \mH$ satisfies
\begin{enumerate}[1.]
  \item \label{it:Carat} (Carath\'eodory) For every $w \in \mH$ the mapping $t \mapsto \Psi(t,w)$ is strongly measurable on $(0,T)$ with values in $\mH$. Moreover, there exists $\mL >0$ such that for almost every $t \in (0,T)$ and every $w_1,w_2 \in \mH$ we have
  \[
    \| \Psi(t,w_1) - \Psi(t,w_2) \| \leq \mL \| w_1 - w_2 \|.
  \]

  \item \label{it:Integr} (Integrability) There is $w_0 \in L^2_\alpha(0,T;\mH)$ for which
  \[
    t \mapsto \Psi(t,w_0(t)) \in L^2_\alpha(0,T;\mH).
  \]

\end{enumerate}

We immediately comment that our assumptions can fit the case where $\Phi$ is merely $\lambda$--convex. Moreover, these assumptions also guarantee the existence of $\psi \in L^2_\alpha(0,T;\mR)$ for which
\[
  \| \Psi(t,w) \| \leq \psi(t) + \mL\|w \|, \quad \forall w \in \mH.
\]
Consequently $w \mapsto \Psi(\cdot, w(\cdot) )$ is Lipschitz continuous in $L^2_\alpha(0,T;\mH)$.

We introduce the notion of \emph{energy solution} of \eqref{eq:LipschitzPerturbation}.

\begin{Definition}[energy solution]\label{def:strongLipschitz}
A function $u \in L^2(0,T;\mH)$ is an energy solution to \eqref{eq:LipschitzPerturbation} if
\begin{enumerate}[(i)]
  \item (Initial condition)
  \[
    \lim_{t \downarrow0} \fint_0^t \| u(s) - u_0 \|^2 \diff s  = 0.
  \]
  
  \item (Regularity) $D_c^\alpha u \in L^2(0,T;\mH)$.
  
  \item (Evolution) For almost every $t \in (0,T)$ we have
  \[
    D_c^\alpha u(t) + \partial \Phi(u(t)) + \Psi(t,u(t)) \ni f(t).
  \]
\end{enumerate}
\end{Definition}

Evidently, an energy solution to \eqref{eq:LipschitzPerturbation} satisfies, for almost every $t \in (0,T)$ and all $w \in \mH$, the EVI 
\begin{equation}
\label{eq:sol-ineq-@t-LipPerturb}
%   D_c^\alpha u(t) + \partial \Phi(u(t)) \ni f(t) - \Psi(t,u(t)), \ a.e.~t \in (0,T], \quad u(0) = u_0.
  \l\langle   D_c^\alpha u(t), u(t) - w \r\rangle + \langle \Psi(t,u(t)), u(t) - w \rangle + \Phi(u(t)) - \Phi(w) \leq \langle f(t), u(t) - w \rangle .
\end{equation}

\subsection{Existence, uniqueness, and stability}
Our main result in this direction is the following.

\begin{Theorem}[well posedness]\label{thm:AllLipschitz}
Assume that the energy $\Phi$ is convex, l.s.c., and with nonempty effective domain. Assume the the mapping $\Psi$ satisfies conditions \ref{it:Carat} and \ref{it:Integr} stated above. Let $u_0 \in D(\Phi)$ and $f \in L^2_\alpha(0,T;\mH)$, then there is a unique energy solution to \eqref{eq:LipschitzPerturbation} in the sense of \Cref{def:strongLipschitz}. Moreover, we have that this solution satisfies
\[
  \| D_c^\alpha u \|_{L^2_\alpha(0,T;\mH)} \leq C,
\]
where the constant depends only on the problem data $\alpha$, $T$, $u_0$, $f$, $\Phi$, and $\Psi$.
\end{Theorem}
\begin{proof}
We begin by proving existence. We essentially follow the idea used for the classical ODEs. A similar argument was also used in the proof of \cite[Theorem 4.4]{LiLiuGeneralized-18}.

For $w \in L^2_\alpha(0,T;\mH)$ we denote by $ \fS(w) \in L^2_\alpha(0,T;\mH)$ the energy solution to
\[
  D_c^\alpha u(t) + \partial \Phi(u(t)) \ni f(t) - \Psi(t,w(t)), \ a.e.~t \in (0,T], \quad u(0) = u_0.
\]
Our assumptions and the results of \Cref{thm:exist-uniq-frac-GF} guarantee that this mapping is well defined, and moreover, $\fS(w) \in L^{\infty}(0,T;\mH)$. We want to show that there exists a fixed point $w$ such that $\fS(w) = w$. If $u_i = \fS(w_i)$ for $i=1,2$, then for almost every $t$ we have
\[
\frac12 D_c^\alpha \| u_1(t) - u_2(t) \|^2 \leq - \langle \Psi(t,w_1(t)) - \Psi(t,w_2(t)), u_1(t) -u_2(t) \rangle.
\]
This readily implies that
\[
\begin{aligned}
\| u_1(t) - u_2(t) \|^2 &\leq \frac\mL{\Gamma(\alpha)} \int_0^t (t-s)^{\alpha-1} \| w_1(s) - w_2(s) \| \| u_1 (s) - u_2(s) \| \diff s \\
&\leq \frac{\mL \; \Vert u_1 - u_2 \Vert_{L^{\infty}(0,t;\mH)}}{\Gamma(\alpha)} \int_0^t (t-s)^{\alpha-1} \| w_1(s) - w_2(s) \| \diff s
\end{aligned}
\]
which as a consequence yields that, for every $t \in [0,T]$,
\[
\| u_1 - u_2 \|_{L^{\infty}(0,t;\mH)} \leq \frac{\mL}{\Gamma(\alpha)} \Vert w_1 - w_2 \Vert_{L^1_{\alpha}(0,t;\mH)}.
\]
We claim that by induction, we can further obtain the following stability result
\begin{equation}\label{eq:Lip-stability-induction}
\Vert \fS^n(w_1) - \fS^n(w_2) \Vert_{L^{\infty}(0,t;\mH)}
\leq \frac{\mL^n \; t^{\alpha n}}{\Gamma(\alpha n+1)} \Vert w_1 - w_2 \Vert_{L^{\infty}(0,t;\mH)}
\end{equation}
for any $t \in [0,T]$ and positive integer $n$. In fact, for $n = 1$, we simply have
\[
\| u_1 - u_2 \|_{L^{\infty}(0,t;\mH)} \leq \frac{\mL}{\Gamma(\alpha)} \Vert w_1 - w_2 \Vert_{L^1_{\alpha}(0,t;\mH)} \le 
\frac{\mL \; t^{\alpha}}{\Gamma(\alpha+1)} \Vert w_1 - w_2 \Vert_{L^{\infty}(0,t;\mH)}.
\]
Furthermore, if \eqref{eq:Lip-stability-induction} holds for $n = k$, then for $n = k+1$
\[
\begin{aligned}
\| \fS^{k+1}(w_1) - \fS^{k+1}(w_2) \|_{L^{\infty}(0,t;\mH)} &\leq \frac{\mL}{\Gamma(\alpha)} \Vert \fS^{k}(w_1) - \fS^{k}(w_2) \Vert_{L^1_{\alpha}(0,t;\mH)} \\
& \leq \frac{\mL}{\Gamma(\alpha)} \sup_{0 \le r \le t} \int_0^r (r-s)^{\alpha - 1} \frac{\mL^k \; s^{\alpha k}}{\Gamma(\alpha k+1)} \Vert w_1 - w_2 \Vert_{L^{\infty}(0,t;\mH)} \diff s \\
%&= \Vert w_1 - w_2 \Vert_{L^{\infty}(0,t;\mH)} \sup_{0 \le r \le t} \frac{\mL^{k+1} \; s^{\alpha(k+1)}}{\Gamma(\alpha(k+1) + 1)}
&= \frac{\mL^{k+1} t^{\alpha (k+1)}}{\Gamma(\alpha (k+1)+1)} \Vert w_1 - w_2 \Vert_{L^{\infty}(0,t;\mH)},
\end{aligned}
\]
which proves \eqref{eq:Lip-stability-induction}. Now consider $w_0 \in L^2_{\alpha} (0,T;\mH)$ and the sequence of functions defined via $w_n = \fS^n (w_0)$. It is easy to see that, for $n \ge 1$, we have $w_n \in L^{\infty}(0,T;\mH)$, and $\sum_{n = 1}^{\infty} \Vert w_n - w_{n+1} \Vert_{L^{\infty}(0,T;\mH)}$ converges because
\[
\sum_{n = 0}^{\infty} \frac{\mL^n \; t^{\alpha n}}{\Gamma(\alpha n+1)} = E_{\alpha} (\mL t^{\alpha}).
\]
This shows that $w_n \to u$ in $L^{\infty}(0,T;\mH)$ for some $u$. Since $w_{n+1} = \fS(w_n)$, it follows immediately that $u = \fS(u)$. This proves the existence of solutions.

As for uniqueness, assume that we have two solutions $u_1$ and $u_2$, for almost every $t$, we have
\[
\frac12 D_c^\alpha \| u_1(t) - u_2(t) \|^2 \leq - \langle \Psi(t,u_1(t)) - \Psi(t,u_2(t)), u_1(t) -u_2(t) \rangle \leq
\mL \; \| u_1(t) - u_2(t) \|^2.
\]
Combining with the fact that $u_1(0) = u_2(0) = u_0$, one obtains that $\| u_1(t) - u_2(t) \|^2 = 0$ for almost every $t$, which proves uniqueness.

Finally, the estimate on the Caputo derivative trivially follows from the iteration scheme. We skip the details.
\end{proof}

For diversity in our arguments, we present an alternative proof. The arguments here are inspired by those of \cite[Theorem 5.1]{MR4040846}.

\begin{proof}[Alternative proof of \Cref{thm:AllLipschitz}]
Let us, for $\mu > \mL^{1/\alpha}$, define
\[
  \Vert w \Vert_\mu^2 = \sup_{t \in [0,T]} e^{-\mu t} \int_0^t(t-s)^{\alpha-1} \| w(s) \|^2 \diff s,
\]
which by the obvious inequalities $e^{-\mu T} \leq e^{-\mu t} \leq 1$, defines an equivalent norm in $L^2_\alpha(0,T;\mH)$.

Let $\fS: L^2_\alpha(0,T;\mH) \to L^2_\alpha(0,T;\mH)$ be as before. As shown, if $u_i = \fS(w_i)$ for $i=1,2$, then for every $t$ we have
% \[
%   \frac12 D_c^\alpha \| u_1(t) - u_2(t) \|^2 \leq - \langle \Psi(t,w_1(t)) - \Psi(t,w_2(t)), u_1(t) -u_2(t) \rangle.
% \]
% This readily implies that
\[
  \| u_1(t) - u_2(t) \|^2 \leq \frac\mL{\Gamma(\alpha)} \int_0^t (t-s)^{\alpha-1} \| w_1(s) - w_2(s) \| \| u_1 (s) - u_2(s) \| \diff s,
\]
which as a consequence yields that, for every $r \in [0,T]$,
\[
  e^{-\mu r} \int_0^r(r-t)^{\alpha-1} \| u_1(r) - u_2(r) \|^2 \diff r \leq \frac{\mL e^{-\mu r}}{\Gamma(\alpha)} \mathrm{I}(r),
\]
where
\[
  \mathrm{I}(r) = \int_0^r (r-t)^{\alpha-1} \int_0^t (t-s)^{\alpha-1} \| w_1(s) - w_2(s) \| \| u_1 (s) - u_2(s) \| \diff s \diff t.
\]
Obvious manipulations then yield
\[
  \mathrm{I}(r) \leq \| u_1 - u_2 \|_\mu \| w_1 - w_2 \|_\mu \int_0^r (r-t)^{\alpha-1} e^{\mu t} \diff t,
\]
which implies
\[
  e^{-\mu r} \int_0^r(r-t)^{\alpha-1} \| u_1(r) - u_2(r) \|^2 \diff r \leq \frac{\mL }{\Gamma(\alpha)} \int_0^r (r-t)^{\alpha-1} e^{-\mu (r-t)} \diff t \leq \frac{\mL}{\mu^\alpha} < 1,
\]
so that $\fS$ is a contraction with respect to the norm $\Vert \cdot \Vert_\mu$.
We conclude then by invoking the contraction mapping principle. This unique fixed point, evidently, is a energy solution in the sense of \Cref{def:strongLipschitz}.

Uniqueness and stability follow as before.
\end{proof}

\subsection{Discretization}

Let us now present the numerical scheme for problem \eqref{eq:LipschitzPerturbation}. We follow the previous notations and conventions regarding discretization so that, for any partition $\mP$ of $[0,T]$ defined as in \eqref{eq:def-partition}, we can also consider the discrete solution defined recursively via
\begin{equation}\label{eq:discrete-frac-GF-LipPerturb}
F_n -\l( D_\mP^{\alpha} \bU \r)_n - \Psi_n(U_n) \in \partial \Phi(U_n),
\end{equation}
where $F_n$ is defined in \eqref{eq:average-Fn} and $\Psi_n : \mH \to \mH$ is defined by
\[
  \Psi_n(w) = \fint_{t_{n-1}}^{t_n} \Psi(t, w) \diff t.
\]
Clearly, for every $n$, $\Psi_n$ is Lipschitz continuous with Lipschitz constant $\mL$. Using the definition of $D^{\alpha}_\mP$ in \eqref{eq:discrete-Caputo} and $\bK_{\mP,nn}^{-1} = (\bK_{\mP,nn})^{-1} = \Gamma(\alpha+1) \tau_n^{-\alpha}$, we can rewrite \eqref{eq:discrete-frac-GF-LipPerturb} as
\[
\Gamma(\alpha+1) \tau_n^{-\alpha} U_n + \Psi_n(U_n) + \partial \Phi(U_n) \ni F_n - \sum_{i=0}^{n-1} \bK_{\mP,ni}^{-1} U_i.
\]
Hence the discrete scheme can be recursively well-defined provided $\mL \tau^{\alpha} < \Gamma(\alpha+1)$. For this reason, moving forward, we will implicitly operate under this assumption.
% \wl{(Just a comment here, to show \eqref{eq:DboundUhatLipPerturb}, we may have more restriction for the smallness of $\tau$ than $\mL \tau^{\alpha} < \Gamma(\alpha+1)$. Maybe the statement you wrote here is enough because we do not need to explain too much here.)}

It is possible to show that the discrete solutions in \eqref{eq:discrete-frac-GF-LipPerturb} satisfy
\begin{equation}\label{eq:DboundUhatLipPerturb}
\| D_c^\alpha \widehat{U}_\mP \|_{L^2_\alpha(0,T;\mH)} \leq C,
\end{equation}
with a constant that depends on problem data but is independent of the partition $\mP$. To see this, we follow the arguments of either proof of \Cref{thm:AllLipschitz}, and realize that while the operator $\fS$ may depend on $\mP$, the estimates that we obtain do not.

\subsection{Error estimates}
Let us now show how to derive error estimates for the problem with Lipschitz perturbation \eqref{eq:LipschitzPerturbation}. We recall that the energy solution $u$ to this problem satisfies \eqref{eq:sol-ineq-@t-LipPerturb}. In addition, for simplicity, we will operate under the assumption that the perturbation does not depend explicitly on time, i.e., $\Psi(t,w) = \Psi(w)$ for all $w \in \mH$. The general case only lengthens the discussion but brings nothing substantive to it, as the additional terms that appear can be controlled via arguments used to control terms of the form 
\[
  f(t) -\overline{F}_\mP(t).
\]

Similar to the discussion before, we define the error estimator
% \[
%   \mE_\mP(t) = \langle D_c^{\alpha} \widehat{U}_\mP(t) + \Psi( \overline{U}_\mP(t)) - \overline{F}_\mP(t), \widehat{U}_\mP(t) - \overline{U}_\mP(t) \rangle + \Phi(\widehat{U}_\mP(t)) - \Phi(\overline{U}_\mP(t))
% \]
\begin{equation}\label{eq:post-error-mEt-LipPerturb}
  \mE_{\mP,\mL}(t) = \mE_\mP(t) + \langle \Psi( \overline{U}_\mP(t)), \widehat{U}_\mP(t) - \overline{U}_\mP(t) \rangle ,
%   
%   \langle D_c^{\alpha} \widehat{U}_\mP(t) + \Psi( \overline{U}_\mP(t)) - \overline{F}_\mP(t), \widehat{U}_\mP(t) - \overline{U}_\mP(t) \rangle + \Phi(\widehat{U}_\mP(t)) - \Phi(\overline{U}_\mP(t))
\end{equation}
which, as before, is nonnegative. In addition, for any $w \in \mH$ we have
\begin{equation*}
\begin{aligned}
& \langle D_c^{\alpha} \widehat{U}_\mP(t) + \Psi(\widehat{U}_\mP(t)) - f(t) , \widehat{U}_\mP(t) - w \rangle + \Phi(\widehat{U}_\mP(t)) - \Phi(w) \\
&= \mE_{\mP,\mL}(t) + \langle \overline{F}_\mP(t) - \Psi(\overline{U}_\mP(t)) - D_c^{\alpha} \widehat{U}_\mP(t), w - \overline{U}_\mP(t) \rangle
+ \Phi(\overline{U}_\mP(t)) - \Phi(w) \\
& \quad + \langle \Psi( \overline{U}_\mP(t)) - \Psi( \widehat{U}_\mP(t)) + f(t) - \overline{F}_\mP(t),  w- \widehat{U}_\mP(t) \rangle \\
&\le \mE_{\mP,\mL}(t) + \langle \Psi( \overline{U}_\mP(t)) - \Psi( \widehat{U}_\mP(t)) + f(t) - \overline{F}_\mP(t),  w- \widehat{U}_\mP(t) \rangle - \sigma(\overline{U}_\mP(t); w).
\end{aligned}
\end{equation*}
Setting $w = u(t)$ in the inequality above and setting $w = \widehat{U}(t)$ in \eqref{eq:sol-ineq-@t-LipPerturb} leads to
\begin{multline}\label{eq:discrete-error-ineq-LipPerturb}
\l\langle D_c^{\alpha} \l( \widehat{U}_\mP-u\r)(t), \widehat{U}_\mP(t) - u(t) \r\rangle + \sigma(\overline{U}_\mP(t); u(t)) + \sigma(u(t); \widehat{U}_\mP(t))
\le \\ \mE_{\mP,\mL}(t) + \langle \Psi( \overline{U}_\mP(t)) - \Psi( \widehat{U}_\mP(t)) + f(t) - \overline{F}_\mP(t),  u(t) - \widehat{U}_\mP(t) \rangle + \langle \Psi( \widehat{U}_\mP(t)) - \Psi( u(t)), u(t) - \widehat{U}_\mP(t) \rangle
\end{multline}
for almost every $t \in (0,T)$. This implies the following error estimates.

\begin{Theorem}[a posteriori: Lipschitz perturbations]\label{thm:post-error-LipPerturb}
	Let $u$ be the unique energy solution of \eqref{eq:LipschitzPerturbation}. Let $\mP$ be a partition of $[0,T]$ defined as in \eqref{eq:def-partition} and let $\bU \in \mH^N$ be the discrete solution given by \eqref{eq:discrete-frac-GF-LipPerturb} starting from $U_0 \in D(\Phi)$. Let $E$ and $\mE_{\mP,\mL}$ be defined in \eqref{eq:def-Error} and \eqref{eq:post-error-mEt-LipPerturb}, respectively, The following a posteriori error estimate holds
	\begin{multline}\label{eq:post-error-LipPerturb}
	E \le \l( \Vert u_0 - U_0\Vert^2 + \frac{2}{\Gamma(\alpha)} \Vert \mE_{\mP,\mL} \Vert_{L^1_{\alpha}(0,T;\mH) } \r)^{1/2} (E_{\alpha}(2\mL T^{\alpha}))^{1/2} \\
	+ \frac{2}{\Gamma(\alpha)} \l( \Vert f - \overline{F}_\mP \Vert_{L^1_{\alpha}(0,T;\mH)} + \mL \Vert \overline{U}_\mP - \widehat{U}_\mP \Vert_{L^1_{\alpha}(0,T;\mH)} \r) E_{\alpha}(2\mL T^{\alpha}).
	\end{multline}
\end{Theorem}
\begin{proof}
We argue as in the proof of \eqref{thm:post-error}. To make formulas shorter we omit the coercivity terms.
From \eqref{eq:Caputo-ineq-norm-sq} and \eqref{eq:discrete-error-ineq-LipPerturb} we infer
\begin{multline}
\frac12 D_c^{\alpha} \Vert \widehat{U}_\mP-u \Vert^2(t) \le \l\langle D_c^{\alpha} \l( \widehat{U}_\mP-u\r)(t), \widehat{U}_\mP(t) - u(t) \r\rangle \\
\le \mE_{\mP,\mL}(t) + \langle \Psi(t, \overline{U}_\mP(t)) - \Psi(t, \widehat{U}_\mP(t)) + f(t) - \overline{F}_\mP(t),  u(t) - \widehat{U}_\mP(t) \rangle 
+ \mL \Vert \widehat{U}_\mP(t) - u(t) \Vert^2(t) \\
%- \sigma(\overline{U}_\mP(t); u(t)) - \sigma(u(t); \widehat{U}_\mP(t))
\le \mE_{\mP,\mL}(t) + \l(  \mL \Vert \overline{U}_\mP(t) - \widehat{U}_\mP(t) \Vert + \Vert f(t) - \overline{F}_\mP(t) \Vert \r) \Vert \widehat{U}_\mP(t) - u(t) \Vert 
+ \mL \Vert \widehat{U}_\mP(t) - u(t) \Vert^2
%- \sigma(\overline{U}_\mP(t); u(t)) - \sigma(u(t); \widehat{U}_\mP(t)) \\
.
\end{multline}
Then the error estimate \eqref{eq:post-error-LipPerturb} follows from \Cref{lemma:frac-Gronwall-diff-ineq} with
\[
\lambda = \mL, \; a(t)= \Vert (\widehat{U}_\mP-u)(t) \Vert, \; %b:= 2 \l( \sigma(\overline{U}_\mP; u) + \sigma(u; \widehat{U}_\mP) \r)
b = 0, \; c= 2 \mE_{\mP,\mL}(t), \; d(t)= \mL \Vert \overline{U}_\mP(t) - \widehat{U}_\mP(t) \Vert + \Vert (f - \overline{F}_\mP)(t) \Vert. \qedhere
\]
\end{proof}

We also comment here that by \Cref{lemma:alpha-int-w}
\[
\Vert \overline{U}_\mP - \widehat{U}_\mP \Vert_{L^1_{\alpha}(0,T;\mH)}
\le C \tau^{\alpha} \Vert D_c^{\alpha} \widehat{U}_\mP \Vert_{L^1_{\alpha}(0,T;\mH)}
\le \frac{CT^{\alpha/2}}{\alpha^{1/2}} \tau^{\alpha} \Vert D_c^{\alpha} \widehat{U}_\mP \Vert_{L^2_{\alpha}(0,T;\mH)},
\]
where the constant $C$ only depends on $\alpha$. In addition, the norm on the right hand side is bounded independently of the partition $\mP$; see \eqref{eq:DboundUhatLipPerturb}.
Hence the convergence rates proved in Theorems \ref{thm:post-rate-fL2alpha} and \ref{thm:post-rate-W2inf} also hold for problems with a Lipschitz perturbation. Since the proofs are almost identical, we only state the theorems below without proofs.

\begin{Theorem}[convergence rate: Lipschitz perturbations]\label{thm:post-rate-fL2alpha-LipPerturb}
	Let $u$ be the energy solution of \eqref{eq:LipschitzPerturbation}. Let $\mP$ be a partition of $[0,T]$ defined as in \eqref{eq:def-partition} and $\bU \in \mH^N$ be the discrete solution given by \eqref{eq:discrete-frac-GF-LipPerturb} starting from $U_0 \in D(\Phi)$. Let $E$ be defined in \eqref{eq:def-Error}. Then we have
	\[
	E \le \Vert u_0 - U_0 \Vert (E_{\alpha}(2\mL T^{\alpha}))^{1/2} + C \tau^{\alpha/2} \l( \Vert f \Vert_{L^2_{\alpha}(0, T; \mH)} + \Vert D_c^{\alpha} \widehat{U}_\mP \Vert_{L^2_{\alpha}(0, T; \mH)} \r),
	\]
	where the constant $C$ depends only on $\alpha, \mL$ and $T$, but not on $\mP$.
\end{Theorem}

\begin{Theorem}[improved rate: smooth energies and Lipschitz perturbations]\label{thm:post-rate-W2inf-LipPerturb}
	Assume that the energy $\Phi$ satisfies \eqref{eq:ineq-smooth-Phi-beta}. Let $u$ be the energy solution to \eqref{eq:LipschitzPerturbation}, and denote by $\mP$ a partition of $[0,T]$ defined as in \eqref{eq:def-partition}. Denote by $\widehat{U}_\mP$ the solution of \eqref{eq:discrete-frac-GF-LipPerturb} starting from $U_0 \in D(\Phi)$. In this setting, if there is $q>1/\alpha$ for which $f \in W^{\alpha(\beta+1)/2,q}(0, T; \mH)$ then the error $E$, defined in \eqref{eq:def-Error}, satisfies
	\begin{multline}
	E \le \Vert u_0 - U_0 \Vert (E_{\alpha}(2\mL T^{\alpha}))^{1/2}
	+ C_1 \tau^{\alpha} \Vert D_c^{\alpha} \widehat{U}_\mP \Vert_{L^2_{\alpha}(0,T;\mH)} \\
	+ C_2 \tau^{\alpha(\beta+1)/2} \l[ \l( \Vert f \Vert_{L^2_{\alpha}(0, T; \mH)} + \Vert D_c^{\alpha} \widehat{U}_\mP \Vert_{L^2_{\alpha}(0, T; \mH)} \r)^{(\beta+1)/2} + |f|_{W^{\alpha(\beta+1)/2,q}(0, T; \mH)} \r],
	\end{multline}
	where the constants $C_1$ and $C_2$ depend only on  $\alpha, \beta, q, \mL$, $T$, and the problem data, but are independent of $\mP$.
\end{Theorem}

Finally we consider the setting of \Cref{sub:LinearProblems} with a Lipschitz perturbation. Similar to \eqref{eq:DboundUhatLipPerturb}, we can show that $\Vert D_c^{\alpha} \widehat{U}_\mP \Vert_{L^2_{\alpha}(0, T; [ \cdot ]_{\mV})}$ is bounded uniformly with respect to the partition $\mP$. For this reason, an improved error estimate analogous to 
\Cref{thm:rateGelfandTriples} can be proved in this case.

\begin{Theorem}[improved rate: quadratic energies and Lipschitz perturbations]\label{thm:rateGelfandTriples-LipPerturb}
	Assume that the energy $\Phi$ is given by \eqref{eq:QuadEnergy}, that the initial data satisfies $\fA u_0 \in \mH$, and that $f \in L^2_{\alpha}(0,T; [ \cdot ]_{\mV})$. Let $u$ be the energy solution to \eqref{eq:LipschitzPerturbation}, and denote by $\mP$ a partition of $[0,T]$ defined as in \eqref{eq:def-partition}. Denote by $\widehat{U}_\mP$ the solution to \eqref{eq:discrete-frac-GF-LipPerturb} starting from $U_0 \in \mH$, such that $\fA U_0 \in \mH$. In this setting, we have that
	\begin{multline}\label{eq:post-rate-linear-LipPerturb}
	E \le \Vert u_0 - U_0 \Vert (E_{\alpha}(2\mL T^{\alpha}))^{1/2} + C \Vert f - \overline{F} \Vert_{L^1_{\alpha}(0,T;\mH)} \\
	+ C \tau^{\alpha} \l( \| \fA U_0 \| + \Vert f \Vert_{L^2_{\alpha}(0,T; [ \cdot ]_{\mV})} + \Vert D_c^{\alpha} \widehat{U}_\mP \Vert_{L^2_{\alpha}(0, T; [ \cdot ]_{\mV})} + \Vert D_c^{\alpha} \widehat{U}_\mP \Vert_{L^2_{\alpha}(0, T; \mH)} \r)
	\end{multline}
	where the constant $C$ depends only on $\alpha, \mL$ and $T$.
\end{Theorem}

\section{Numerical illustrations} \label{sec:experiments}

In this section we present some simple numerical examples aimed at illustrating, and extending, our theory. All the computations were done with an in-house code that was written in MATLAB$^\copyright$.

\subsection{Practical a posteriori estimators} \label{sub:estimators}
We begin by commenting that, unlike the a posteriori estimators for the classical gradient flow proposed in \cite{MR1737503}, our a posteriori estimator $\mE_\mP$ is not constant on each subinterval of our partition $\mP$; see \eqref{eq:post-error-mEt}. Here we mention more computationally friendly alternatives, and their properties.

First, we define an estimator that is piecewise constant in time via
\begin{equation*}
 \mD_\mP(t) = \max_{s \in \l[ \lfloor t \rfloor_\mP, \tceil_\mP \r]} \l\{
  \langle D_c^{\alpha} \widehat{U}_\mP(s) - \overline{F}(s), \widehat{U}_\mP(s) - \overline{U}_\mP(s) \rangle + \Phi(\widehat{U}_\mP(s)) - \Phi(\overline{U}_\mP(s)) \r\}
\end{equation*}
This is clearly an upper bound for $\mE_\mP(t)$.

One may also consider the simpler indicator
\begin{equation}
\label{eq:defDP}
  \wt{\mE}_{\mP,n} = \langle \l( D_{\mP}^{\alpha} \bU \r)_n - F_n, U_{n-1} - U_n \rangle + \Phi(U_{n-1}) - \Phi(U_n), \quad n= 1, \ldots, N.
\end{equation}
Although it is not always true that $\mE_\mP(t) \le \wt{\mE}_{\mP,n(t)}$, this indicator is convenient to use in practice and gives reasonable results. In fact, this is the one that we implemented in the numerical examples of \Cref{sub:adaptivity} below.

\subsection{A linear one dimensional example}

As a first simple example we consider the one dimensonal fractional ODE
\begin{equation}\label{eq:fracODE}
D_c^{\alpha} u + \lambda u = 0, \quad u(0) = 1,
\end{equation}
with $\lambda > 0$. From \eqref{eq:GronwallIsExp} we have
$
u(t) = E_{\alpha}(-\lambda t^{\alpha}).
$
This, obviously, fits our framework with $\mH = \mR$, and $\Phi(w)= \frac\lambda2 |w|^2$. Notice also that all the assumptions of \Cref{sub:SmoothEnergy} are also satisfied with $\beta =1$. Thus, we expect a rate of order $\mO(\tau^\alpha)$ when using \eqref{eq:discrete-frac-GF} to approximate the solution over a uniform partition with time step $\tau$.

\begin{table}{\tiny
  \begin{tabular}{ccc}
    $\alpha = 0.3$ & $\alpha = 0.5$ & $\alpha = 0.7$\\
    {
      \begin{tabular}{|c|c|c|} 
        \hline
        $\tau$ & $|u(1) - U_N|$ & rate \\ 
        \hline
        5.000$e$-02 & 4.563$e$-04 & --- \\ 
        2.500$e$-02 & 3.702$e$-04 & 0.301417 \\ 
        1.250$e$-02 & 3.005$e$-04 & 0.300979 \\ 
        6.250$e$-03 & 2.440$e$-04 & 0.300664 \\ 
        3.125$e$-03 & 1.981$e$-04 & 0.300445 \\ 
        1.563$e$-03 & 1.609$e$-04 & 0.300297 \\ 
        7.813$e$-04 & 1.307$e$-04 & 0.300199 \\ 
        3.906$e$-04 & 1.061$e$-04 & 0.300133 \\ 
        1.953$e$-04 & 8.619$e$-05 & 0.300090 \\ 
        9.766$e$-05 & 7.001$e$-05 & 0.300062 \\ 
        4.883$e$-05 & 5.686$e$-05 & 0.300043 \\ 
        2.441$e$-05 & 4.619$e$-05 & 0.300030 \\ 
        \hline
      \end{tabular}
    }&
    {
      \begin{tabular}{|c|c|c|} 
        \hline
        $\tau$ & $|u(1) - U_N|$ & rate \\ 
        \hline
        5.000\texttt{e}-02 & 2.829\texttt{e}-04 & --- \\ 
        2.500\texttt{e}-02 & 1.996\texttt{e}-04 & 0.503051 \\ 
        1.250\texttt{e}-02 & 1.409\texttt{e}-04 & 0.502309 \\ 
        6.250\texttt{e}-03 & 9.954\texttt{e}-05 & 0.501710 \\ 
        3.125\texttt{e}-03 & 7.032\texttt{e}-05 & 0.501248 \\ 
        1.563\texttt{e}-03 & 4.969\texttt{e}-05 & 0.500902 \\ 
        7.813\texttt{e}-04 & 3.512\texttt{e}-05 & 0.500648 \\ 
        3.906\texttt{e}-04 & 2.483\texttt{e}-05 & 0.500463 \\ 
        1.953\texttt{e}-04 & 1.755\texttt{e}-05 & 0.500330 \\ 
        9.766\texttt{e}-05 & 1.241\texttt{e}-05 & 0.500235 \\ 
        4.883\texttt{e}-05 & 8.773\texttt{e}-06 & 0.500166 \\ 
        2.441\texttt{e}-05 & 6.203\texttt{e}-06 & 0.500118 \\ 
        \hline
      \end{tabular}
    }&
    {
      \begin{tabular}{|c|c|c|} 
        \hline
        $\tau$ & $|u(1) - U_N|$ & rate \\ 
        \hline
        5.000\texttt{e}-02 & 1.235\texttt{e}-04 & --- \\ 
        2.500\texttt{e}-02 & 7.571\texttt{e}-05 & 0.705417 \\ 
        1.250\texttt{e}-02 & 4.646\texttt{e}-05 & 0.704620 \\ 
        6.250\texttt{e}-03 & 2.852\texttt{e}-05 & 0.703871 \\ 
        3.125\texttt{e}-03 & 1.752\texttt{e}-05 & 0.703207 \\ 
        1.563\texttt{e}-03 & 1.076\texttt{e}-05 & 0.702638 \\ 
        7.813\texttt{e}-04 & 6.616\texttt{e}-06 & 0.702160 \\ 
        3.906\texttt{e}-04 & 4.068\texttt{e}-06 & 0.701764 \\ 
        1.953\texttt{e}-04 & 2.502\texttt{e}-06 & 0.701437 \\ 
        9.766\texttt{e}-05 & 1.539\texttt{e}-06 & 0.701170 \\ 
        4.883\texttt{e}-05 & 9.465\texttt{e}-07 & 0.700952 \\ 
        2.441\texttt{e}-05 & 5.823\texttt{e}-07 & 0.700774 \\ 
        \hline
      \end{tabular}
    }
  \end{tabular}}
  \caption{Convergence rate for the approximation of \eqref{eq:fracODE} using scheme \eqref{eq:discrete-frac-GF} over a uniform partition of size $\tau$. As predicted by \Cref{sub:SmoothEnergy}, the rate is $\mO(\tau^\alpha)$.}
\label{table:linear1d}
\end{table}

Table~\ref{table:linear1d} shows, for $\lambda = 0.001$ and different values of $\alpha$, the difference $|u(1) - U_N|$ which we use as a proxy for the error $E_\mH$ of \eqref{eq:def-Error}. The rate of convergence is verified.

\subsection{Adaptive time stepping}\label{sub:adaptivity}

We now illustrate the use of the a posteriori error estimator $\mE_\mP$ given in \eqref{eq:post-error-mEt} to drive the selection of the size of the time step. For a given tolerance $\ve$ we, at every step, choose the local time step $\tau_n$ to guarantee that
\[
\frac{2T^{\alpha}}{\Gamma(\alpha+1)} \wt{\mE}_{\mP,n} \le \ve^2,
\]
where $\wt{\mE}_{\mP,n}$ is given in \eqref{eq:defDP}. Then, by \Cref{thm:post-error}, we expect that
\[
  \| u - \widehat{U}_\mP \|_{L^\infty(0,T;\mH)} \leq \ve,
\]
provided the approximation error $\| f - \overline{F}_\mP \|_{L^1_\alpha(0,T;\mH)}$ is negligible. Notice that to drive the process we are using the simpler estimator $\wt{\mE}_{\mP}$; see the discussion in \Cref{sub:estimators}.

\begin{figure}
	\includegraphics[scale=0.5]{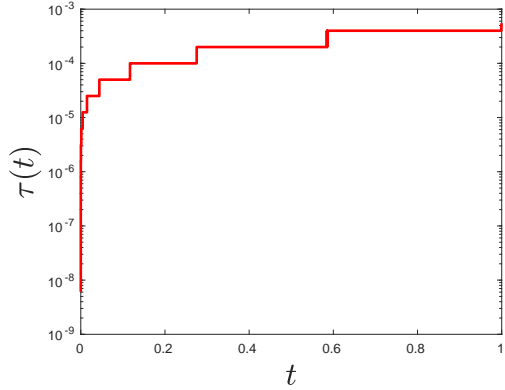}
	\caption{Adaptive time stepping for problem \eqref{eq:fracODE} with $T = 1, \lambda = 1, \alpha = \tfrac12$ is used to achieve a tolerance of $\ve = 10^{-4}$. The adaptive solver uses $8,747$ time intervals with minimum time step $6.1035 \times 10^{-9}$ and max time step $5.4969 \times 10^{-4}$.}\label{fig:adaptive-time-step}
	
	% Maybe we need one or two more pictures here. I can run for different \alpha, or different tolerance showing the ratio between the real error and tolerance is stable.
\end{figure}

We consider the linear problem \eqref{eq:fracODE} with $\lambda = 1$ and $\alpha = \tfrac12$ and set $\ve = 10^{-4}$. Figure~\ref{fig:adaptive-time-step} shows the local time step $\tau(t)$ for $t \in [0,T]$. As expected, due to the weak singularity of $u$ at $t=0$ the time step must be rather small for small times. For larger times, however, the solution is smoother and larger local time steps can be taken. With this process we obtain that
\[
  \| u - \widehat{U}_\mP \|_{L^\infty(0,T;\mH)} \approx 1.805993 \times 10^{-5},
\]
and this requires $N = 8,747$ time subintervals. For comparison, choosing a uniform time step of $\tau = 6.1035 \times 10^{-6}$ we require $N=163,840$ time intervals. This achieves an error of $\ve =4.944 \times 10^{-5}$, which is slightly higher than that obtained with our adaptive procedure. This clearly shows the advantages and possibilities for this strategy.

% 
% 
% 
% As a comparison, we also run experiments with same values of $T, \lambda, \alpha$ but using a uniform time step $\tau = 6.1035 \times 10^{-6}$. The uniform time step solver uses $163840$ time intervals and achieves an error $4.944 \times 10^{-5}$ which is slightly higher than the error for adaptive solver with $\ve = 10^{-4}$. However, on my laptop with Intel i5-8250U CPU @1.60GHz, the time spent for this experiment is about $156.4423$ seconds, while the time spent for the adaptive time step solver is only about $1.015448$ seconds. This indicates that the adaptive solver has advantages in solving such a problem.
% }

\subsection{Some nonlinear one dimensional examples}\label{sub:nl1d}

We now, while staying in one dimension, depart from the linear theory and illustrate the performance of our method in a series of nonlinear examples of increasing difficulty. In all the examples we set $\mH = \mR$ and $f=0$. Thus, we will only specify the energy and initial condition in each case.

In all the examples, since the exact solution is not known, we compare the solutions at different time levels. Specifically, we let $\tau_k = 2^{-k}$ and upon denoting by $U(N_k)$ the approximate solution at $T=1$ computed with step size $\tau_k$, we compute
\[
\mathrm{rate}_k  =  \log_2( |U(N_{k-1}) - U(N_{k-2})| ) - \log_2( |U(N_k) - U(N_{k-1})| ).
\]

\newcounter{exxx}
\setcounter{exxx}{1}

\subsubsection{Example \theexxx}

We let $p \in (1,2)$ and set
\[
  \Phi(w) = \frac\lambda{p}|w|^p, \qquad u_0 = \frac1{10}.
\]
Notice that this example fits the framework of \Cref{sub:SmoothEnergy} with $\beta = p-1$. However, as mentioned there, it is not expected that the solution reaches zero in finite time, so we do not expect a reduced rate.

To compute the discrete solution, at every time step, we need to solve a nonlinear equation of the form
\[
  U_n + c \; U_n |U_n|^{p-2} - W_n = 0, \qquad c = \dfrac{\lambda \tau^{\alpha}}{\Gamma(\alpha+1)}
\]
where $W_n$ is known. We found the solution to this problem using Newton's method, which works for small values of $\tau$.

\begin{table}
  \begin{center}
    \begin{tabular}{|c|c|c|}  
      \hline
      $\tau$ & $|U(N_k) - u(N_{k-1})|$ & rate \\ 
      \hline
      7.813\texttt{e}-04 & --- & --- \\ 
      3.906\texttt{e}-04 & 1.256\texttt{e}-06 & --- \\ 
      1.953\texttt{e}-04 & 6.276\texttt{e}-07 & 1.001307 \\ 
      9.766\texttt{e}-05 & 3.135\texttt{e}-07 & 1.001298 \\ 
      4.883\texttt{e}-05 & 1.568\texttt{e}-07 & 0.999272 \\ 
      2.441\texttt{e}-05 & 7.827\texttt{e}-08 & 1.002774 \\ 
      1.221\texttt{e}-05 & 3.924\texttt{e}-08 & 0.996178 \\ 
      \hline
    \end{tabular}    
  \end{center}
\caption{Convergence rate for  $\alpha = 0.5$, $p = 1.5$, and $\lambda = 1$ in Example~\theexxx~of \Cref{sub:nl1d}. The rate seems to be of order $\mO(\tau)$, which is better than what the theory predicts.}
\label{table:pEnergy}
\end{table}

Table~\ref{table:pEnergy} shows the results for $\alpha = 0.5$, $p = 1.5$, and $\lambda = 1$. These clearly indicate a rate of $\mO(\tau)$.

\addtocounter{exxx}{1}
\subsubsection{Example \theexxx.}
We set 
\[
  \Phi(w) = \lambda( u\ln u - u), \qquad u_0 = 0.
\]
with $\lambda >0$, so that $D(\Phi)=[0, \infty)$. Notice that $u_0 \in D(\Phi)\setminus D(\partial\Phi)$.

At each time step one needs to solve a problem of the form
\[
  U_n + c \ln(U_n) - W_n=0, \qquad  c = \dfrac{\lambda \tau^{\alpha}}{\Gamma(\alpha+1)},
\]
and $W_n$ is known. This is solved with a Newton scheme, which runs into difficulties at the initial time step. We go around this issue by using as initial value for the iteration a very small positive number.

\begin{table}
  \begin{center}
    \begin{tabular}{|c|c|c|} 
      \hline
      $\tau$ & $|U(N_k) - u(N_{k-1})|$ & rate \\ 
      \hline
      5.000\texttt{e}-02 & --- & --- \\ 
      2.500\texttt{e}-02 & 6.761\texttt{e}-07 & --- \\ 
      1.250\texttt{e}-02 & 5.330\texttt{e}-07 & 0.342991 \\ 
      6.250\texttt{e}-03 & 4.088\texttt{e}-07 & 0.382916 \\ 
      3.125\texttt{e}-03 & 3.077\texttt{e}-07 & 0.409583 \\ 
      1.563\texttt{e}-03 & 2.290\texttt{e}-07 & 0.426328 \\ 
      7.813\texttt{e}-04 & 1.689\texttt{e}-07 & 0.439295 \\ 
      3.906\texttt{e}-04 & 1.238\texttt{e}-07 & 0.447818 \\ 
      1.953\texttt{e}-04 & 9.039\texttt{e}-08 & 0.453936 \\ 
      9.766\texttt{e}-05 & 6.579\texttt{e}-08 & 0.458399 \\ 
      4.883\texttt{e}-05 & 4.777\texttt{e}-08 & 0.461709 \\ 
      2.441\texttt{e}-05 & 3.463\texttt{e}-08 & 0.464209 \\ 
      1.221\texttt{e}-05 & 2.507\texttt{e}-08 & 0.466136 \\ 
      6.104\texttt{e}-06 & 1.813\texttt{e}-08 & 0.467656 \\ 
      3.052\texttt{e}-06 & 1.310\texttt{e}-08 & 0.468883 \\ 
      \hline
    \end{tabular}    
  \end{center}
\caption{Convergence rate for  $\alpha = 0.5$ and $\lambda = 10^{-6}$ in Example~\theexxx~of \Cref{sub:nl1d}. The rate seems to be of order $\mO(\tau^\alpha)$, which is better than what the theory predicts.}
\label{table:logarithm}
\end{table}

Table~\ref{table:logarithm} presents the results for $\alpha = 0.5$ and $\lambda = 10^{-6}$. These indicate that the convergence rate is $\mO(\tau^{\alpha})$. Similar results for other choices of $\alpha$ and $\lambda$ were obtained.

\addtocounter{exxx}{1}
\subsubsection{Example \theexxx.}

\begin{table}
  \begin{center}
    \begin{tabular}{|c|c|c|} 
      \hline
      $\tau$ & $|U(N_k) - u(N_{k-1})|$ & rate \\ 
      \hline
      5.000\texttt{e}-02 & --- & --- \\ 
      2.500\texttt{e}-02 & 3.370\texttt{e}-07 & --- \\ 
      1.250\texttt{e}-02 & 1.881\texttt{e}-07 & 0.840996 \\ 
      6.250\texttt{e}-03 & 1.033\texttt{e}-07 & 0.864944 \\ 
      3.125\texttt{e}-03 & 5.607\texttt{e}-08 & 0.881286 \\ 
      1.563\texttt{e}-03 & 3.019\texttt{e}-08 & 0.893168 \\ 
      7.813\texttt{e}-04 & 1.615\texttt{e}-08 & 0.902281 \\ 
      3.906\texttt{e}-04 & 8.599\texttt{e}-09 & 0.909574 \\ 
      1.953\texttt{e}-04 & 4.559\texttt{e}-09 & 0.915606 \\ 
      9.766\texttt{e}-05 & 2.408\texttt{e}-09 & 0.920723 \\ 
      4.883\texttt{e}-05 & 1.268\texttt{e}-09 & 0.925149 \\ 
      2.441\texttt{e}-05 & 6.660\texttt{e}-10 & 0.929039 \\ 
      1.221\texttt{e}-05 & 3.490\texttt{e}-10 & 0.932497 \\ 
      6.104\texttt{e}-06 & 1.825\texttt{e}-10 & 0.935603 \\ 
      \hline
    \end{tabular}    
  \end{center}
\caption{Convergence rate for  $\alpha = 0.5$ and $\lambda = 10^{-6}$ in Example~\theexxx~of \Cref{sub:nl1d}. The rate seems to be of order $\mO(\tau)$, which is better than what the theory predicts.}
\label{table:fancy}
\end{table}

As a final example we consider
\[
  \Phi(w) = -\lambda\sqrt{1 - (1-u)_+^2}, \qquad u_0 = 0.
\]
Notice that $D(\Phi) = [0,\infty)$ and, once again, $u_0 \in D(\Phi)\setminus D(\partial\Phi)$. Table~\ref{table:fancy} presents the results for $\alpha = 0.5$ and $\lambda = 10^{-6}$. We, again, seem to get a rate that is better than what the theory predicts.

\section*{Acknowledgement}
AJS is partially supported by NSF grant DMS-1720213.

\bibliographystyle{amsplain}
\bibliography{FracGF}
\end{document}